\documentclass[reqno,12pt,a4paper]{article}
\usepackage[hypertexnames=false]{hyperref}
\usepackage{amsmath}
\usepackage{amsthm}
\usepackage{amsfonts}
\usepackage{amssymb}
\usepackage{amscd}
\usepackage{quiver}
\usepackage{mathrsfs}
\usepackage{stmaryrd}
\usepackage{url}
\usepackage{calc}
\usepackage[all]{xy}
\SelectTips{cm}{10}
\usepackage[normalem]{ulem}
\usepackage{cleveref}

\theoremstyle{plain}
\newtheorem{theorem}{Theorem}[section]
\newtheorem{lemma}[theorem]{Lemma}
\newtheorem{corollary}[theorem]{Corollary}
\newtheorem{proposition}[theorem]{Proposition}

\theoremstyle{definition}
\newtheorem{definition}[theorem]{Definition}

\newtheorem{example}[theorem]{Example}

\theoremstyle{remark}
\newtheorem*{remark}{Remark}

\newtheorem*{notation}{Notation}

\newdir{c}{{}*!/-5pt/@^{(}}
\newcommand{\pbscale}{.25}
\newcommand{\pboffset}{.5}
\newcommand{\xycorner}[3]{\save [];#2**{}?(\pbscale)="a",[];#1**{}?(\pbscale);"a"**{}?(\pboffset);"a"**\dir{-},[];#3**{}?(\pbscale);"a"**{}?(\pboffset);"a"**\dir{-} \restore}
\newcommand{\po}[1][{1}]{\xycorner{[l]**{}?(#1)}{[ul]**{}?(#1)}{[u]**{}?(#1)}}
\newcommand{\leftbox}[2]{{}\phantom{#1} \save []+L*+<.5pc>!!<0pt,\the\fontdimen22\textfont2>!L{#1#2} \restore}
\newcommand{\rightbox}[2]{{}\phantom{#2} \save []+R*+<.5pc>!!<0pt,\the\fontdimen22\textfont2>!R{#1#2} \restore}

\entrymodifiers={+!!<0pt,\the\fontdimen22\textfont2>}

\newcommand{\mcA}{\mathcal A}
\newcommand{\mcB}{\mathcal B}
\newcommand{\mcC}{\mathcal C}
\newcommand{\mcD}{\mathcal D}
\newcommand{\mcE}{\mathcal E}
\newcommand{\mcF}{\mathcal F}
\newcommand{\mcG}{\mathcal G}
\newcommand{\mcI}{\mathcal I}
\newcommand{\mcK}{\mathcal K}
\newcommand{\mcL}{\mathcal L}
\newcommand{\mcM}{\mathcal M}
\newcommand{\mcO}{\mathcal O}
\newcommand{\mcP}{\mathcal P}
\newcommand{\mcR}{\mathcal R}
\newcommand{\mcS}{\mathcal S}
\def\A{\operatorname{\mathcal{A}}}

\def\C{\operatorname{\mathcal{C}}}
\def\D{\operatorname{\mathcal{D}}}
\def\O{\operatorname{\mathcal{O}}}

\def\Z{\operatorname{\mathbb{Z}}}
\def\N{\operatorname{\mathbb{N}}}
\def\R{\operatorname{\mathbb{R}}}

\newcommand{\bbZ}{\mathbb Z}

\newcommand{\colim}{\operatorname{colim}}
\newcommand{\incl}{\operatorname{in}}
\newcommand{\Op}{\operatorname{Op}}
\newcommand{\pr}{\operatorname{pr}}

\newcommand{\sheafify}{\operatorname{sh}}
\newcommand{\Spec}{\operatorname{Spec}}
\newcommand{\spec}{\operatorname{spec}}
\newcommand{\Pts}{\operatorname{Pts}}
\newcommand{\red}{\mathrm{r}}
\newcommand{\fix}{\mathrm{f}}
\newcommand{\mono}{\mathrm{m}}
\newcommand{\can}{\mathrm{can}}
\newcommand{\lan}{\operatorname{lan}}

\newcommand{\mcFL}{\operatorname{\mcF\mcL}}
\newcommand{\El}{\operatorname{El}}

\newcommand{\FinLoc}{\operatorname{\mathsf{FinLoc}}}
\newcommand{\DistOp}{\operatorname{\mathsf{DistOp}}}
\newcommand{\rad}{\operatorname{rad}}
\newcommand{\im}{\operatorname{im}}
\newcommand{\yoneda}{y}
\newcommand{\yonedash}{\yoneda'}
\newcommand{\Cone}{C}

\newcommand{\CAT}{\operatorname{\mathsf{CAT}}}
\newcommand{\PSH}[1]{\operatorname{\mathsf{PSH}}_{#1}}
\newcommand{\PSh}[1]{\operatorname{\mathsf{PSh}}_{#1}}
\newcommand{\SH}[2]{\operatorname{\mathsf{SH}}^{#1}_{#2}}
\newcommand{\Sh}[2]{\operatorname{\mathsf{Sh}}^{#1}_{#2}}
\newcommand{\Top}{\operatorname{\mathsf{Top}}}
\def\LRSp{\operatorname{\mathsf{LRSp}}}
\def\RSp{\operatorname{\mathsf{RSp}}}
\def\CRing{\operatorname{\mathsf{CRing}}}
\def\Ring{\operatorname{\mathsf{Ring}}}
\def\Top{\operatorname{\mathsf{Top}}}
\def\Euc{\operatorname{\mathsf{Euc}}}
\def\Set{\operatorname{\mathsf{Set}}}
\def\SET{\operatorname{\mathsf{SET}}}

\newcommand{\AP}{(\mcA, \mcP)}
\newcommand{\ATop}{\Top_\mcA}
\newcommand{\APTop}{\Top_{\mcA,\mcP}}
\newcommand{\rMArMPTop}{\Top_{\red_\mcM\mcA,\red_\mcM\mcP}}
\newcommand{\APAffSch}{\operatorname{\mathsf{AffSch}}_{\mcA,\mcP}}

\newcommand{\op}{\mathrm{op}}

\newcommand{\Ra}{\Rightarrow}
\newcommand{\la}{\leftarrow}
\newcommand{\lra}{\longrightarrow}

\newcommand{\xlra}[1]{\xrightarrow{\ #1\ }}
\newcommand{\xla}[1]{\xleftarrow{#1}}
\newcommand{\xlla}[1]{\xleftarrow{\ #1\ }}
\newdir{ >}{{}*!/-5pt/@{>}}
\newcommand{\cof}[1][]{\mathbin{\:\!\!\xymatrix@1@C=15pt{{}\ar@{ >->}[r]^{#1} & {}}}}
\newcommand{\fib}[1][]{\mathbin{\:\!\!\xymatrix@1@C=15pt{{}\ar@{->>}[r]^{#1} & {}}}}
\newcommand{\loc}{\mathbin{\:\!\!\xymatrix@1@C=15pt{{}\ar[r]^\sim & {}}}}
\newcommand{\rightarrows}{\mathbin{\:\!\!\xymatrix@1@C=15pt{{}\ar@<.5ex>[r] \ar@<-.5ex>[r] & {}}}}

\newcommand{\pbd}[1][{1}]{\xycorner{[r]**{}?(#1)}{[ur]**{}?(#1)}{[u]**{}?(#1)}}
\newcommand{\regepi}[1][]{\mathbin{\:\!\!\xymatrix@1@C=15pt{{}\ar@{->>}[r]^{#1} & {}}}}
\newcommand{\hra}[1][]{\mathbin{\:\!\!\xymatrix@1@C=15pt{{}\ar@{c->}[r]^{#1} & {}}}}

\begin{document}

\author{J.\ Jurka\footnote{The research of J.J.\ was supported by the Grant agency of the Czech republic under the grant 22-02964S.}, T.\ Perutka, L.\ Vok\v{r}\'{i}nek\footnote{The research of L.V.\ was supported by the Grant agency of the Czech republic under the grants 19-00902S and 22-02964S.}}
\title{Constructing generalized schemes using cone injectivity}
\date{\today}
\maketitle
\begin{abstract}
We provide a generalization of the construction of a spectrum of a commutative ring as a locally ringed space, applicable to cone injectivity classes in general contexts, especially in locally finitely presentable categories.
In its full generality, the spectrum functor fails to be fully faithful and we study reasonable sufficient conditions under which it is. Further, assuming the full faithfulness, we introduce a generalization of another concept from algebraic geometry -- the functor of points -- and prove equivalence of the two resulting notions of schemes.
\end{abstract}

\newpage
\tableofcontents

\newpage
\section{Introduction}

Ever since the time of Alexander Grothendieck, people have been amazed by the power of the scheme theoretic methods he has developed. Although schemes might appear somewhat esoteric at first sight, their applications are far reaching -- not only in algebraic geometry, but also for example in number theory. 

Therefore, it comes as no surprise that people have been trying to use the ideas of scheme theory in many different contexts (e.g.\ $p$-adic or $\mathbb{F}_1$ geometry), hoping for similar success. How should one do that? When constructing schemes, one first defines categories $\RSp$ and $\LRSp$ of \emph{ringed} and \emph{locally ringed spaces} and then a \emph{spectrum functor}
\[\Spec\colon \CRing^{op}\to \LRSp.\]
Having these notions, it is easy to define schemes as locally ringed spaces admitting an open cover by objects in the image of $\Spec$ (\emph{affine schemes}).

Usually, it is not that hard do define appropriate analogues of $\RSp$ and $\LRSp$; the difficult part is to construct the spectrum functor. Even in the classical setting, the construction is quite involved, entirely non-obvious, and thus it is not clear how to transfer it to different settings. However, the categorical viewpoint can help us at this point as $\Spec$ enjoys a universal property: it is a right adjoint to the functor of ``global sections'' $$\Gamma\colon \LRSp\to \CRing^{op}$$ which is easy to define in any context. So if we try to construct some analogue of $\Spec$, we can make sure that our construction is the correct one by verifying the adjointness property.

In our paper, we do not study \emph{analogues} of schemes, but rather a \emph{generalization} of schemes: our aim is to produce an abstract framework such that the examples of spectrum functors in the literature appear as special cases of a general construction. Generalized schemes were already studied; originally by Hakim \cite{Hakim}, more recently for example by Brandenburg \cite{Brandenburg}, Toën and Vaquié \cite{TV}, or Lurie \cite{DAG} (in the $\infty$-categorical setting). In our setting, the ``basic objects'' need not be commutative rings, but rather objects of an arbitrary locally finitely presentable category $\A$. Our generalized spectrum functor is then of the form $$\Spec\colon \A^{op}\to\APTop$$ where the codomain is the appropriate generalization of the category of locally ringed spaces.

The first important novelty of our approach is that we deduce the existence of $\Spec$ \textit{formally} via the Adjoint Functor Theorem. Hence, we do not need to guess what the $\Spec$ might look like but rather deduce the construction already knowing it exists. Another advantage of our approach is that one is led naturally to considering colimits of generalized ringed and locally ring spaces that are of great importance regardless of the proof. However, unlike its proof, the main theorem is not new -- it is very much just a categorical version of the logic-oriented results of \cite{Coste}, see also \cite{Aratake}. In addition, there is a more general approach of \cite{Diers} that we comment on later.

The price for the generalization is that, unlike in the classical case of commutative rings, the spectrum functor $\Spec$ needs not be fully faithful, so that the ``affine scheme'' $\Spec R$ does not give a faithful picture of the algebraic object $R$. While this question was treated in the above mentioned papers, we did not find the criteria very practical. In our take, we identify a number of obstructions to fully faithful $\Spec$ in the second part of the paper. For most of the obstructions, if these do not vanish, it is possible to replace $\mcA$ by a better-behaved subcategory and finally arrive at one for which $\Spec$ is fully faithful.

The full faithfulness of the spectrum functor is important for the final part of our paper that generalizes the so-called functor of points approach which studies a locally ringed space $X$ through maps from affine schemes: $NX = \LRSp(\Spec -, X) \in [\mcA, \Set]$. This presheaf $NX$ satisfies a natural sheaf condition; in our generalized setting, we obtain a functor 
\[(\text{generalized locally ringed spaces}) \xlra{N} (\text{sheaves on $\A^{op}$})\]
valued in a certain sheaf category. In recent developments of algebraic and arithmetic geometry, one often defines analogues of schemes sitting in this category\footnote{e.g.\ in the recent development of analytic spaces by Clausen and Scholze, although that takes places in the $\infty$-categorical setting} rather than inside generalized locally ringed spaces. The main theorem of Part III of our paper shows that the restriction of $N$ to the appropriate subcategories of schemes on both sides, is an equivalence.


\section{Main results}

For the purposes of this section, we assume $\mcA$ to be a locally finitely presentable category. All results hold under milder assumptions which we will make more precise later. 

In $\mcA$, we fix a set of cones
\[C=\{(a_{ij} \colon A_i \to B_{ij})\mid i\in I\}\]
with all $A_i, B_{ij}$ finitely presentable and all $a_{ij}$ epic. We consider the subcategory\footnote{We have chosen the letter '$\mcP$' as for 'points' -- each point of $\Spec R$ is associated with an object of $\mcP$.} $\mcP$ of $\mcA$ spanned by objects cone-injective w.r.t.\ $C$ -- the \emph{local} objects -- and maps injective w.r.t.\ the collection of all the $a_{ij}$ -- the \emph{admissible} maps. 
We then define \emph{localizations} of $R\in \mcA$ as all the morphisms $R\to K$ obtained by transfinite compositions of pushouts of coproducts of the maps $a_{ij}$. We say that a localization is \emph{finite} if it is obtained using finite coproducts in a finite number of steps, and we say that it is a \emph{local form} if its codomain lies in $\mcP$. Taking all local forms (up to isomorphism) $R\to P_\alpha$ of a given $R$, the factorization system generated by all the $a_{ij}$ produces a factorization of $R\to\prod_{\alpha} P_{\alpha}$ through an object $\red R$ which we call the \textit{reduction} of $R$. We have the full subcategory $\red \mcA\subseteq \A$ of reduced objects. We will make these definitions precise and further comment on them at the beginning of Part I.

This approach as well as the notation are motivated by classical scheme theory: if we choose $\mcA$ to be the category of commutative rings, we get the subcategory of local rings and local homomorphisms as a category $\mcP$ of local objects and admissible maps for the cone 
\[\xymatrix@C=0pc@R=1pc{
	& \Z[x,y]/(x+y-1) \ar[ld] \ar[rd] \\
	\Z[x,x^{-1},y]/(x+y-1) & & \Z[x,y,y^{-1}]/(x+y-1)
}\]
(together with a certain empty cone to avoid the trivial ring, but we will comment on this later). In this setting, finite localizations of $R$ correspond to localizations -- in sense of ring theory -- w.r.t.\ elements of $R$ and local forms correspond to localizations w.r.t.\ prime ideals of $R$. Each object is reduced (non-reduced objects occur only when $\Spec$ is not fully faithful).

In order to make things geometric, we need to define the analogues of ringed spaces and locally ringed spaces in this generalized setting. The idea is simple: we replace sheaves of rings on topological spaces by $\A$-valued sheaves.

\begin{definition}
The category $\ATop$ of \emph{$\mcA$-spaces} is the category of topological spaces $X$ equipped with an $\mcA$-valued sheaf $\mcO_X \colon \Op(X)^\op \to \mcA$. Maps are continuous maps $\varphi \colon X \to Y$ together with $\varphi^{\sharp} \colon \mcO_Y \to \varphi_*\mcO_X$ or equivalently $\varphi_\sharp \colon \varphi^* \mcO_Y \to \mcO_X$.

The non-full subcategory $\APTop \subseteq \ATop$ of \emph{$\AP$-spaces} consists of those $\mcA$-spaces whose stalks $\mcO_{X,p}$ lie in $\mcP$ (i.e.\ are local) and those maps for which the induced maps on stalks $\varphi_p \colon \mcO_{Y,\varphi(p)} \to \mcO_{X,p}$ also lie in $\mcP$ (i.e.\ are admissible).
\end{definition}


In the classical case where $\mcA$ is the category of commutative rings, $\mcP$ is the category of local rings and local homomorphisms, $\ATop=\RSp$, $\APTop=\LRSp$ are categories of ringed spaces and locally ringed spaces. Then the functor $\Spec$ is a right adjoint to the functor of global sections seen as a composition $\Gamma\circ\incl$ as below
\[\xymatrix{
	\LRSp \ar@{c->}[r]^-\incl & \RSp \ar[r]^-\Gamma & \CRing^{op} \ar@/^3ex/[ll]^-\Spec
}\]
In order to generalize this to our setting, we start by noticing that the functor $\Gamma\colon \Top_{\A}\to \A^{op}$, defined still as $\Gamma(X,\O_X) = \mcO_X(X)$, has a right adjoint $R \mapsto R^*$, the one-point space with global sections $R$. Hence, in order to find the right adjoint to $\Gamma\circ \incl$, it suffices to find the right adjoint to $\incl$. 

\[\xymatrix{
\APTop \ar@<0.5ex>[r]^-{\incl} & \ATop \ar@<0.5ex>@{-->}[l]^-{\spec} \ar@<0.5ex>[r]^-{\Gamma} & \mcA^\op \ar@<0.5ex>[l]^-{(-)^*}
}\]

We call this hypothetical functor $\spec$ and its existence is a content of our first main theorem.



\begin{theorem} \label{theorem:existence_of_spec}
The inclusion $\incl \colon \APTop \subseteq \ATop$ admits a right adjoint $\spec$. The image $\Spec R = \spec R^*$ of the one-point space $R^*$ is the $\AP$-space with:
\begin{itemize}
\item
	underlying set consisting of all isomorphism classes of local forms of $R$,
\item
	topology generated by the open sets $\Pts k$ where, for a finite localization $k \colon R \cof K$, the set $\Pts k$ consists exactly of (isomorphism classes of) local forms of $R$ which factor through $k$,
\item
	the structure sheaf obtained as the sheafification of the right Kan extension of the canonical presheaf $\mcO_R^\can(\Pts k) = \red K$.
\end{itemize}   
\end{theorem}

\begin{proof}
The proof is given in the first part of the paper, mainly throughout Sections~\ref{section:existence_Spec} and~\ref{section:concrete_description_Spec}.
\end{proof}

In Section \ref{sec: diers}, we relate our results to the work of Diers and Osmond (\cite{Osmond}, \cite{Diers}). In Section \ref{sec:examples}, we cover some examples of geometric spaces which fall into our framework.

\bigskip

We denote the composite adjunction as $\Gamma \dashv \Spec$. In Part II, we analyze this adjunction carefully with emphasis on its fixed points, leading to the in-depth study of obstructions to Spec being fully faithful. That seems not to appear in the current literature and it lets us identify the abstract properties which enables Spec to be fully faithful (see Example \ref{ex: classic spec ff} for the classical case). Also, in case Spec is not fully faithful, we can often find a restriction that is fully faithful restriction and has the same image.

We say that $R \in \mcA$ is a fixed point of the adjunction $\Gamma \dashv \Spec$ if the counit $\varepsilon \colon R \to \overline R := \Gamma \Spec R$ (written here in $\mcA$ rather than $\mcA^\op$) is an isomorphism. We will write $\fix\mcA$ for the full subcategory of $\mcA$ on these fixed points. Dually, an $\AP$-space $X$ is a fixed point for this adjunction, and we write $X \in \fix\APTop$, if the unit $\eta \colon X \to \Spec \Gamma X$ is an isomorphism. It is classical that any adjunction restricts to an equivalence
\[\fix\mcA^\op \simeq \fix\APTop\]
and we may thus study the fixed points equivalently via their spectra. We will show that
\[\mcP \subseteq \fix\mcA \subseteq \red\mcA \subseteq \mcA.\]
Thus, in order for $R \in \mcA$ to be a fixed point, it is necessary that $R$ is reduced 
and, consequently, we may safely restrict to $\red\mcA$, especially when fully faithful $\Spec$ is desired. At the same time, we may restrict to the subcategory $\mono\mcA$ of objects for which the collection of all local forms is jointly monic; again $\fix\mcA \subseteq \mono\mcA$. For our next theorem, we say that a map $f \colon R \to S$ is \emph{flat} if the cobase change $f_* \colon R/\mcA \to S/\mcA$ preserves finite limits.

\begin{theorem} \label{theorem:Gamma_Spec_idempotent}
Assume that $\mcA = \red \mcA$ and $\mcA = \mono \mcA$. If $\Spec R$ is compact (e.g.\ when each cone in $\Cone$ has finitely many components) and if all local forms of $R$ are flat then $\varepsilon \colon R \to \overline R$ induces an isomorphism on spectra and $\overline R$ is then a fixed point. If this happens for all $R \in \mcA$, then the adjunction $\Gamma \dashv \Spec$ is idempotent and thus:
\begin{itemize}
\item
	$\fix\mcA \subseteq \mcA$ is reflective (with reflection $\varepsilon$) and equals the full image of $\Gamma$,
\item
	$\fix\APTop \subseteq \APTop$ is also reflective and equals the full image of $\Spec$, denoted $\APAffSch$ (consisting of affine schemes -- see below),
\item
	There results an equivalence $\fix\mcA^\op \simeq \APAffSch$.
\end{itemize}
\end{theorem}

\begin{proof}
The claim about compactness is Theorem~\ref{theorem:compactness}. That $\overline R$ is a fixed point is proved as Theorem~\ref{thm:good_objects_one}. This implies that the adjunction $\Gamma \dashv \Spec$ is idempotent and the rest is contained in Theorem~\ref{theorem:idempotent_adjunction}.
\end{proof}

In general, this does not imply that $\varepsilon$ itself is an isomorphism, i.e.\ that $R$ is a fixed point. However, we have the following recognition result.

\begin{proposition}
Assume that $\mcA = \red \mcA$. A map $f \colon R \to S$ induces an isomorphism on spectra $f^* \colon \Spec S \xrightarrow\cong \Spec R$ if and only if its pushout along any local form $p \colon R \cof P$
\[\xymatrix{
R \ar[r]^-{f} \ar@{ >->}[d]_-p & S \ar@{ >->}[d] \\
P \ar[r]_-\cong & f_* P \po
}\]
is an isomorphism.
\end{proposition}

\begin{proof}
This is essentially just Theorem~\ref{thm:characterizing_spec_iso}.
\end{proof}

We may then say that local forms \emph{detect isomorphisms} if the condition from the proposition implies that $f$ itself is an isomorphism. This is known to hold in the classical case and we give a counterexample for (reduced) integral domains. We may conclude:

\begin{corollary}
Assume that $\mcA = \red \mcA$ and $\mcA = \mono \mcA$. If all local forms of $R$ are flat and detect isomorphisms then $R$ is a fixed point.\qed
\end{corollary}

\bigskip

Finally, in Part III, we study the relationship between $\AP$-spaces and certain sheaves on $\mcA^\op$, assuming that $\mcA = \fix\mcA$, i.e.\ that $\Spec$ is fully faithful; in the situation of Theorem~\ref{theorem:Gamma_Spec_idempotent}, this can be achieved by restricting from $\mcA$ to $\fix\mcA$. Namely, consider the functor $NX = \APTop(\Spec -, X)$, something that algebraic geometers usually call the \emph{functor of points} of $X$ (whereas category theorists prefer to call this a nerve functor). It has a partial left adjoint given by the left Kan extension $|\ | = \lan_y \Spec$:
\[\xymatrix{
& \mcA^\op \ar[ld]_-y \ar[rd]^-\Spec \\
[\mcA, \Set] \ar@{-->}@<0.5ex>[rr]^-{|\ |} & & \APTop \ar@<0.5ex>[ll]^-{N}
}\]
Our assumption on $\Spec$ being fully faithful implies that both triangles commute. For our last theorem, we say that an $\AP$-space is an \emph{affine scheme} if it lies in the image $\APAffSch$ of $\Spec$ and a \emph{scheme} is then an $\AP$-space that admits an open cover by affine schemes. To give a corresponding definition in $[\mcA, \Set]$, we introduce a sheaf condition that forces an affine scheme to be a union of affine subschemes if this happens on the level of $\AP$-spaces. Working in the resulting category of sheaves, a sheaf is an \emph{affine scheme} if it lies in the image of $y$ and a \emph{scheme} if it admits an ``open cover'' by affine schemes; we postpone the precise definition of an open cover. This is our last main theorem concerning $\AP$-spaces.

\begin{theorem}
Assume that $\Spec$ is fully faithful. When restricted to full subcategories of schemes on both sides the partial adjunction $|\ | \dashv N$ becomes an equivalence of categories.
\end{theorem}

\begin{proof}
This is proved as Corollary~\ref{corollary:equivalence_of_schemes}.
\end{proof}

This important result has already been proved in some special cases: in the context of classical schemes, this is an old result proved in \cite{Demazure} and in case of $\mathbb{F}_1$-schemes, this was proved by Vezzani in \cite{Vezzani}.

Throughout the text, we further prove two purely categorical results of independent interest. The first one is Theorem \ref{thm:universal_property_sheaves} on the universal property of categories of sheaves. While the result is folklore at least for sheaves on small categories, we prove it for sheaves on categories which are not necessarily small -- which increases the difficulty as we must prove the existence of certain left Kan extensions. The second one is Theorem \ref{theorem:factorization_from_universal_colimits} on categorical properties of the coproduct completion which appears to be new, supplementing the results of Adámek and Rosický in \cite{HowNice}.

\newpage
\part{Existence and description of Spec}

In this part, it is enough to assume that $\mcA$ be a complete and cocomplete category such that:
\begin{itemize}
\item
	For any topological space $X$, the inclusion $\Sh{}{X} \to \PSh{X}$ of $\mcA$-valued sheaves into $\mcA$-valued presheaves admits a left adjoint -- the sheafication functor $\sheafify$.
\item
	The isomorphisms of sheaves are reflected by the (total) stalk functor: $F \to G$ is an isomorphism if and only if the induced map $F_p \to G_p$ is an isomorphism, for each point $p \in X$.
\end{itemize}

Both these properties hold if $\mcA$ is locally finitely presentable; we will elaborate on this in Section~\ref{subsec:Heller_Rowe_formula}.

We will use the easily verified fact that $\sheafify$ then automatically preserves stalks: The stalk of a presheaf $F$ at $p \in X$ is the pullback $p^*F$ and we have
\[\mcA(F_p, R) \cong \mcA(F, p_*R) \cong \mcA(\sheafify F, p_*R) \cong \mcA((\sheafify F)_p, R)\]
since the pushforward $p_*$ preserves sheaves.

\section{Definitions and examples}

\begin{definition}
A \emph{cone} in $\A$ is a family of maps $(a_j\colon A\to B_j)_{j \in J}$ with a specified common domain which we sometimes call a summit of the cone; thus, empty cones (those with $J = \emptyset$) having different summits are considered different.
\end{definition}

Let $\Cone=\{(a_{ij} \colon A_i \to B_{ij})\mid i\in I\}$ be a fixed set of cones in $\mcA$ with finitely presentable domains and codomains (the index $i$ ranges over a set $I$, while, for a fixed $i \in I$, the index $j$ ranges over a set $J_i$). In addition, we require that all the components $a_{ij}$ are \emph{epimorphisms}.

\begin{definition}
We say that an object $R \in \mcA$ is \emph{local} (or that $R$ is cone injective w.r.t.\ $\Cone$) if, for each $i$, and for each map $A_i \to R$, there exists $j\in J_i$ and an extension
\[\xymatrix{
A_i \ar[r] \ar[d]_-{a_{ij}} & R \\
B_{ij} \ar@{-->}[ru]
}\]
\end{definition}


Since we assume that $A_i \to B_{ij}$ are epimorphisms, the extension is unique (the index $j$ need not be). For a set of cones $\Cone$ as above, denote by $\Cone^{\downarrow}$ the collection of all the maps $a_{ij}$ appearing in the cones from $\Cone$.

\begin{definition}
A map $f \colon R \to S$ is \emph{admissible} if it has the right lifting property w.r.t.\ $\Cone^{\downarrow}$, i.e.\ for each commutative square as below, there exists a lift making both triangles commute:
\[\xymatrix{
A_i \ar[r] \ar[d]_-{a_{ij}} & R \ar[d]^-f \\
B_{ij} \ar@{-->}[ru] \ar[r] & S
}\]
\end{definition}

We note that in this case we do not make use of the common domain $A_i$ for the maps $a_{ij}$, with varying $j$. Notationally, we will emphasize that $f$ is admissible by decorating the map as $f \colon R \loc S$.

We define $\mcP \subseteq \mcA$ to be the non-full subcategory of local objects and admissible maps.

\begin{definition}
We say that $f$ is a \emph{localization} (or that $f$ is cellular w.r.t.\ $\Cone^{\downarrow}$) if $f \colon R \to S$ is obtained as a transfinite composition of pushouts of coproducts of the $a_{ij}$.

Further, we say that $f$ is a \emph{finite localization} if the transfinite composition and all the coproducts used are finite.
\end{definition}

Notationally, we will emphasize that $f$ is a localization by decorating the map as $f \colon R \cof S$. Localizations have the left lifting property w.r.t.\ admissible maps, implying in particular that a map that is both a localization and admissible must be an isomorphism.

\begin{definition}
A \emph{local form} of $R$ is a localization $p \colon R \cof P$ with local codomain $P\in\mcP$.
\end{definition}

We also recall that $R$ is reduced if the canonical map $R \to \prod_\alpha P_\alpha$, with components all the local forms of $R$, is admissible. Apart from the following examples, this notion will not be used until Section~\ref{sec:reduction}.

\begin{example}
Our main illustrating example lives in the category of commutative rings.
Localness with respect to the cone
\[\xymatrix@C=0pc@R=1pc{
	& \Z[x,y]/(x+y-1) \ar[ld] \ar[rd] \\
	\Z[x,x^{-1},y]/(x+y-1) & & \Z[x,y,y^{-1}]/(x+y-1)
}\]
means that a commutative ring $R$ is a local ring or the trivial ring. In order to get rid of the trivial ring we add the empty cone with summit the trivial ring $\{0 = 1\}$. Localness with respect to this empty cone simply means that there is no map $\{0 = 1\}\to A$ (no extension exists for the cone is empty). Admissibility of a ring homomorphism $f$ means simply that it reflects invertibility: $f(r)$ invertible $\Ra$ $r$ invertible. (Finite) localizations are obtained by adding inverses to (a finite number of) elements. Local forms are the localizations at prime ideals, since these are exactly the local localizations. All objects are reduced, as any non-invertible element lies in some prime ideal and is thus non-invertible in the product of all localizations.
\end{example}

\begin{example}
There are two modifications whose spectra have the same underlying set as the spectrum of the previous example. The first modification is induced by the cone
\[\xymatrix@C=0pc@R=1pc{
	& \bbZ[x, y]/(xy) \ar[ld] \ar[rd] \\
	\bbZ[x, y]/(x) & & \bbZ[x, y]/(y)
}\]
and the empty cone with summit $\{0=1\}$. Its local objects are the integral domains, its admissible maps are the injective maps, and its local forms are obtained as quotients by various prime ideals.

The second modification is induced by
\[\xymatrix@C=2pc@R=1pc{
	& \bbZ[x] \ar[ld] \ar[rd] \\
	\bbZ[x,x^{-1}] & & \bbZ[x]/(x)
}\]
and the empty cone with summit $\{0=1\}$. Its local objects are  the fields, its admissible maps are the maps that are injective and reflect invertibility, and its local forms are obtained as fraction fields of the quotients by various prime ideals.

In both these modifications, reduced objects are the reduced rings in the usual sense, i.e.\ those with trivial nilradical. In particular, the yet to be proved inclusion $\fix\mcA \subseteq \red\mcA$ shows that $\Spec$ fails to be fully faithful in these modifications and it is thus necessary to restrict to reduced commutative rings. However, even this restriction fails to be fully faithful, as Example~\ref{example:spec_not_fully_faithful} shows.
\end{example}

\section{Proof of the existence of $\Spec$} \label{section:existence_Spec}

In the first part of the paper, we deal with the existence and description of the spectrum functor -- Theorem~\ref{theorem:existence_of_spec}. We start proving the existence of the adjoint $\spec$ to the inclusion of $\APTop$ into $\ATop$. This will be an application of the adjoint functor theorem, so we will prove that all colimits exist in $\ATop$ and that $\APTop$ is closed under them and finally we will verify the solution set condition.

%

\subsection{Characterization of $\AP$-spaces and $\AP$-maps}

First we prove a useful characterization of $\AP$-spaces among $\mcA$-spaces. Let $X \in \ATop$ and $f \colon A_i \to \mcO_X(U)$ be a map. For each $j\in J_i$, we define a subset $U_j^\mathrm{max} \subseteq U$ to consist of all the points $p \in U$ for which an extension through $B_{ij}$ exists in the diagram
\[\xymatrix{
A_i \ar[r]^-f \ar[d] & \mcO_X(U) \ar[r] & \mcO_{X,p} \\
B_{ij} \ar@{-->}[rru]
}\]
By the finite presentability of $A_i$ and $B_{ij}$, one may then replace the stalk $\mcO_{X,p}$ by $\mcO_X(V_p)$ for some open $V_p \ni p$ and, by uniqueness of the extensions, these are compatible with restrictions and the sheaf condition for $\mcO_X$ thus provides an extension
\[\xymatrix{
A_i \ar[r]^-f \ar[d] & \mcO_X(U) \ar[r] & \mcO_X(U_j^\mathrm{max}) \\
B_{ij} \ar@{-->}[rru]_-{g_j}
}\]
Moreover, we see that $U_j^\mathrm{max}=\bigcup_{p\in U_j^\mathrm{max}}V_p$ is open. When $X \in \APTop$ the sets $U_j^\mathrm{max}$ give a particular open cover of $U$ which we refer to as the \emph{canonical} open cover. This proves the necessity part of the following criterion in which any open cover suffices.

\begin{lemma} \label{lemma:local_criterion}
An $\mcA$-space $X$ is an $\AP$-space if and only if for each map $f \colon A_i \to \mcO_X(U)$ there exists an open cover $U = \bigcup U_j$ together with ``partial'' extensions $g_j \colon B_{ij} \to \mcO_X(U_j)$ as above.
\end{lemma}

\begin{proof}
It remains to prove sufficiency. For any extension problem
\[\xymatrix{
A_i \ar[r] \ar[d] & \mcO_{X,p} \\
B_{ij} \ar@{-->}[ru]
}\]
we may first replace the stalk by $\mcO_X(U)$, then use the property from the statement to get an open cover $U = \bigcup U_j$, choose $j$ so that $p \in U_j$ and then get an extension
\[\xymatrix{
A_i \ar[r] \ar[d] & \mcO_X(U_j) \ar[r] & \mcO_{X,p} \\
B_{ij} \ar@{-->}[ru]
}\]
to $\mcO_X(U_j)$ and thus also to $\mcO_{X,p}$, as required.
\end{proof}

Next we give a similar criterion for maps.

\begin{lemma} \label{lemma:local_criterion_map}
A map of $\mcA$-spaces $\varphi\colon X\to Y$ is admissible if and only if the map of structure sheaves $\varphi_\sharp \colon \varphi^*\mcO_Y(U) \to \mcO_X(U)$ is admissible for every $U$.
\end{lemma}

\begin{proof}
It is quite easy to verify that a map of sheaves is objectwise admissible iff it is stalkwise admissible: The forward implication follows from the fact that admissible maps are closed under filtered colimits and it works also for presheaves. For the backward implication one finds lifts locally in a neighbourhood of every point, by uniqueness and the sheaf condition for $\mcF$ they provide a global lift:
\[\xymatrix{
A_i \ar[r] \ar[d] & \mcF(U) \ar[r] \ar[d] & \mcF(U_p) \ar[r] & \mcF_p \ar[d] \\
B_{ij} \ar[r] \ar@{-->}[rrru] \ar@{-->}[rru] \ar@{.>}[ru] & \mcG(U) \ar[rr] & & \mcG_p
}\]
Now apply this to the map $\varphi_\sharp \colon \varphi^* \mcO_Y \to \mcO_X$ whose effect on stalks is
\[\varphi_p \colon \mcO_{Y, \varphi(p)} = (\varphi^*\mcO_Y)_p \to \mcO_{X, p}\]
(the stalk at $p$ is given by the inverse image along the inclusion of $p$ in $X$).
\end{proof}

\subsection{Colimits in $\ATop$ and $\APTop$}

We proceed with the proof of the main theorem. First of all, we can see that the category $\Top_{\mcA}$ of $\mcA$-spaces is cocomplete using the Grothendieck fibration $\Top_{\mcA}\to \Top$ (see Example~\ref{ex:ATop}). Secondly, we show that $\APTop$ is closed under colimits, i.e.\ when all the stalks in a diagram $X_k$, $k \in \mcK$, are local and maps between them are admissible, the same is true for the colimit $X = \colim X_k$ and the components $\lambda_k$ of the colimit cocone.
\[\xymatrix@R=0.5pc{
\rightbox{U_k \subseteq {}}{X_k} \ar[rd]^-{\lambda_k} \ar[dd]_-{\alpha_*} \\
& \leftbox{X}{{} \supseteq U} \\
\rightbox{U_l \subseteq {}}{X_l} \ar[ru]_-{\lambda_l}
}\]
Consider an open subset $U \subseteq X$ and a map $A_i \to \mcO_X (U)$. Then, for each $k$, the preimage $U_k = (\lambda_k)^{-1}(U) \subseteq X_k$ admits, by Lemma~\ref{lemma:local_criterion}, the canonical open cover $U_k = \bigcup U_{kj}^\mathrm{max}$. The $U_{kj}^\mathrm{max}$ are compatible under taking inverse images in the diagram, i.e.\ $p \in U_{kj}^\mathrm{max}$ iff $\alpha_*p \in U_{lj}^\mathrm{max}$, for in the diagram
\[\xymatrix{
A_i \ar[r] \ar[d]_-{a_{ij}} & \mcO_{X_l, \alpha_*p} \ar[d]^-\sim \\
B_{ij} \ar@{-->}[r] \ar@{-->}[ru] & \mcO_{X_k, p}
}\]
one extension exists if and only if the other exists, by virtue of the admissibility of the map between stalks. Therefore, by Example~\ref{example:Top}, the collection of open sets $U_{kj}^\mathrm{max}$ describes an open set $U_j \subseteq X$. Since also $\mcO_X(U_j) = \lim \mcO_{X_k}(U_{kj}^\mathrm{max})$ by Example~\ref{ex:ATop}, the (unique) partial extensions $B_{ij} \to \mcO_{X_k}(U_{kj}^\mathrm{max})$ yield a partial extension $B_{ij} \to \mcO_X(U_j)$ and Lemma~\ref{lemma:local_criterion} proves $X \in \APTop$.

The admissibility of the components of the colimit cocone $\lambda_k$ is similar, using Lemma~\ref{lemma:local_criterion_map}: We need to find a lift in
\[\xymatrix{
A_i \ar[r] \ar[d]_-{a_{ij}} & \leftbox{\lambda_k^*\mcO_X(U)}{{} = \colim_{V \supseteq \lambda_k(U)} \mcO_X(V)} \ar[d] \\
B_{ij} \ar[r] \ar@{-->}[ru] & \mcO_{X_k}(U)
}\]
By finite presentability of $A_i$ we get a factorization through some $\mcO_X(V)$ in the colimit. Expressing this as $\lim \mcO_{X_l}(\lambda_l^{-1}(V))$ we consider, for each $l$, the maximal open subset $W_l \subseteq \lambda_l^{-1}(V)$ where the extension
\[\xymatrix{
A_i \ar[r] \ar[d]_-{a_{ij}} & \mcO_{X_l}(\lambda_l^{-1}(V)) \ar[d] \\
B_{ij} \ar@{-->}[r] & \mcO_{X_l}(W_l)
}\]
exists; by assumption, $U \subseteq W_k$. As in the previous part, the $W_l$ describe an open subset $W \subseteq V$ with $\lambda_k(U) \subseteq W$ and so the extensions $B_{ij} \to \mcO_{X_l}(W_l)$ describe a single extension $B_{ij} \to \mcO_X(W)$ and finally a lift in the original diagram.

\subsection{Solution set condition}

This is the technical part of the proof: While colimits of $\AP$-spaces are useful outside of the existence proof, we do not find other uses for the arguments of this part. The solution set condition for the inclusion $\mathrm{in} \colon \APTop \to \ATop$ is the condition that the comma category $\mathrm{in}/Y$ contains a weakly terminal set of objects. We will be using more comma categories associated with inclusions of subcategories and we do not want to introduce names for each of the inclusions, so we will use an alternative notation:
\begin{notation}
    Let $i\colon\mathcal D\to\mcC$ be the inclusion of the subcategory $\mathcal D$ into $\mcC$. Then for any object $A$ of $\mcC$ we will denote the comma category $i/A$ by $\mathcal D/_{\mcC} A$.
\end{notation}

This should be reminiscent of the usual notation for pullback (a lax version of which the comma category is an example). 

Thus, given $Y \in \ATop$ we need to find a weakly terminal subset of $\APTop/_{\ATop}Y$. Let $(\varphi, \varphi_\sharp) \colon (X, \mcO_X) \to (Y, \mcO_Y)$ be an arbitrary map with $X \in \APTop$. Using the fibration structure on the forgetful functor $\ATop \to \Top$ as in Example \ref{ex:ATop}, we obtain the vertical-cartesian factorization of $(\varphi, \varphi_\sharp)$ as follows:
\[(\varphi, \varphi_\sharp) \colon (X, \mcO_X) \xlra{(1,\varphi_{\sharp})} (X, \varphi^*\mcO_Y) \xlra{(\varphi,1)} (Y, \mcO_Y)\]
This essentially reduces the possible structure sheaves $\mcO_X$ to those of the form $\varphi^*\mcO_Y$ except the latter does not have stalks in $\mcP$. Thus, for each open $U \subseteq X$, apply the usual small object argument to $\varphi_\sharp \colon \varphi^* \mcO_Y(U) \to \mcO_X(U)$, thus producing by functoriality a factorization
\[\mcO_X \xlla{\chi_\sharp} F \xlla{\psi_\sharp} \varphi^* \mcO_Y\]
in $\PSh{X}$ and upon replacing $F$ by its sheafification also in $\Sh{}{X}$. Lemma~\ref{lemma:local_criterion_map} shows that the presheaf version of $F \to \mcO_X$ is also stalkwise admissible. Since the stalks of $\mcO_X$ are local, so are those of $F$.


This means that $\varphi$ factors through the $(\mcA,\mcP)$-space $(X, F)$,
\[(\varphi, \varphi_\sharp) \colon (X, \mcO_X) \xlra{(1,\chi_\sharp)} (X, F) \xlra{(\varphi,\psi_\sharp)} (Y, \mcO_Y)\]
and in effect we reduced all the possible $\mcO_X$ to those obtained from $\varphi^*\mcO_Y$ by localization followed by sheafification; for a fixed $\varphi \colon X \to Y$ there is a set of such, by the proof of Theorem~\ref{thm:components_multireflection_small}. However, the underlying spaces $X$ and maps $\varphi$ still form a proper class. Thus, form an equivalence relation on $X$: $p \sim p'$ if their images in $Y$ agree, i.e.\ $\varphi(p) = q = \varphi(p')$ and if the two induced maps
\[\xymatrix@=0.5pc{
F_p \ar@{-->}[dd]_-\cong \\
& & \mcO_{Y,q} \ar[llu] \ar[lld] \\
F_{p'}
}\]
are isomorphic -- note that, as localizations (we obtained $F$ via small object argument), they are epis so the isomorphism will be unique. Now we need to show that $X/{\sim}$ inherits a structure sheaf from $X$ with local stalks and that $\varphi$ factors through it:
\[\varphi \colon (X, \mcO_X) \xlra{(1,\varphi)} (X, F) \xlra{\pr} (X/{\sim}, F/{\sim}) \xlra{\overline \varphi} (Y, \mcO_Y).\]
This will finish the proof since the $\AP$-spaces of the form $(X/{\sim}, F/{\sim})$ form, up to isomorphism, a set: localizations of each $\mcO_{Y,q}$ form a set by Theorem~\ref{thm:components_multireflection_small}, thus we get a set of all possible underlying sets $X/{\sim}$, thus a set of possible topologies and, finally, a set of possible structure sheaves as above.

To finish the proof, we remark that $X/{\sim}$ is a colimit; namely, it is a coequalizer of
\[\xymatrix{
\sim \ar@/^1ex/[r]^-{\pr_l} \ar@/_1ex/[r]_-{\pr_r} & X
}\]
with the equivalence relation ${\sim} \subseteq X \times X$ topologized discretely and given the structure sheaf with stalks at $(p, p')$ the local $F_p \cong F_{p'}$. Since this diagram lies in $\APTop$, so does its colimit.

\section{Factorization system generated by cones}

To conclude the existence part of Theorem~\ref{theorem:existence_of_spec}, we have to prove Theorem~\ref{thm:components_multireflection_small} about the number of local forms of a given object $R \in \mcA$. At the same time, we will prepare ground for the description part of Theorem~\ref{theorem:existence_of_spec}.

There is a factorization system associated with the collection $C^\downarrow$ of maps $a_{ij}$. A factorization is obtained by the small object argument and uniqueness is guaranteed by the generating maps $a_{ij}$ being epic; this implies that the factorization is functorial. Also note that by uniqueness every cofibration is in fact cellular (apply the small object argument to the cofibration and apply uniqueness).

We recall the notation $\cof$ for localizations and $\loc$ for admissible maps. We will make a frequent use of the following simple lemma:

\begin{lemma} \label{lemma:factorization_pushout}
If $f$ is epic then the any commutative square as below is cocartesian:
\[\xymatrix{
R \ar[r]^-f \ar[d]_-h & S \ar[d]^-g \\
T \ar@{=}[r] & T
}\]
\end{lemma}

\begin{proof}
Since $h = g f$, the square can be decomposed as
\[\xymatrix@R=0.75pc{
R \ar[r]^-f \ar[d]_-f & S \ar@{=}[d] \\
S \ar@{=}[r] \ar[d]_-g & S \ar[d]^-g \po \\
T \ar@{=}[r] & T \po
}\]
with the upper square cocartesian since $f$ is assumed epic.
\end{proof}

\begin{lemma} \label{lemma:cancellation_properties}
In a commutative triangle
\[\xymatrix@=1pc{
R \ar[rr]^-h \ar[rd]_-f & & T \\
& S \ar[ru]_-g
}\]
\begin{itemize}
\item
	$h$ admissible $\Ra$ $f$ admissible;
\item
	$h$ localization and $f$ epi $\Ra$ $g$ localization;
\item
	$h$ finite localization and $f$ epi $\Ra$ $g$ finite localization.
\end{itemize}
\end{lemma}

\begin{proof}
The first point is easy to verify (using that the $a_{ij}$ are epic), as is the second, we proceed with the third point. The square
\[\xymatrix{
R \ar[r]^-f \ar[d]_-h & S \ar[d]^-g \\
T \ar@{=}[r] & T
}\]
is a pushout by Lemma~\ref{lemma:factorization_pushout}, i.e.\ $g$ is a pushout of $h$, so a finite localization.
\end{proof}

\begin{lemma} \label{lemma:initial_object_into_B}
The components of the comma category $R/_\mcA\mcP$ are in bijection with isomorphism classes of local forms of $R$.

Every local form $R \cof P$ is an initial object of the respective component.
\end{lemma}

\begin{proof}
By factoring a map $f \colon R \to Q$ with $Q\in \mcP$ into a localization $p$ followed by an admissible map (thus $p$ is a local form), we obtain a map $p \to f$ in $R/_\mcA\mcP$; for a fixed $f$, the object $p$ is unique up to isomorphism and, upon fixing it, the map $p \to f$ is unique. For a map $f \to f'$ in $R/_\mcA\mcP$ the factorization gives a diagram on the left, interpreted in $R/_\mcA\mcP$ on the right:
\[\xymatrix@=.5pc{
& & & & P \ar[rr]^-\sim \ar[dd] & & Q \ar[dd]^-\sim & & & & p \ar[rr] \ar[dd] & & f \ar[dd] \\
\mcA \colon & & R \ar@{ >->}[rru]^-p \ar@{ >->}[rrd]_-{p'} & & & & & & & & & & & & {:}R/_\mcA\mcP \\
& & & & P' \ar[rr]^-\sim & & Q' & & & & p' \ar[rr] & & f'
}\]
By the previous lemma, the map $P \to P'$ is both a localization and admissible, thus an isomorphism and hence $p\cong p'$. This implies that the object $p \in R/_\mcA\mcP$ is shared by the whole component and as such is initial.
\end{proof}

The previous lemma essentially says that the inclusion $\mcP \to \mcA$ admits a multi-adjoint (see \cite{Osmond}), i.e.\ that $\mcP$ is a multi-reflective subcategory, except for the size issues that we address now.

\begin{theorem}\label{thm:components_multireflection_small}
The class of components $\pi_0(R/_\mcA\mcP)$ is small, i.e.\ a set. In other words, the collection of local forms of $R$ up to isomorphism is a set.
\end{theorem}

\begin{proof}
We will present a more refined argument later. For the time being, start with a local form $p \colon R \cof P$. By finite presentability, the small object argument gives a factorization of $p$ into a countable relative cell complex followed by an admissible map. The uniqueness of the factorization gives that $p$ is, in fact, itself a countable relative cell complex. In each step, there is no need to glue a cell of one shape $a_{ij} \colon A_i \to B_{ij}$ along one attaching map $f \colon A_i \to R$ multiple times, since the $a_{ij}$'s are epic. Thus, each step allows for a set of alternatives (given by sets of triples $(i, j, f)$ used in that step) and there is a countable number of steps.
\end{proof}

In fact, we proved that the collection of localizations of $R$, up to isomorphism, is small.

\section{Proof of the concrete description} \label{section:concrete_description_Spec}

We continue with the proof of Theorem~\ref{theorem:existence_of_spec}, namely its descriptive part.

\subsection{Points and stalks}

Consider $\mcP^* \subseteq \APTop$ the full subcategory consisting of one-point spaces (they are exactly the spaces $P^*$ for $P \in \mcP$). We will now describe the points and the stalks of $(X, \mcO_X) \in \APTop$ categorically, in a way similar to Lemma~\ref{lemma:initial_object_into_B}. Since we will need to distinguish between an $\AP$-space and its underlying topological space, we will use $(X, \mcO_X)$ for the first and $X$ for the second.

\begin{lemma} \label{lemma:terminal_object_from_Bstar}
The components of the comma category $\mcP^*/_{\APTop}(X,\mcO_X)$ are in bijection with points of $X$.

Every point $p \colon * \to X$ induces a map $\overline p \colon (\mcO_{X,p})^* \to (X,\mcO_X)$ (a cartesian lift of $p$) that is a terminal object of the respective component.
\end{lemma}

\begin{proof}
Similarly to Example~\ref{ex:ATop}, $\APTop \to \Top$ is a fibrational Grothendieck construction for the diagram $\Top^\op \to \CAT$, sending a space $X$ to the opposite of the category of $\mcA$-valued sheaves on $X$ with stalks in $\mcP$ (the inverse image preserves stalks). Thus, the cartesian lift of $p \colon * \to X$ is the map $\overline p \colon (p^*\mcO_X)^* \to (X, \mcO_X)$ whose sheaf component is the identity; clearly $p^*\mcO_X = \mcO_{X,p}$. Thus, the (vertical, cartesian) factorization of a map $\varphi \colon P^* \to (X, \mcO_X)$,
\[\varphi \colon P^* \to (\mcO_{X,p})^* \to (X, \mcO_X),\]
gives a unique map from $\varphi$ to $\overline p$ in the comma category $\mcP^*/_{\APTop}(X,\mcO_X)$.
\end{proof}


The adjunction $\Gamma \dashv \Spec$, when restricted to $\mcP^* \subseteq \APTop$ yields
\[\APTop(P^*, \Spec R) \cong \mcA(R, \Gamma P^*) = \mcA(R, P),\]
naturally in $P$ and $R$, thus giving
\[\mcP^*/_{\APTop} \Spec R \cong (R/_\mcA\mcP)^\op,\]
still natural in $R$. Lemmata~\ref{lemma:initial_object_into_B} and~\ref{lemma:terminal_object_from_Bstar} give the following conclusions:
\begin{itemize}
\item
	(components) The points $p \in \Spec R$ are identified with isomorphism classes of local forms $p \colon R \cof P$.
\item
	(terminal objects) The stalk at $p \in \Spec R$ is identified with $P$.
\end{itemize}

The naturality in $R$ gives an interpretation of the map induced by $f \colon R \to S$ on the spectra: Let $q \in \pi_0(S /_\mcA \mcP)^\op$ be a point of $\Spec S$, represented by a local form $q \colon S \cof Q$. Its image under $f^* \colon \pi_0(S /_\mcA \mcP)^\op \to \pi_0(R /_\mcA \mcP)^\op$ is represented by a local form $p \colon R \cof P$ is obtained by factoring $qf$ into a localization $p$ followed by an admissible map $g$, as in
\[\xymatrix{
R \ar@{ >->}[r]^-p \ar[d]_-f & P \ar[d]_-\sim^-g & & \rightbox{p \in {}}{\Spec R} & P^* \ar[l]_-{p} \\
S \ar@{ >->}[r]_-q & Q & & \rightbox{q \in {}}{\Spec S} \ar[u]_-{f^*} & Q^* \ar[l]^-{q} \ar[u]_-{g^*}
}\]
Thus $f^*(q) = p$ and the square on the right commutes since it corresponds to $qf=gp$ in $\mcA(R, Q)$ under the above correspondence. Since $p$ and $q$ are cartesian and $g^*$ is vertical, the map on stalks $(f^*)_q \colon P \to Q$ is just $g$.

%

\subsection{Distinguished open sets}

We will now exhibit certain open subsets of $\Spec R$. It will be convenient to denote the structure sheaf of $\Spec R$ simply as $\O_R := \mcO_{\Spec R}$. The counit of the adjunction gives a canonical map $\varepsilon_R \colon R \to \Gamma \Spec R$ (for commutative rings, this happens to be an isomorphism, but generally, this is not the case). Thus, $R$ admits a canonical map into all objects in the structure sheaf $\mcO_R(U)$.

Let $k \colon R \cof K$ be a finite localization. In our identification of points with local forms, $k^* \colon \Spec K \to \Spec R$ is given by composing a local form $K \cof P$ with the epic $k$ (resulting in a local form of $R$) and as such is injective on points. We will now show that its image is an open subset: to match the classical terminology, we will call these open sets \emph{distinguished} and will denote them $\Pts k$.

Starting from the other side, a local form $p \colon R \cof P$ lies in $\Pts k$ iff $p$ factors through $k$; if this is the case, it factors through $\O_R(U)$, for some $U \ni p$:
\[\xymatrix{
\rightbox{p \colon {}}{R} \ar@{ >->}[r]^-k \ar[rrd] & K \ar[r] \ar@{-->}[rd] & \leftbox{P}{{} = \colim_{U \ni p} \O_R(U)} \\
& & \O_R(U) \ar[u] \ar[r] & P'
}\]
(by finite presentability of $K \in R/\mcA$). Consequently, for any $p' \in U$, the corresponding localization $p' \colon R \to P'$ factors through $K$ as well and thus $U \subseteq \Pts k$.

We will now show that, in fact, $\Spec K \to \Spec R$ is an \emph{open embedding}: that is, a map of $\A$-spaces $(i, i^\sharp) \colon (X,\O_X) \to (Y,\O_Y)$ which is a homeomorphism $i$ onto an open subset on the topological level and on the level of sheaves, $i_\sharp \colon i^*\O_Y\to\O_X$ is an isomorphism. Since the image $\Pts k$ of $k^*$ is open, we may consider the restriction of $\Spec R$ to this image and thus obtain a map $\Spec K \to \Spec R|_{\Pts k}$. We will now construct a map in $\APTop$ in the opposite direction by showing that a factorisation
\[\xymatrix{
& \Spec K \ar[d] \\
\Spec R|_{\Pts k} \ar@{c->}[r] \ar@{-->}[ru] & \Spec R
}\]
exists by translating it to an equivalent extension problem (where the bottom middle term is identified with the global sections of $\Spec R|_{\Pts k}$):
\[\xymatrix{
& & K \ar@{-->}[ld] \ar@{.>}[lld] \\
\O_R(U) & \O_R(\Pts k) \ar[l] & R \ar[u] \ar[l]
}\]
By the above arguments, such an extension exists locally in some neighbourhood $U$ of any given $p \in \Spec K$ -- the dotted arrow in the diagram. Since these local extensions are unique and $\O_R$ is a sheaf, they glue to give an extension to $\O_R(\Pts k)$ as required. It follows that $k^* \colon \Spec K \to \Spec R$ is an open embedding and since it is also a stalkwise iso, it gives an isomorphism $\Spec K \cong \Spec R|_{\Pts k}$. Using this identification, the counit $K \to \Gamma \Spec K$ then becomes $K \to \mcO_R(\Pts k)$.

\subsection{Topology and structure sheaf}

Finally, we will show that the distinguished opens $\Pts k$ form a basis of the topology of $\Spec R$ and give a description of the structure sheaf $\O_R$, slightly different from that claimed in the main statement; the relationship of the two approaches is addressed in Section~\ref{sec:reduction}.

We will prove the claims by constructing from $\Spec R$ a minimalistic version of it that is still an $\AP$-space with the same stalks; the universal property will then ensure that it is in fact isomorphic to $\Spec R$. Concretely, let $\Spec' R$ be $\Spec R$ equipped with a topology generated by the basis $\Pts k$, for all finite localizations $R \to K$ -- these are clearly closed under finite intersections (these correspond to pushouts of localizations). The identity of the underlying sets is a continuous map $\varphi \colon \Spec R \to \Spec' R$ that we will now promote to a map of $\AP$-spaces and then show to be an isomorphism. In order to do so, we need to define a structure sheaf $\mcO'_R$ on $\Spec' R$ and provide it with a map $\mcO'_R \to \varphi_* \mcO_R$ where the codomain is easily seen to be just the restriction of $\mcO_R$ to the opens of $\Spec' R$. We know that the counit is a map $K \to \mcO_K(\Pts k)=\mcO_R(\Pts k)$ and it is thus tempting to set $\mcO'_R(\Pts k) = K$. This prescription is only defined on the distinguished opens and does not satisfy the sheaf condition; most importantly, it is not well defined since $\Pts k$ does not determine $K$.

We can correct the above deficiency by considering, for any distinguished open $U$, the diagram $\mcFL_U$ of all finite localizations $k$ with $\Pts k = U$ and defining $R_U = \colim_{k \in \mcFL_U} K$, i.e.\ the union of all finite localizations with the same associated point set $U$. This is easily seen to be a localization $r_U \colon R \cof R_U$, possibly infinite. At the same time, for any $k \in \mcFL_U$, the corresponding component of the colimit cocone is a localization $K \to R_U$ that induces an isomorphism
\[\Spec K \cong \lim_{k \in \mcFL_U} \Spec K \cong \Spec R_U,\]
since the diagram of spectra is simply connected (in fact, cofiltered) and constant. By our identification of the spectra of localizations with subsets of $\Spec R$, this means in particular that $\Pts r_U = U$.

\begin{lemma}
If $V \subseteq U$ then a factorization $R_U \to R_V$ exists.
\end{lemma}

\begin{proof}
Let $L \in \mcFL_V$ and define a functor $\mcFL_U \to \mcFL_V$ by a pushout along $R \to L$. This is well-defined since $\Spec$ takes pushouts in $A$ to pullbacks in $\APTop$. The bottom map in the following diagram 
\[\xymatrix{
R \ar[r] \ar[d] & L \ar[d] \\
K \ar[r] & L+_RK \po
}\]
then gives a natural transformation
\[\xymatrix@=.5pc{
\mcFL_U \ar[dd] \ar[rrd] \\
{} \POS[];[rr]**{}?(.33)*{\Downarrow} & & \mcA \\
\mcFL_V \ar[rru]
}\]
with the diagonal functors the canonical inclusions and thus induces the required map $R_U \to R_V$ by taking colimits.
\end{proof}

We thus have a partial presheaf defined on the distinguished opens of $\Spec' R$, given by $U \mapsto R_U$ and we extend it to a presheaf on $\Spec' R$ using a right Kan extension\footnote{With respect to the arguments in the next paragraph, a more reasonable choice seems to be that of the left Kan extension. However, the result would almost never be equal to the structure sheaf $\mcO_R$, while Proposition~\ref{prop:Spec_fully_faithful_ORcan_sheaf} shows that the right Kan extension works better in this respect. At the same time, the sheafification of both Kan extensions is the same (they agree on distinguished opens).} along the inclusion of distinguished opens into all opens; since the inclusion is full, the result is indeed an extension which we call the \emph{canonical presheaf} and denote $\O_R^{can}$. Finally, we define $\O_R' = \sheafify \O_R^\textrm{can}$ as the sheafification.

The counits $K \to \mcO_R(\Pts k)$ induce a natural map $R_U \to \mcO_R(U)$ for $U$ distinguished open and thus, for any open $V \subseteq \Spec' R$, also
\[\lim_{\text{d.o. } U \subseteq V} R_U \to \lim_{\text{d.o. } U \subseteq V} \mcO_R(U)\]
where the left hand side is the usual formula for the right Kan extension $\mcO_R^\textrm{can}(V)$ and the right hand side is $\mcO_R(V)$ by the sheaf property of $\mcO_R$, since the collection of all distinguished open subsets of $V$ forms a particular open cover of $V$. Finally, the resulting map $\mcO_R^\textrm{can} \to \varphi_* \mcO_R$ induces a map from the sheafification $\mcO'_R \to \varphi_* \mcO_R$ and this finishes the construction of $\varphi \colon \Spec R \to \Spec R'$. We need to verify that this is a map of $\AP$-spaces. Both the right Kan extension and the sheafification preserve stalks (the first does since a stalk can be computed as a colimit over any neighbourhood basis, e.g.\ that formed by all distinguished open neighbourhoods, the second by the general property of the sheafification) and, thus, the stalk of $\mcO_R'$ at $p \colon R \to P$ is
\[\O'_{R,p} = \colim_{U \ni p} R_U \cong \colim_{\Pts k \ni p} K = \colim_{p \colon R \to K \to P} K \cong P\]
where the last isomorphism comes from the general procedure of rewriting a relative cell complex as a directed colimit of finite cell complexes, see~\cite{fsoa}. This implies easily that $\varphi \colon \Spec R \to \Spec' R$ is a stalkwise isomorphism and as such is, in particular, a map of $\AP$-spaces. Since $\Spec' R$ admits a canonical map into $R^*$, the universal property of $\Spec R$ gives a filler $\psi$ in the diagram:
\[\xymatrix{
\Spec R \ar[d]_-\varphi \ar[rd] \\
\Spec' R \ar[r] \ar@{-->}[d]_-\psi & R^* \\
\Spec R \ar[ru]
}\]
As $\varphi$ and $\psi$ are the identity maps on the underlying point sets, this shows that the topologies coincide; since $\varphi$ was also shown to be a stalkwise isomorphism, it is then an isomorphism of $\AP$-spaces.

\section{Connection to the spectrum of Diers} \label{sec: diers}


We relate our spectral construction to a Diers spectrum as described in \cite{Diers} and \cite{Osmond}. A \emph{Diers context} is a right multi-adjoint $U\colon \mcP\to \mcA$ with $\mcA$ locally finitely presentable satisfying certain technical condition: any local form $p\colon R\to P$ is a filtered colimit of all maps $k\colon R\to K$ factorizing $p$ with $K$ finitely presentable in $R/\mcA$ such that $k$ is left orthogonal to the maps in the image of $U$. We then define the category $\Top_U$ of $U$-spaces as the following pseudo-pullback in $\CAT$ (where the map $\ATop\to \prod \mcA$ is a joint stalk functor):

\[\begin{tikzcd}
	{\mathsf{Top}_U} & {\mathsf{Top}_{\mathcal A}} \\
	{\prod \mathcal P} & {\prod \mathcal A}
	\arrow["{\prod U}", from=2-1, to=2-2]
	\arrow["{\text{stalk}}", from=1-2, to=2-2]
	\arrow[from=1-1, to=2-1]
	\arrow[from=1-1, to=1-2]
\end{tikzcd}\]

In particular, if $\mcP$ is the subcategory of local objects and admissible maps with respect to a set of cones satisfying the prerequisites of Theorem \ref{theorem:existence_of_spec}, the inclusion $U\colon \mcP\to\A$ is a Diers context (it is a multi-reflection by Lemma~\ref{lemma:initial_object_into_B} and the technical condition is satisfied since any local form is a filtered colimit of finite localizations factoring through it). The category $\Top_U$ is then precisely $\APTop$. We have the following theorem of Osmond, building heavily on the work of Diers in \cite{Diers}:

\begin{theorem}
    (\cite{Osmond}, Thm. 2.14) The functor $\Top_U\to \ATop$ admits a right adjoint $\spec$. 
\end{theorem}

This shows that if $\mcA$ is locally finitely presentable, our Theorem \ref{theorem:existence_of_spec} is a special case of the theorem above. What differs are the proof methods: in both \cite{Diers} and \cite{Osmond} the spectrum is first described explicitly and only then, this description is used to prove that $\spec$ is an adjoint.

\section{Further examples} \label{sec:examples}

%
%
We are now going to identify spectra and schemes for some particular choices of $\mcA$ and $\mcP$.

\subsection{Ordinary (relative) schemes} \label{subsec: classical_schemes}

If we start with $\mcA$ the category of commutative $R$-algebras for any commutative ring $R$ and $\mcP$ the subcategory of local $R$-algebras (i.e.\ the underlying ring is local) and local homomorphisms, we can recognize $\mcP$ as a subcategory of local objects and admissible maps with respect to the cone
\[\xymatrix@C=0pc@R=1pc{
	& R[x,y]/(x+y-1) \ar[ld] \ar[rd] \\
	R[x,x^{-1},y]/(x+y-1) & & R[x,y,y^{-1}]/(x+y-1)
}\]
together with the empty cone with summit the trivial ring $\{0=1\}$. $\A$ and $\mcP$ satisfy conditions of the main theorem, hence we can construct the spectrum functor, which is just the relative $\Spec$ used in algebraic geometry. Schemes in this context are thus just $R$-schemes in the usual sense. For an $R$-algebra $A$, points of $\Spec A$ correspond to prime $R$-ideals of $A$.

\subsection{Deitmar’s $\mathbb{F}_1$-schemes} \label{subsec: Deitmar}


$\mathbb{F}_1$-geometry is a part of algebraic geometry which tries to define a geometry behaving like ``schemes over field with one element'' in order to prove the Riemann hypothesis in a similar way the Weil conjectures were proved. So far, most of the results are just conjectural as the optimal setting has not yet been found. For a nice exposition of different proposals for $\mathbb{F}_1$-schemes, see \cite{Funland}. 

Here, we are going to take a look at $\mathbb{F}_1$-schemes in the sense of Deitmar and Kato (\cite{Deitmar}, \cite{Funland}). We start with the category $\A=\mathbf{CMon}$ of commutative monoids and their homomorphisms and we consider the cone injectivity with respect to the following cone with \emph{additive} monoids of natural numbers and integers (where the map on the left is the identity and the map on the right is the inclusion): 
\[\xymatrix@C=2pc@R=1pc{
	& \N \ar[ld]_-1 \ar[rd]^-i \\
	\N & & \Z
}\]
Therefore, \emph{any} monoid is local (obviously, any map $\mathbb{N}\to M$ factors through the identity on $\N$). Admissible maps are homomorphisms reflecting invertible elements; this is precisely what the injectivity w.r.t.\ $i$ says. So, the subcategory $\mcP$ of local objects and admissible maps is a wide subcategory given by all monoids and maps that reflect invertibility.

Schemes defined in $(\A,\mcP)$-spaces in this context are $\mathbb{F}_1$-schemes in sense of Deitmar \cite{Deitmar}. For a commutative monoid $M$, points of $\Spec M$ look fairly similar to a spectrum of a ring in the previous example. Indeed, if we have a map $\N\to M$ sending $1$ to $a\in M$, pushout of $i$ corresponds adding inverse of $a$ to $M$: 
\[\xymatrix{
\N \ar[r]^-{1 \mapsto a} \ar[d]_-i & M \ar[d] \\
\Z \ar[r] & M[a^{-1}]
}\]
Recall that an ideal of a monoid $M$ is a subset $I\subseteq M$ such that $I\cdot M\subseteq I$ and it is prime if $xy\in I$ implies $x\in I$ or $y\in I$ for any $x,y\in M$. Then the set underlying $\Spec M$ is precisely the set of prime ideals of $M$.

\subsection{$C^{\infty}$-schemes} \label{subsec: Coo}

In this section, we are going to define a framework in which smooth manifolds arise as affine schemes for certain algebraic objects. This was developed in \cite{MoerdijkReyes} and further studied for example in \cite{Joyce}.

In this context, $\A=C^{\infty}\Ring$ is the category of $C^\infty$-rings, i.e.\ the category of product preserving functors $\Euc\to \Set$ where $\Euc$ is the category with $\R^n$ as objects (for $n\ge0$) and smooth functions $\R^m\to \R^n$ as morphisms. In other words, $\A$ is the category of $\Set$-valued models for the Lawvere theory $\Euc$ and as such it is locally finitely presentable. 

For any $C^\infty$-ring $\mathfrak C\colon \Euc\to\Set,$ the ``underlying set'' $\mathfrak C(\R)$ has a structure of a commutative $\R$-algebra\footnote{Note that addition and multiplication of real numbers are smooth maps $\R^2\to \R$ and as such they induce addition and multiplication on $\mathfrak C(\R)$.}. The assignment $\mathfrak{C}\mapsto \mathfrak C(\R)$ then gives a functor from $C^\infty$-rings to commutative $\R$-algebras.


Denote by $\mcP$ the subcategory of those $C^\infty$-rings $\mathfrak{C}$ such that the underlying $\R$-algebra $\mathfrak{C}(\R)$ is local (in the classical sense) and similarly for admissible maps. We claim that the inclusion $\mcP\to \A$ is a right multiadjoint. Since $C^{\infty}\Ring$ is a category of models for a Lawvere theory, we can build a free $C^\infty$-ring $C^{\infty}[x,y] = C^{\infty}(\R^2)$ on two generators $x,y$ as well as a localization $\mathfrak C[a^{-1}]$ inverting an element $a\in\mathfrak C(\R)$. It is known that we can also form a quotient $\mathfrak{C}/I$ for any $R$-ideal $I$ of underlying $\R$-algebra $\mathfrak{C}(\R)$ (cf. \cite{MoerdijkReyes}). Hence we are able to recognize $\mcP$ as the subcategory of local objects and admissible maps with respect to the following cone
\[\xymatrix@C=-1pc@R=1pc{
	& C^\infty[x,y]/(x+y-1) \ar[ld] \ar[rd] \\
	(C^\infty[x,y]/(x+y-1))[x^{-1}] & & (C^\infty[x,y]/(x+y-1))[y^{-1}]
}\]
together with the empty cone with summit the trivial $C^{\infty}$-ring. It is not hard to see that both maps appearing in the cone above are epic and all objects are finitely presentable. Statement then follows from Theorem \ref{theorem:existence_of_spec}. 

The corresponding spectrum functor was explicitly constructed in \cite{MoerdijkReyes}; points of $\Spec \mathfrak{C}$ are so-called $C^{\infty}$-radical prime ideals of $\mathfrak{C}.$

\newpage
\part{Obstructions to fully faithful Spec}

\section{Overview of the obstructions}

In the second part of the paper, we want to study conditions, under which $\Spec$ is fully faithful. This is the case in classical algebraic geometry and it turns out to be very important also in the generalized context. It is well known that this is equivalent to the counit $\varepsilon$ of $\Gamma \dashv \Spec$ being an isomorphism. We define $\fix\mcA \subseteq \mcA$ to be the full subcategory of the fixed points of the adjunction, i.e.\ those objects $R \in \mcA$ at which the counit $\varepsilon_R \colon R \to \overline R$ (interpreted in $\mcA$) is an isomorphism, where we abbreviate $\overline R = \Gamma \Spec R$.

There is a slightly weaker condition of $\varepsilon_R$ being a \emph{geometric isomorphism}, i.e.\ inducing an isomorphism on spectra. In such a situation, $\fix \mcA$ is exactly the full image of $\Gamma$, it is a reflective subcategory and $\Spec|_{\fix\mcA}$ is fully faithful. This is the content of a general well known theorem that we summarize first.

\subsection{Idempotent adjunctions}

A monad $T$ on $\mcA$ is said to be \emph{idempotent} if its multiplication $\mu$ is an isomorphism or, equivalently, if its unit $\eta$ satisfies $T \eta = \eta T$. In this case, the Eilenberg-Moore category $\mcA^T$ of algberas happens to be a reflective subcategory of all objects at which $\eta$ is an isomorphism, i.e.\ the fixed points of $T$, and the reflection is given by $T$. There is an obvious dual notion of an idempotent comonad. An adjunction $L \dashv R$ is said to be \emph{idempotent} if the associated monad $RL$ is idempotent or, equivalently, if the associated comonad $LR$ is idempotent. In this case the fixed points of $RL$ are the full image of $R$ and the fixed points of $LR$ are the full image of $L$ and the adjunction restricts to an equivalence between these, very much like for a Galois connection. Since the multiplication for $RL$ is given by $\mu = R \varepsilon L$ we thus obtain the following theorem.

\begin{theorem} \label{theorem:idempotent_adjunction}
Given an adjunction $L \dashv R$, assume that the transformation $R \varepsilon L$ is an isomorphism. Then the adjunction restricts to an equivalence between the full image of $R$ and the full image of $L$. \qed
\end{theorem}


\subsection{The consecutive obstructions}

In our situation, the condition from the previous theorem reads specifically that $\Spec \varepsilon \Gamma$ should be an isomorphism. In such a situation, we will ultimately replace $\mcA$ by $\fix\mcA$, so it is not a big deal to restrict from $\mcA$ to an arbitrary reflective subcategory $\mcB$ satisfying $\fix\mcA \subseteq \mcB \subseteq \mcA$. This will enable us to prove useful properties of $\Spec|_\mcB$ and finally give conditions under which the counit becomes a geometric isomorphism on $\mcB$. We will introduce two such reflective subcategories $\red\mcA$ and $\mono\mcA$: The subcategory $\red\mcA$ of reduced objects enables for a tighter (bijective, in fact) correspondence between distinguished open sets and finite localizations. The subcategory $\mono\mcA$ of mono-reduced objects allows for more efficient computation with the sheafification (e.g.\ all the canonical presheaves are monopresheaves). Both are introduced in Section~\ref{sec:reduction}, together with a precise description of geometric isomorphisms that allows to determine when these are actually isomorphisms. Assuming that $\mcA = \red \mcA$ and $\mcA = \mono \mcA$, Section~\ref{sec:flat_local_forms} gives conditions for $\varepsilon$ to be a geometric isomorphism.
\[\xymatrix{
& & & \red \mcA \ar@{c->}[rd] \\
\mcP \ar@{c->}[r] & \fix \mcA \ar@{c->}[r] & \im \Gamma \ar@{c->}[ru] \ar@{c->}[rd] & & \mcA \\
& & & \mono \mcA \ar@{c->}[ru]
}\]
In the classical example of local rings where $\Spec$ is fully faithful, all these inclusions except the one on the left happens to be identities.

\bigskip

We now summarize the obstructions to full faithfulness of $\Spec$:
\begin{itemize}
\item
    $\mcA = \red \mcA$ and $\mcA = \mono \mcA$; if not, replace $\mcA$ by $\red \mcA \cap \mono \mcA$.
\item
    $\varepsilon$ should be a geometric isomorphism (at least on the image of $\Gamma$).
\item
    Every geometric isomorphism should be an isomorphism; if not, replace $\mcA$ by $\fix\mcA$.
\end{itemize}

\section{Cone small object argument}

Here we present a way of seeing local objects and admissible maps at the same time in a uniform way. This will allow us to present a variation of the small object argument that constructs all the local forms at once and endows this collection with a universal property -- it is a (multi)reflection. Our formulation takes place in the product completion of $\mcA$.

\subsection{Product completion}

First a quick reminder on product completions. We define $\prod \mcA$ as the opfibrational Grothendieck construction of the functor $\Set^\op \to \CAT$, $I \mapsto \mcA^I$, so that it forms an opfibration over $\Set^\op$. We will now describe this category in more elementary terms. An object of $\prod \mcA$ is a pair $(I, (A_i)_{i \in I})$ where $I$ is a set and $(A_i)_{i \in I}$ is an $I$-indexed collection in $\mcA$; since the indexing set $I$ is implicit in the collection, we will use just $(A_i)_{i \in I}$ or even $(A_i)$ to denote this object. A singleton family consisting of $A$ will be denoted simply by $A$, i.e.\ we consider $\mcA$ as a subcategory of $\prod \mcA$. Finally the empty collection will be denoted $(\ )$. A map $f \colon (A_i)_{i \in I} \to (B_j)_{j \in J}$ consists of a map of sets $J \to I$ and a collection of maps $A_{f(j)} \to B_j$ for each $j \in J$. Such a map is then a product of the maps $f_i \colon A_i \to (B_j)_{j \in J_i}$, where $J_i = f^{-1}(i)$.
\[\xymatrix@=0.5pc{
& & A_i \ar[ldd]_-{f_j} \ar[rdd]^-{f_{j'}} \\
f_i \colon {} \\
& B_j & \cdots & B_j'
}\]
We think of each $f_i$ as a cone, since it is just a collection of maps $f_j \colon A_i \to B_j$ with a common domain, the \emph{components} of the cone. In the opposite direction, we will say that the $f_i$ are \emph{factors} of $f$ and, more generally, a factor is the product of the $f_i$ over any subset of $I$.

\subsection{Cone weak factorization systems}

We will now study lifting properties which we denote by $f \boxslash g$. A lifting problem in $\prod \mcA$ looks like
\[\xymatrix{
(A_i) \ar[r] \ar[d]_-f & (X_k) \ar[d]^-g \\
(B_j) \ar[r] \ar@{-->}[ru] & (Y_l)
}\]
and is similarly a product of lifting problems (the products in the diagram are taken in $\prod A$)
\[\xymatrix{
A_i \ar[r] \ar[d]_-{f_i} & (X_k)_{k \in K_i} \ar@{=}[r] \ar[d]^-{g_i} & \prod_{k \in K_i} X_k \ar[d]^-{\prod_{k \in K_i} g_k} \\
(B_j)_{j \in J_i} \ar[r] \ar@{-->}[ru] \ar@{.>}[rru] & (Y_l)_{l \in L_i} \ar@{=}[r] & \prod_{k \in K_i} (Y_l)_{l \in L_k}
}\]
and solving the original problem is equivalent to solving each of these problems. We see that $f \boxslash g$ iff for every decomposition of $g$ into factors $g_i$ we have $\forall i \colon f_i \boxslash g_i$. Since one can use identities on the empty collection for some of these factors $g_i$, it is easy to see that $f \boxslash g$ implies $f_i \boxslash g$ for all $i$. In other words, the left class $\mcL$ in any weak factorization system $(\mcL, \mcR)$ is closed under factors. As usual, $\mcR$ is closed under products. The above analysis also shows the implications: $\mcL$ is closed under products with the identity on the initial object $\Ra$ $\mcR$ is closed under factors $\Ra$ $\mcL$ is closed under products.

\begin{lemma}
Let $(\mcL, \mcR)$ be a weak factorization system on $\prod \mcA$. Then $\mcL$ is closed under products iff $\mcR$ is closed under factors. \qed
\end{lemma}

\begin{definition}
A \emph{cone weak factorization system} is a system in the product completion that satisfies the equivalent conditions of the previous lemma.
\end{definition}

Intuitively, both classes $\mcL$ and $\mcR$ are determined by the cones that they contain. We will thus say that a cone weak factorization system is \emph{generated} by a collection $\Cone$ of cones $a_i \colon A_i \to (B_{ij})_{j \in J_i}$ if $\mcR$ is the class of all factors of maps in $\mcI^\boxslash$ and $\mcL$ is the corresponding class ${}^\boxslash\mcR$ (we do not need to add products since $\mcR$ is closed under factors).

Again, we call the elements of $\mcL$ localizations and the elements of $\mcR$ admissible maps. It will be of some importance to describe explicitly what it means for $R \to (S_l)$ to be admissible, depending on the cardinality of the indexing set. For the empty indexing set, the lifting problem
\[\xymatrix{
A_i \ar[r] \ar[d]_-{a_i} & R \ar[d] \\
(B_{ij}) \ar[r] \ar@{-->}[ru] & (\ )
}\]
is exactly the localness of $R$ with respect to the cone $a_i$ from the introduction. For a singleton, the lifting problem
\[\xymatrix{
A_i \ar[r] \ar[d]_-{a_i} & R \ar[d]^-f \\
(B_{ij}) \ar[r] \ar@{-->}[ru] & S
}\]
is exactly the injectivity of $f$ with respect to each component $a_{ij}$, i.e.\ the admissibility from the introduction. The case of more than one factor is of limited use for us, except when related to the collection $\Cone^{\downarrow}$ of the cone components rather than the collection $\Cone$ of cones:
\[\xymatrix{
A_i \ar[r] \ar[d]_-{a_{ij}} & R \ar[d]^-f \\
B_{ij} \ar[r] \ar@{-->}[ru] & (S_l)
}\]
is exactly the joint admissibility of $f$ with respect to the components $a_{ij}$ or, equivalently, the admissibility of the map $R \to \prod S_l$ in $\mcA$ (with the product interpreted in $\mcA$).

\subsection{Cone small object argument}

We will now present a variation of the small object argument that takes into account the compatibility with products. We need to assume smallness of the $A_i$; to make the notation simpler, we will assume that these are finitely presentable. Because of our application, we will formulate the small object argument directly for a map $R \to (\ )$, i.e.\ it produces an $\mcL$-injective (=local) replacement of $R$.

Making $(R_k) \to (\ )$ local is equivalent to making local each of the factors $R_k \to (\ )$. This is achieved as usual by attaching cells in all possible ways:
\[\xymatrix{
\coprod_{i \in I,\, f \colon A_i \to R_k} A_i \ar[r] \ar[d] & R_k \ar[d] \\
\coprod_{i \in I,\, f \colon A_i \to R_k} (B_{ij}) \ar[r] & R_k' \po
}\]
where the bottom left is $(\coprod_{i \in I,\, f \colon A_i \to R_k} B_{ij(i, f)})$ with the product indexed over all functions $j$ taking a pair $(i, f)$ to some $j(i, f) \in J_i$; in plain words, the function chooses for each attaching map a shape of the attached cell and the product then ranges over all possible choices and then so does the pushout $R_k'$. Applying this procedure for each factor $R_k$ and taking product of the results then produces $(R_k)' = \prod_k R_k'$; this is then indexed by the factor $k$ and a choice function $j$ for the cell shapes. Finally, one applies this procedure countably many times $R^{(0)} = R$, $R^{(n+1)} = R^{(n)}{}'$ and takes the colimit $R^{(\infty)} = \colim_n R^{(n)}$. The factors of $R^{(\infty)}$ are now indexed by a sequence of choice functions, one for each step; in effect, in each step one glues cells along all possible attaching maps, choosing always cell shapes in all possible ways. One proves in the usual way that this produces a local object, using finite presentability of the $A_i$.

We will now prove a crucial property of this version of small object argument when applied to cones $a_i \colon A_i \to (B_{ij})$ with all components epic. To make things easier, we add to this cone also all the possible wide pushouts of the components (this clearly does not affect the cone injectivity). We may then determine, for any map $f \colon A_i \to P$ to a local object, the maximal $B_{ij}$ to which $f$ admits an extension. We will now show that $R^{(\infty)}$ contains all local forms of $R$. Thus, let $p \colon R \to P$ be a local form. For each $f \colon A_i \to R$ consider the maximal $B_{ij}$ for which an extension in
\[\xymatrix{
A_i \ar[r]^-f \ar[d] & R \ar@{ >->}[d]^-p \\
(B_{ij}) \ar@{-->}[r] & P
}\]
exists. Using this particular choice of $j = j(i, f)$ we obtain a factor $R_1$ of $R^{(1)}$ through which $p$ factors, so that every lifting problem as below has a solution in $R_1$:
\[\xymatrix@R=1pc{
A_i \ar[r]^-f \ar[dd] & R \ar@{ >->}[d]^-{p_1} \\
& R_1 \ar@{ >->}[d] \\
(B_{ij}) \ar[r] \ar@{-->}[ru] & P
}\]
(this is so for the maximal $B_{ij}$ and thus also for any smaller). Now proceed inductively with $R$ replaced by $R_1$ etc.\ and finally obtain a particular factor $R_\infty$ of $R^{(\infty)}$ such that
\[\xymatrix{
A_i \ar[r] \ar[d] & R_\infty \ar@{ >->}[d] \\
(B_{ij}) \ar[r] \ar@{-->}[ru] & P
}\]
i.e.\ the map $R_\infty \to P$ belongs to $\mcR$. Since it also belongs to $\mcL$ by Lemma~\ref{lemma:cancellation_properties} (cancellation lemma), it must be an isomorphism and thus $P$ is one of the factors of $R^{(\infty)}$.

The local object $R^{(\infty)}$ contains factors multiple times. By disposing off the extra occurences (which is achieved by a retract so that this is still an $(\mcL, \mcR)$-factorization), we obtain a map $R \to (P_\alpha)$ that has a strong universal property, namely it is a multi-reflection of $R$ into $\mcP$, i.e.\ it satisfies the following universal lifting property:
\[\xymatrix{
R \ar[r]^-f \ar@{ >->}[d] & Q \ar[d]^-\sim \\
(P_\alpha) \ar[r]_-\sim \ar@{-->}[ru]^-{\exists!}_-\sim & (\ )
}\]
The existence of a lift follows from factoring $f \colon R \cof P \loc Q$ using the ordinary weak factorization system associated with $\Cone^{\downarrow}$; since $P$ is then a local form of $R$, it is one of the $P_\alpha$ and, by construction, a unique such; the factorization is also unique.

\begin{theorem} \label{thm:multi_reflection}
The inclusion $\mcP \to \mcA$ admits a multi-reflection. For each $R \in \mcA$, this is the collection $R \to (P_\alpha)$ of all local forms of $R$ (one representative of each isomorphism class). \qed
\end{theorem}

\section{Reduction} \label{sec:reduction}

Using the formalism of the previous section, we will now reinterpret the construction of the localization $R_U$ associated with a distinguished open set $U$. Namely, we will construct a reduction functor $\red \colon \mcA \to \mcA$ -- a reflection on the full subcategory $\red \mcA$ of reduced objects. We will then show that the localization $R_{\Pts k}$ associated with the distinguished open set $\Pts k$ is exactly the reduction $\red K$, in particular independent of the localization $k$ used to describe $U = \Pts k$.

The unit of the multi-reflection from Theorem~\ref{thm:multi_reflection} is a map $\ell_R \colon R \to (P_\alpha)$ in the product completion $\prod \mcA$. Its components are the local forms $p_\alpha \colon R \to P_\alpha$ and these are in bijection with points of $\Spec R$. For this reason, we will index the local forms of $R$ by points of $\Spec R$, i.e.\ we will write $\alpha \in \Spec R$.

\subsection{General reduction}

For the later use, we will construct a reduction functor associated with any factorization system $(\mcE, \mcM)$ on $\mcA$ with $\mcE$ composed of \emph{epimorphisms}. We will then apply this to the (localization, admissible) factorization and later also to the (regular epi, mono) factorization. The previous section describes an extension of $(\mcE, \mcM)$ to a cone weak factorization system on $\prod \mcA$; alternatively, the factorization of $\ell_R$ in the display below can be executed in $\mcA$ by replacing the formal product in $\prod \mcA$ by the actual product in $\mcA$.

We say that $R$ is $\mcM$-\emph{reduced} if the map $\ell_R \colon R \to (P_\alpha)$ lies in $\mcM$. We denote the full subcategory of $\mcM$-reduced objects by $\red_\mcM\mcA$. For a general $R$, there is a unique $(\mcE, \mcM)$-factorization
\[\ell_R \colon R \xlra{\eta_R} \red_\mcM R \lra (P_\alpha).\]
By the uniqueness, $R$ is $\mcM$-reduced iff $\eta_R$ is an isomorphism.

We say that $f \colon R \to S$ is a \emph{geometric isomorphism}, if it induces an isomorphism of spectra $f^* \colon \Spec S \to \Spec R$.

\begin{theorem} \label{thm:properties_reduction_one}
$\red_\mcM \mcA$ is an epi-reflective subcategory of $\mcA$, the unit of the corresponding adjunction is $\eta_R \colon 1 \to \red_\mcM$ and it is a geometric isomorphism.
\end{theorem}

\begin{proof}
Let $R \in \mcA$ and $S \in \red_\mcM \mcA$. Since $\eta_R \colon R \to \red_\mcM R$ is epic (it belongs to $\mcE$), the precomposition map
\[\eta_R^* \colon \red_\mcM \mcA(\red_\mcM R, S) \lra \mcA(R, S)\]
is injective. It is also surjective -- a preimage of $f$ is obtained as $\eta_S^{-1} \circ \red_\mcM f$:
\[\xymatrix{
R \ar[r]^-f \ar[d]_-{\eta_R} & S \ar[d]^-{\eta_S}_-\cong \\
\red_\mcM R \ar[r]_-{\red_\mcM f} & \red_\mcM S
}\]


The second point is a simple application of Proposition~\ref{prop:geometric_isomorphism_easy}: Since $\eta_R$ is epic, and since every local form $p_\alpha$ factors through $\red_\mcM R$ by construction, the square
\[\xymatrix{
R \ar[r]^-{\eta_R} \ar@{ >->}[d]_-{p_\alpha} & \red_\mcM R \ar@{ >->}[d] \\
P_\alpha \ar@{=}[r] & P_\alpha \po
}\]
is a pushout by the proof of Lemma~\ref{lemma:cancellation_properties}.
\end{proof}

Thus, from the point of view of the associated spectrum, we may replace $R$ by its reduction $\red_\mcM R$. In fact, it is useful to work instead in the category $\red_\mcM \mcA$, but in order to do so, we need to equip it with a collection of cones and from the point of view of this section also with a factorization system.

The cones in $\red_\mcM \mcA$ are taken to be the $\mcM$-reductions $\red_\mcM a_i \colon \red_\mcM A_i \to (\red_\mcM B_{ij})$ of the cones in $\mcA$. The factorization system $(\mcE, \mcM)$ in $\mcA$ restricts to one in $\red_\mcM \mcA$, since the following lemma shows that any $(\mcE, \mcM)$-factorization of a map in $\red_\mcM \mcA$ stays in $\red_\mcM \mcA$.

\begin{lemma}
If $f \colon R \to S$ lies in $\mcM$ then $S \in\red_\mcM \mcA$ $\Ra$ $R \in \red_\mcM \mcA$.
\end{lemma}

\begin{proof}
Using that $\mcE$ consists of epis, the proof of Lemma~\ref{lemma:cancellation_properties} gives that $\ell_R \in \mcM$ in the following square:
\[\xymatrix{
R \ar[r]^-{f \in \mcM} \ar[d]_-{\ell_R} & S \ar[d]^-{\ell_S \in \mcM} \\
(P_\alpha) \ar[r] & (Q_\beta)
}\qedhere\]
\end{proof}


Now we compare the above introduced notions in $\red_\mcM \mcA$ with those in $\mcA$.

\begin{theorem} \label{thm:properties_reduction_two}
The following claims hold:
\begin{itemize}
\item
	$\Gamma$ takes values in $\red_\mcM \mcA$, so that $\mcP \subseteq \im \Gamma \subseteq \red_\mcM \mcA$.
\item
	The notions of a local object and of an admissible map in $\red_\mcM \mcA$ coincide with those in $\mcA$. Consequently, $\red_\mcM \mcP = \mcP$.
\item
	The adjunction $\Gamma \dashv \Spec$ restricts to the adjunction $\red_\mcM \Gamma \dashv \red_\mcM \Spec$. In particular, the notion of a local form in $\red_\mcM \mcA$ coincides with that in $\mcA$.
\item
	All objects of $\red_\mcM \mcA$ are $\mcM$-reduced in $\red_\mcM \mcA$.
\end{itemize} 
\end{theorem}

\begin{proof}
%
%
%
The adjunction $\Gamma \dashv \Spec$ gives equivalence
\[\xymatrix{
R \ar[r]^-f \ar[d]_-{\eta_R} & \Gamma X \ar@{}[rd]|-{\equiv} & & \Spec \red_\mcM R \ar[d]^-\cong \\
\red_\mcM R \ar@{-->}[ru] & & X \ar[r] \ar@{-->}[ru] & \Spec R
}\]
and Theorem~\ref{thm:properties_reduction_one} gives the iso on the right. Thus, $\Gamma X$ is injective w.r.t.\ the unit $\eta_R$; since $\eta_R$ is epic, this easily gives that $\Gamma X \in \red_\mcM \mcA$ (take $f = 1_{\Gamma X}$).

The second point follows easily from the reflectivity: For $R \in \red_\mcM \mcA$,
\[\xymatrix{
A_i \ar[r]^-f \ar[d]_-{a_i} & R \ar@{}[rd]|-{\equiv} & & \red_\mcM A_i \ar[r]^-f \ar[d]_-{\red_\mcM a_i} & R \\
(B_{ij}) \ar@{-->}[ru] & & & (\red_\mcM B_{ij}) \ar@{-->}[ru] 
}\]
Together with $\mcP \subseteq \red_\mcM \mcA$ of the previous point, it follows that $\red_\mcM \mcP = \mcP$.

The previous two points yield $\rMArMPTop = \APTop$ so that $\Spec$ takes values in $\rMArMPTop$ and $\Gamma$ takes values in $\red_\mcM \mcA$; this easily gives the third point -- local forms $R \to P$ are the adjuncts of those maps $P^* \to \Spec R$ that are terminal in the components of $\mcP^* /_{\APTop} \Spec R$.

For the last point, the unit of the multi-reflection for $\red_\mcM \mcA$ is the same as that for $\mcA$ and as such lies in $\mcM$.
\end{proof}

\begin{proposition} \label{prop:geometric_isomorphism_easy}
Let $f \colon R \to S$ be a map whose pushout along any local form
\[\xymatrix{
R \ar[r]^-{f} \ar@{ >->}[d]_-p & S \ar@{ >->}[d]^-{f_*p} \\
P \ar[r]_-\cong & f_* P \po
}\]
is an isomorphism. Then $f$ is a geometric isomorphism.
\end{proposition}

Later, we will also prove a converse to this statement, see Theorem~\ref{thm:characterizing_spec_iso}.

\begin{proof}
The action of $f^*$ on a point $q \in \Spec S$, i.e.\ on a local form $q \colon S \to Q$, is obtained by factoring the composite $qf$ into a localization followed by an admissible map
\[\xymatrix{
R \ar[r]^-f \ar@{ >->}[d]_-{f^*q} & S \ar@{ >->}[d]^-q \\
f^*Q \ar[r]_-\sim & Q
}\]
and we show now that $f_*$ and $f^*$ are mutually inverse: Clearly, the pushout square in the statement can be seen as such a factorization for $Q = f_*P$ and thus $P \cong f^* f_* P$. The factorization of the above square through the pushout is a map $f_* f^* Q \to Q$ and Lemma~\ref{lemma:cancellation_properties} shows that it is both a localization and admissible, hence an iso. The action of $f^*$ on stalks is identified with the bottom map $P \to f_*P$ in the square from the statement and is thus an iso. It remains to show that $f^*$ is open, so let $l \colon S \to L$ be a finite localization and $p \in (f^*)^{-1} \Pts l$, i.e.\ $f_* p \in \Pts l$. We thus have the following diagram
\[\xymatrix{
R \ar[r] \ar[dd] & S \ar@{ >->}[d] \\
& L \ar[d] \\
P \ar[r] & f_* P
}\]
and we need to prove existence of a finite localization $R \to K$ containing $P$ for which
\[\underbrace{(f^*)^{-1} \Pts k}_{\Pts (f_*k)} \subseteq \Pts l,\]
i.e.\ such that there is a factorization $S \to L \to f_*K$. Since $P$ is a filtered colimit of finite localizations $K$, its pushout $f_*P$ is a filtered colimit of the pushouts $f_* K$ and the map $L \to f_* P$ factors through some of the $f_* K$, by finiteness of the localization $L$.
\end{proof}

\subsection{Admissible case}

The most important instance of the previous construction is the case of the factorization system generated by the components of the cones $C^\downarrow$. The resulting reduction will be denoted simply by $\red$ instead of $\red_\mcM$ for the class of admissible maps $\mcM$. In concrete terms, we say that $R$ is \emph{reduced} if $\ell_R \colon R \to (P_\alpha)$ is admissible; for general $R$, we get a factorization
\[R \cof \red R \loc (P_\alpha).\]

\begin{proposition}
$\red R$ is the largest localization of $R$ among those localizations $R \cof K$, for which $\Pts k = \Spec R$.
\end{proposition}

\begin{proof}
In the diagram below, the bottom map exists by the assumption,
\[\xymatrix{
R \ar@{ >->}[r] \ar@{ >->}[d] & \red R \ar[d]^-\sim \\
K \ar[r] \ar@{-->}[ru] & (P_\alpha)
}\]
hence also the diagonal.
\end{proof}

This easily implies the claim from the section introduction.

\begin{corollary}
$R_U = \red K$ for any distinguished open set $U = \Pts k$. \qed
\end{corollary}

We have already proved that the distinguished open sets $U$ are faithfully described by the localizations $R_U$, here we present a short alternative proof in terms of reductions.

\begin{lemma}
Let $K$ and $L$ be two localizations of $R$. Then $\Pts k \subseteq \Pts l$ iff there exists a factorization $\red L \cof \red K$.
\end{lemma}

\begin{proof}
Similarly to the previous proof:
\[\xymatrix@C=1pc{
R \ar@{ >->}[r] \ar@{ >->}[d] & K \ar@{ >->}[r] & \red K \ar[d]^-\sim \\
\red L \ar[rr] \ar@{-->}[rru] & & \leftbox{(P_\alpha)}{{}_{\alpha \in \Pts k}}
}\qedhere\]
\end{proof}

\begin{corollary} \label{cor:dist_open_finite_localization_correspondence}
Assuming $\mcA = \red \mcA$, isomorphism classes of finite localizations $k \colon R \cof K$ are in bijection with distinguished opens $\Pts k \subseteq \Spec R$ and the canonical presheaf is obtained as the right Kan extension of the partial presheaf $\Pts k \mapsto K$. \qed
\end{corollary}

In $\red \mcA$, we obtain a ``reduced'' cone weak factorization system, generated by cones $\red a_i \colon \red A_i \to (\red B_{ij})$. According to Theorem~\ref{thm:properties_reduction_two}, the local objects, admissible maps and local forms are exactly those of $\mcA$. Finite localizations are more complicated and it is thus not immediately clear that the distinguished opens in $\Spec R$ and in $\Spec \red R$ agree; this is addressed in Section~\ref{sec:reduced_localizations}. Theorem~\ref{thm:properties_reduction_two} gives that all objects are reduced w.r.t.\ $\mcA$-admissible maps in $\red \mcA$, but these are equal to $\red \mcA$-admissible maps, so that we get:

\begin{corollary}
All objects of $\red \mcA$ are reduced. \qed
\end{corollary}


\subsection{Geometric isomorphisms}

In general, $\Spec$ is not fully faithful and, in particular, it does not reflect isomorphisms. We recall that a map is called a geometric isomorphism if its image under $\Spec$ is an isomorphism.

\begin{theorem} \label{thm:characterizing_spec_iso}
A map $f \colon R \to S$ is a geometric isomorphism if and only if the pushout of $f$ along any local form becomes an isomorphism in $\red\mcA$,
\[\xymatrix{
R \ar[r]^-f \ar@{ >->}[d] & S \ar@{ >->}[d] \\
P \ar[r] \ar@/_1em/[rr]_-\cong & f_* P \ar[r] \po & \red(f_* P)
}\]
i.e.\ if and only if the composite across the bottom is an isomorphism.
\end{theorem}

\begin{proof}
%
According to Theorem~\ref{thm:properties_reduction_one}, the units $\eta_R$, $\eta_S$ are geometric isomorphisms, so $f$ is a geometric isomorphism iff $\red f$ is a geometric isomorphism:
\[\xymatrix{
R \ar[r]^-f \ar[d]_-{\eta_R} & S \ar[d]^-{\eta_S} \\
\red R \ar[r]_-{\red f} & \red S
}\]
In addition, Theorem~\ref{thm:properties_reduction_two} says that the two potential notions of a geometric isomorphism, defined through $\Spec$ and $\red\Spec$, are the same. We may thus work in the category $\red \mcA$ all the time and ignore all the reductions. The backward implication is then exactly Proposition~\ref{prop:geometric_isomorphism_easy}.

For the forward implication, i.e.\ assuming that $f$ is a geometric iso, denote the unique preimage of $p$ under $f^*$ by $q$; we get a square
\[\xymatrix{
R \ar[r]^-{f} \ar@{ >->}[d]_-p & S \ar@{ >->}[d]^-q \\
P \ar[r]_-\cong & Q
}\]
with the bottom map iso since it is the action of $f^*$ on stalks. It remains to show that this square is cocartesian. To this end, express $P$ as a colimit of all its finite sublocalizations $K$; then the distinguished open sets $\Pts k$ form a neighbourhood basis of $p \in \Spec R$. Since $f^*$ is a homeomorphism, their preimages $\Pts(f_* k)$ form a neighbourhood basis of $q \in \Spec S$ and thus the $f_* K$ form a cofinal subdiagram in the diagram of all finite sublocalizations of $Q$, by Corollary~\ref{cor:dist_open_finite_localization_correspondence}. In particular, $Q$ is the colimit of the $f_* K$ and is thus itself a pushout of $P$.
\end{proof}

We can view the theorem as an obstruction to $\Spec$ being fully faithful, expressed solely in terms of $\mcA$ (or $\red\mcA$); coupled with results of Section~\ref{sec:flat_local_forms}, this will provide a sufficient condition for $\Spec$ being fully faithful. Say that \emph{local forms reflect isomorphisms} if pushouts of $f \colon R \to S$ along all local forms of $R$ are iso $\Ra$ $f$ is iso.

\begin{corollary} \label{cor:Spec_conservative}
Assume that $\mcA = \red\mcA$. Then $\Spec \colon \mcA^\op \to \APTop$ reflects isomorphisms iff local forms reflect isomorphisms. \qed
\end{corollary}

\subsection{Monomorphism case}

Here we assume that $\mcA$ is locally presentable, even though some parts only require that a (strong epi, mono) factorization exists in $\mcA$.

We can apply the general reduction construction to the (strong epi, mono) factorization system, denoting the reduction by $\mono$ instead of $\red_\mcM$ for the class of monos $\mcM$. In concrete terms, we say that $R$ is \emph{mono-reduced} if $\ell_R \colon R \to (P_\alpha)$ is monic. This will be utilized in Section~\ref{sec:covers}. Here we want to elaborate on the last point of Theorem~\ref{thm:properties_reduction_two}, saying that all objects of $\mono \mcA$ are reduced w.r.t.\ $\mcA$-monos in $\mono \mcA$. Since $\mono \mcA$ is a reflective subcategory, the $\mcA$-monos in $\mono \mcA$ are exactly the $\mono \mcA$-monos, so we can conclude:

\begin{corollary}
All objects of $\mono \mcA$ are mono-reduced. \qed
\end{corollary}

\begin{remark}
Interestingly, we can cofibrantly generate a factorization system $(\mcE_\lambda, \mcM_\lambda)$ by all epis between $\lambda$-presentable objects and by the general theory, the canonical map $\ell_R \colon R \to (P_\alpha)$ lies in $\mcM_\lambda$ for any fixed point $R \in \fix \mcA$. Since this holds for every $\lambda$, this canonical map $\ell_R$ is a strong mono.
\end{remark}

\begin{remark}
There is also a cofibrantly generated factorization system $(\mcE, \mcM)$ whose right class consists of maps that are monic and admissible at the same time. We then get a reflection onto the intersection $\mono \mcA \cap \red \mcA$.
\end{remark}

\begin{remark}
If $\mcA$ is in fact locally \emph{finitely} presentable then, assuming $\mcA = \red \mcA$, the condition $\mcA = \mono \mcA$ is equivalent to the canonical presheaves $\mcO_R^\can$ on $\Spec R$ being monopresheaves (this is almost clear for distinguished opens and extends easily to all opens by taking appropriate limits). Thus, in order to produce a sheafification $\mcO_R$, it is enough to perform the plus-construction once (see Section~\ref{subsec:Heller_Rowe_formula}).
\end{remark}

\section{First consequences of fully faithful $\Spec$}



For the further use, we record the following implication:

\begin{lemma} \label{lemma:Spec_ff_implies_reduced}
If $\Spec$ is fully faithful then $\mcA = \red\mcA$.
\end{lemma}

\begin{proof}
If $\fix \mcA \subseteq \red \mcA \subseteq \mcA$ is equality, both inclusions have to be.
\end{proof}

First we study a simple criterion for $\Spec$ being fully faithful. By the above lemma, we may restrict to the case $\mcA = \red\mcA$.

\begin{proposition} \label{prop:Spec_fully_faithful_ORcan_sheaf}
Assume that $\mcA = \red\mcA$. Then $\Spec \colon \mcA^\op \to \APTop$ is fully faithful iff the canonical presheaf $\O_R^\can$ is a sheaf for each $R$ (then it equals $\mcO_R$).
\end{proposition}

\begin{proof}
$\Spec$ is fully faithful if and only $\varepsilon$ is an isomorphism:
\[\varepsilon_R \colon R= \mcO^\can_R(R) \to \mcO_R(R) = \Gamma \Spec R.\]
Thus, the counit is an isomorphism for all finite localizations $K$ of $R$ iff $\mcO_R^\can \to \mcO_R$ is an isomorphism on all distinguished opens $\Pts k \subseteq \Spec R$. It remains to show that in this case, it is an isomorphism for every open $U \subseteq \Spec R$. As a right Kan extension, $\mcO_R^\can(U)$ is given by the limit
\[\mcO_R^\can(U) = \lim_{\Pts k \subseteq U} \mcO_R^\can(\Pts k)\]
over the canonical cover of $U$ by all distinguished opens. Since the same is true for any sheaf such as $\mcO_R$, the canonical map $\mcO_R^\can(U) \to \mcO_R(U)$ is a limit of isomorphisms and thus an isomorphism.
\end{proof}

For the rest of this section, we assume that $\Spec$ is fully faithful. We say that $X \in \APTop$ is \emph{affine} if it is isomorphic to $\Spec R$ for some $R \in \mcA$. Since any isomorphism $\Spec R \cong \Spec S$ is now induced by a unique isomorphism $R \cong S$ and this induces a bijection of sets of finite localizations, we may then say that a subset $U \subseteq X$ of an affine is \emph{distinguished open} if $U \cong \Spec K$ is also affine and the inclusion is induced by a finite localization $R \cof K$.

\begin{lemma} \label{lemma:affine_cancellation_properties}
Assume that $\Spec$ is fully faithful. Let $Z \subseteq Y \subseteq X$ be a chain of open embeddings of affines.
\begin{itemize}
\item
	(composition) If $Z$ is distinguished open in $Y$ and $Y$ is distinguished open in $X$ then $Z$ is distinguished open in $X$.
\item
	(cancellation) If $Z$ is distinguished open in $X$ then it is distinguished open in $Y$.	
\end{itemize}
\end{lemma}

\begin{proof}
The first point is easier, so we concentrate on the second which translates easily to the claim: If the composition $R \to S \to K$ is a finite localization then so is the second map. This holds by Lemma~\ref{lemma:cancellation_properties} provided that the first map is an epi. Since $\Spec$ is assumed fully faithful, we only need $Y \subseteq X$ monic, which is guaranteed by the next lemma.
\end{proof}

\begin{lemma}
Any open embedding of $\mcA$-spaces is a monomorphism.
\end{lemma}

\begin{proof}
Let $(\iota, \iota_\sharp)$ be an open embedding of $X$ into $Y$, i.e.\ $\iota \colon X \to Y$ is an open embedding of topological spaces and $\iota_\sharp \colon \iota^* \mcO_Y \to \mcO_X$ is an isomorphism. Consider a pair of maps
\[\xymatrix{
(\varphi, \varphi_\sharp) \colon T \ar@<.5ex>[r]^-{(\psi_0, \psi_{0\sharp})} \ar@<-.5ex>[r]_-{(\psi_1, \psi_{1\sharp})} & X \ar[r]^-{(\iota, \iota_\sharp)} & Y
}\]
coequalized by $(\iota, \iota_\sharp)$. Since $\iota$ is injective, we get $\psi_0=\psi=\psi_1$. On the level of structure sheaves, 
\[\xymatrix{
\psi^* \iota^* \mcO_Y \ar[r]^-{\psi^*\iota_\sharp} & \psi^*\mcO_X \ar@<.5ex>[r]^-{\psi_{0\sharp}} \ar@<-.5ex>[r]_-{\psi_{1\sharp}} & \mcO_T
}\]
the two maps $\psi_{0\sharp}$, $\psi_{1\sharp}$ get equalized by an isomorphism $\psi^* \iota_\sharp$ and are thus equal.
\end{proof}

%
%
%

Now, we are ready to prove the statement we chose to call ``Affine Communication Lemma'' as it is a generalization of a technical lemma (5.3.1 in \cite{Vakil}) which subsumes the principle usually called the Affine Communication Lemma by algebraic geometers (5.3.2 in \cite{Vakil}).

\begin{theorem} \label{thm:affine_communication_lemma}
Assume that $\Spec$ is fully faithful. Let $U, V \subseteq X$ be two affine opens and $x \in U \cap V$. Then there exists an open affine $W \ni x$ that is distinguished open both in $U$ and $V$.
\end{theorem}

\begin{proof}
Since distinguished open subsets form a basis of topology on affines, we can find $W' \ni x$ contained in $U \cap V$ that is distinguished open in $U$. Further, we can find $W \ni x$ contained in $W'$ that is distinguished open in $V$.
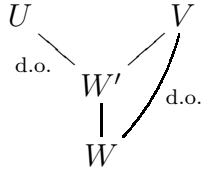
\[\xymatrix@=1pc{
U \ar@{-}[rd]_-{\textrm{d.o.}} & & V \ar@{-}[ld] \ar@{-}@/^1ex/[ldd]^-{\textrm{d.o.}}\\
& W' \ar@{-}[d] \\
& W
}\]
By Lemma~\ref{lemma:affine_cancellation_properties}, we know that $W$ is also distinguished open in $W'$ and thus also in $U$.
\end{proof}

%

\section{Covers, sheafification} \label{sec:covers}

In order to study the counit $\varepsilon \colon R \to \overline R = \Gamma \Spec R$, we need to get hold of the global sections of the structure sheaf $\mcO_R$, i.e.\ of the sheafification of the canonical presheaf. We will make use of the classical formula for the sheafification defined via (hyper)covers, so we will start by setting up notation for them. For us, a (hyper)cover will be a certain simplicial object in the coproduct completion, defined via a condition on its matching maps.


\subsection{Semi-simplicial objects, matching maps}

For an augmented semi-simplicial object
\[\xymatrix{
\cdots \ar@<1ex>[r] \ar@<0ex>[r] \ar@<-1ex>[r] & U_1 \ar@<.5ex>[r] \ar@<-.5ex>[r] & U_0 \ar[r] & U_{-1}
}\]
we define the $n$-th \emph{matching object} $M_nU$ by truncating the object $<n$, then right Kan extending and finally taking the $n$-th object of the result. We will suffice with $M_0U = U_{-1}$ and $M_1U = U_0 \times_{U_{-1}} U_0$ (the kernel pair of the augmentation map). The $n$-th \emph{matching map} is the canonical map $U_n \to M_nU$.

Dually, for an augmented semi-cosimplicial object
\[\xymatrix{
K^{-1} \ar[r] & K^0 \ar@<.5ex>[r] \ar@<-.5ex>[r] & K^1 \ar@<1ex>[r] \ar@<0ex>[r] \ar@<-1ex>[r] & \cdots
}\]
the latching objects $L^nK$ and latching maps $L^nK \to K^n$ are defined using left Kan extensions, with $L^0K = K^{-1}$ and $L^1K = K^0 +_{K^{-1}} K^0$ the cokernel pair of the augmentation map.

\subsection{Covers}

We consider a category $\mcC$ equipped with a Grothendieck topology $J$. By a \emph{cover} of $X \in \mcC$, we will understand a map $U_0 \to X$ in the coproduct completion $\coprod \mcC$ with a singleton codomain whose components form a $J$-cover of $X$ in $\mcC$. More generally, $U_0 \to (X^i)^{i\in I}$ is a cover if each of its summands is a cover of $X^i$ in the previous sense. For us, a \emph{hypercover} will be an augmented semi-simplicial object $U_\bullet \colon \Delta_+^\op \to \coprod \mcC$ such that the matching maps $U_n \to M_nU$ are covers (for our purposes, only $n \leq 1$ will be relevant):
\[\xymatrix{
\cdots \ar@<1ex>[r] \ar@<0ex>[r] \ar@<-1ex>[r] & U_1 \ar@<.5ex>[r] \ar@<-.5ex>[r] & U_0 \ar[r] & X
}\]
This means explicitly that $U_0 = (U_0^i)^{i \in I_0} \to X$ is a cover and that for any $i_0, i_1 \in I_0$ the collection $(U_1^j)^{j \in I_1/(i_0,i_1)} \to U_0^{i_0} \times_X U_0^{i_1}$ is a cover, where we denote by $I_1/(i_0, i_1)$ the subset of the index set $I_1$ of elements mapping to $i_0$ and $i_1$ via the face maps $d_1$ and $d_0$.

The semi-simplicial object $U_\bullet$ is a convenient presentation of a larger diagram $\El U_\bullet$ that contains the same data but not grouped into formal coproducts, see Appendix~\ref{section:Grothendieck_constructions}. Concretely, $\El U_\bullet$ takes $(n, i)$ to $U_n^i$ and $d_k \colon j \mapsto i$ to $d_k^j \colon U_n^j \to U_{n-1}^i$. We abbreviate $H_0 U_\bullet = \colim \El U_\bullet$.

\begin{example}
In $\Top$ or $\APTop$, where the $J$-covers are understood to be jointly surjective collections of open embeddings, $H_0 U_\bullet$ is obtained by gluing the $U_0^i$ together along all the relevant $U_1^j$ and is thus isomorphic to $X$.
\end{example}

By Lemma~\ref{lem:limits_colimits_coproduct_completion}, $H_0 U_\bullet$ is equally obtained by taking $\colim U_\bullet$ in $\coprod\mcC$ and subsequently taking the coproduct of all the components (there may be multiple when $\El I$ is disconnected). Denoting the latter operation by $\Sigma \colon \coprod\mcC \to \mcC$, we thus have alternative presentations
\[H_0 U_\bullet \cong \Sigma \colim U_\bullet \cong \colim \Sigma U_\bullet\]
(the last since $\Sigma$ clearly preserves colimits).

Assume that $\mcA = \red\mcA$. A distinguished hypercover of $\Spec R$ is given by an augmented semi-simplicial object in $\coprod \APTop$ with matching maps distinguished covers, and is induced by an augmented semi-cosimplicial object
\[\xymatrix{
R \ar[r] & K^0 \ar@<.5ex>[r] \ar@<-.5ex>[r] & K^1 \ar@<1ex>[r] \ar@<0ex>[r] \ar@<-1ex>[r] & \cdots
}\]
in the product completion $\prod \mcA$ whose latching maps are distinguished opcovers in the sense of the following definition. We will then say that the above is a \emph{distinguished hyperopcover}.

\begin{definition}
A cone $R \to (K_i)$ whose components are finite localizations is \emph{opcovering} or a distinguished \emph{opcover} if every local form $R \cof P$ factors through one of these components; diagramatically in $\prod \mcA$,
\[\xymatrix{
R \ar@{ >->}[r] \ar[d] & P \\
(K_i) \ar@{-->}[ru]
}\]
\end{definition}

We will also have a chance to encounter opcovers by localizations that are not necessarily finite; their definition is obvious.

\begin{lemma}
Distinguished opcovers are closed under pushouts and compositions. \qed
\end{lemma}

\subsection{Heller--Rowe formula} \label{subsec:Heller_Rowe_formula}

Later, we will need a concrete description of the sheafification functor, given by the Heller--Rowe formula
\[(\sheafify F) V = \colim_{U_\bullet \to V} H^0(FU_\bullet)\]
where the colimit runs over hypercovers $U_\bullet$ of $V$. This formula works under various sets of assumptions, we will now briefly comment on the case of a locally finitely presentable $\mcA$. A hypercover version of \cite[Proposition~B.III.2.2]{Barr_Grillet_vanOsdol} shows that $\sheafify F$ is a presheaf through which any map from $F$ to a sheaf factors uniquely (this is a bit technical), so that it remains to prove that $\sheafify F$ is a sheaf. Since the functors $\mcA(R, -)$, for finitely presentable $R$, preserve limits and jointly reflect isomorphisms, it is enough to show that $\mcA(R, \sheafify F)$ satisfies the sheaf condition; but since $\mcA(R, -)$ also preserves filtered colimits, we easily get that $\mcA(R, \sheafify F) \cong \sheafify \mcA(R, F)$ and as such is a sheaf by the classical result in $\Set$.

A related result, with analogous proof, says that for a monopresheaf $F$ (such as the canonical presheaves in the case $\mcA = \mono\mcA$), one may restrict the colimit to covers (instead of hypercovers). In both versions, when $V$ is compact, one may further restrict to finite (hyper)covers as they form a cofinal subcategory.

Corollary~\ref{cor:dist_open_finite_localization_correspondence} now easily implies that $\overline R$ is given by the formula
\[\overline R = \colim_{R \to K^\bullet} H^0 K^\bullet\]
(the colimit ranges over all distinguished hyperopcovers of $R$). This is why it is important to have a condition for $R \to H^0 K^\bullet$ being an isomorphism. The lemma below represents a rather degenerate case (a hyperopcover where one of the constituents is everything).

\begin{lemma}
Assume that $\mcA = \red\mcA$.  If one of the components $R \to K^0_{i_0}$ of an augmented semi-cosimplicial object is an isomorphism and $d^0_i \colon K^0_i \to (K^1_j)_{j \in I^1/(i_0,i)}$ is monic for all $i$, then the map $R \to H^0 K^\bullet$ is an isomorphism.
\end{lemma}

\begin{proof}
We will show that the inverse is given by $H^0 K^\bullet \to K^0_{i_0} \to R$ where the first map is the limit projection and the the second is the inverse to the isomorphism from the statement. Clearly one of the composites is $1_R$ and the other composite $H^0 K^\bullet \to H^0 K^\bullet$ is induced by the cone whose component $H^0 K^\bullet \to K^0_i$ is the composition across the top
\[\xymatrix{
& & R \ar[ld]_-\cong \ar[rd] \\
H^0 K^\bullet \ar[r] \ar@/_3ex/[rrr] & K^0_{i_0} \ar[rd]_-{d^1_j} & & K^0_i \ar[ld]^-{d^0_j} \\
& & K^1_j
}\]
The composition with $d^0_j$ easily gives the same map as the limit projection $H^0 K^\bullet \to K^0_i$. This holding for each $j$, the joint monicity shows that $H^0 K^\bullet \to H^0 K^\bullet$ is also the identity.
\end{proof}

%

Assuming in addition that $\mcA = \mono \mcA$, the monicity condition is automatically satisfied for distinguished hyperopcovers: The square in the previous proof is equivalently a map from the pushout $K^0_{i_0} +_R K^0_i \cong K^0_i$ and the hyperopcover condition thus says that $K^0_i \to (K^1_j)_{j \in I^1/(i_0,i)}$ is a distinguished opcover, i.e.\ there exists a factorization
\[K^0_i \to (K^1_j)_{j \in I^1/(i_0,i)} \to (P_\alpha)_{\alpha \in K^0_i}.\]
Since the composite is monic, so is the first map.

\begin{proposition} \label{prop:split_hyperopcover}
Assume that $\mcA = \red \mcA$ and $\mcA = \mono \mcA$. If one of the components $R \to K^0_{i_0}$ of a distinguished hyperopcover is an isomorphism, then the map $R \to H^0 K^\bullet$ is an isomorphism. \qed
\end{proposition}

\section{Compactness of $\Spec R$} 

In this section, we assume that the cones in $\Cone$ are finite, i.e.\ that each defining cone $a_i \colon A_i \to (B_{ij})_{j \in J_i}$ has a finite indexing set $J_i$. We will prove that $\Spec R$ is then  compact. This was proved in~\cite{Aratake} by identifying $\Spec R$ with a spectrum of a distributive lattice that is a spectral space hence compact. By combining this proof with that of compactness of a spectrum of a distributive lattice, one obtains the following proof which we present here for completeness, but we strongly recommend reading the proof in~\cite{Aratake}.

\begin{theorem} \label{theorem:compactness}
Assume that all cones in $\Cone$ are finite. Then $\Spec R$ is compact for every $R \in \mcA$.
\end{theorem}

\begin{proof}
By Alexander lemma, it suffices to find a finite subcover for any cover by subbasic opens, in our case this means by distinguished opens. Thus, let $\mcL = \{L_k\}$ be an opcovering family of finite localizations. Assuming that no finite sub-opcover exists, we will produce a local form that does not factor through any of the $L_k$, yielding a contradiction.

An adapted filter is a subset $\mcF \subseteq \FinLoc(R)$ satisfying
\begin{itemize}
\item
	$\mcF$ is upward closed,
\item
	$\mcF$ is closed under finite distinguished opcovers, i.e.\ if $L$ admits a finite opcover by elements of $\mcF$ then $L$ belongs to $\mcF$.
\end{itemize}
An adapted filter $\langle \mcL \rangle$ generated by $\mcL$ is obtained by first closing the set $\mcL$ upwards and then by closing it under finite opcovers. Thus, the bottom element $R \in \langle \mcL \rangle$ iff $R$ admits a finite opcover from $\mcL$.

Assuming that no finite subcover exists, $\langle \mcL \rangle$ does not contain $R$, i.e.\ it is a proper adapted filter. Let now $\mcF$ be a maximal proper adapted filter above $\langle \mcL \rangle$, which exists by Zorn lemma. We will now show that it is prime in the following sense: If the pushout of $L_0$ and $L_1$ belongs to $\mcF$ then one of the $L_i$ belongs to $\mcF$. For otherwise, there exist two finite opcovers of $R$ -- each by $L_i$ and some finite subset of $\mcF$. Taking pushouts of these finite opcovers, $R$ admits an opcover by $L_0 +_R L_1 \in \mcF$ and a finite subset of $\mcF$, and thus $R \in \mcF$, a contradiction. It is now easy to see that the complement $\mcI$ of $\mcF$ is a collection of finite localizations that is non-empty, downward closed and closed under pushouts, so its colimit $P = \colim \mcI$ is a localization. By filteredness of $\mcI$ and by the second property of $\mcF$, this $P$ is then a local form. Since $L_k \notin \mcI$ and $\mcI$ is downward closed, $P$ does not factor through any of the $L_k$, a contradiction. Thus, a finite sub-opcover exists.
\end{proof}

\section{Flat local forms} \label{sec:flat_local_forms}

%

Throughout this section, we assume that $\mcA$ is locally finitely presentable with $\mcA = \red \mcA$ and $\mcA = \mono \mcA$.

We say that a map $f \colon R \to S$ is \emph{flat} if the cobase change $f_* \colon R/\mcA \to S/\mcA$ preserves finite limits. We will study the counit of the adjunction $\Gamma \dashv \Spec$; for simplicity we denote $\overline R = \Gamma \Spec R$ so that we write the counit as $\varepsilon_R \colon R \to \overline R$.

\begin{theorem} \label{thm:good_objects_one}
If $\Spec R$ is compact and all local forms of $R$ are flat then the counit $\varepsilon_R \colon R \to \overline R$ is a geometric isomorphism.
\end{theorem}

The same conclusion holds if all finite localizations of $R$ are flat. This may be obtained as a special case of Corollary~\ref{cor:good_objects_two} or by a simple modification of the proof below.

\begin{proof}
Consider a finite distinguished hyperopcover
\[\xymatrix{
R \ar[r] & K^0 \ar@<.5ex>[r] \ar@<-.5ex>[r] & K^1 \ar@<1ex>[r] \ar@<0ex>[r] \ar@<-1ex>[r] & \cdots
}\]
By taking a pushout along any local form $p \colon R \to P$, we obtain
\[\xymatrix{
R \ar[r]^-f \ar[d]_-p & H^0 K^\bullet \ar[r] \ar[d] & K^0 \ar@<.5ex>[r] \ar@<-.5ex>[r] \ar[d] & K^1 \ar@<1ex>[r] \ar@<0ex>[r] \ar@<-1ex>[r] \ar[d] & \cdots \\
P \ar[r] & H^0 \widetilde K^\bullet \ar[r] \po & \widetilde K^0 \ar@<.5ex>[r] \ar@<-.5ex>[r] \po & \widetilde K^1 \ar@<1ex>[r] \ar@<0ex>[r] \ar@<-1ex>[r] \po & \cdots
}\]
where the leftmost square is cocartesian, since $p$ is assumed flat. By picking some $K^0_{i_0}$ through which $p$ factors, we get $P \cong \widetilde K^0_{i_0}$ by Lemma~\ref{lemma:factorization_pushout}; Proposition~\ref{prop:split_hyperopcover} then shows that $P \to H^0 \widetilde K^\bullet$ is an iso, so that Theorem~\ref{thm:characterizing_spec_iso} concludes that $f$ is a geometric iso. By passing to the limit
\[\Spec \overline R = \Spec \colim_{R \to K^\bullet} H^0 K^\bullet = \lim\nolimits_{R \to K^\bullet} \Spec H^0 K^\bullet \lra \Spec R,\]
we obtain the iso from the statement.
\end{proof}

\begin{example} \label{ex: classic spec ff}
In the category $\mcA$ of commutative rings with $\mcP$ the local rings and their local maps, every local form is flat. We conclude that every $R$ is a fixed point, since $\Spec$ reflects isomorphisms by virtue of Theorem~\ref{thm:characterizing_spec_iso} and the well known fact that a map of $R$-modules is an isomorphism iff it is so after localization with respect to every prime of $R$.
\end{example}


\begin{lemma}
Let $f \colon R \to S$ be a map through which a distinguished hyperopcover $R \to K^\bullet$ factors in such a way that the augmentation $S \to K^0$ is an opcovering collection of localizations (not necessarily finite). Then $f$ is a geometric isomorphism.
\end{lemma}

\begin{proof}
%
Since $S \to K^0$ is jointly surjective, $H_0 \Spec K^\bullet \to \Spec S$ is surjective, while $H_0 \Spec K^\bullet \to \Spec R$ is an isomorphism since all matching maps are covering collections of open embeddings.
\[\xymatrix{
\rightbox{\colim \El \Spec K^\bullet \cong {}}{H_0 \Spec K^\bullet} \ar[r]^-{\textrm{surj.}} \ar@{=}[d] & \Spec S \ar[d]^-{f^*} \\
H_0 \Spec K^\bullet \ar[r]_-{\cong} & \Spec R
}\]
Consequently, $f^*$ is bijective with a continuous inverse given by the composition $\Spec R \xla{\cong} H_0 \Spec K^\bullet \to \Spec S$. Moreover, all maps in the cone
\[\El \Spec K^\bullet \to \Spec S\]
are stalkwise isomorphisms (they are embeddings), the same is then true for $\colim \El \Spec K^\bullet \to \Spec S$ and thus also for $f^*$.
%
\end{proof}

We will now apply this to the situation $S = H^0 K^\bullet$ of the limit of any distinguished hyperopcover $R \to K^\bullet$, obtaining the following instance.

\begin{corollary} \label{cor:good_objects_two}
Assume that for any distinguished hyperopcover $R \to K^\bullet$ the induced map $H^0 K^\bullet \to K^0$ from the limit is an opcovering collection of localizations. Then the canonical map $\varepsilon_R \colon R \to \overline R$ is a geometric isomorphism.

If $\Spec R$ is compact, this condition is required only for finite distinguished hyperopcovers. \qed
\end{corollary}



Finally, we specialize the above general result to an important case easily applicable e.g.\ to the integral domains.

\begin{theorem} \label{thm:good_objects_two}
Suppose that admissible maps are monic and that jointly monic families of localizations are opcovering. Then the counit $\varepsilon_R \colon R \to \overline R$ is a geometric isomorphism; finite families are sufficient for compact spectra.
\end{theorem}

\begin{proof}
We will deduce this from the last corollary, so we need to verify its hypothesis. For any distinguished hyperopcover $R \to K^\bullet$ and the factorization
\[R \to H^0 K^\bullet \to K^0,\]
the dual of Lemma~\ref{lemma:cancellation_properties}.1 shows that $H^0 K^\bullet \to K^0$ consists of localizations. Since limit projections are always jointly monic, they are opcovering, by assumption.
\end{proof}

\begin{example} \label{example:spec_not_fully_faithful}
In the category $\mcA$ of reduced commutative rings with $\mcP$ integral domains and their monomorphisms, the above condition holds. This is because a finite family of localizations, i.e.\ of quotients $R \to R/I_k$, is jointly monic iff $I_1 \cap \cdots \cap I_n = 0$. In that case, any prime $\mathfrak{p}$ contains some $I_k$ and so the corresponding local form $R \to R/\mathfrak{p}$ factors through the localization $R \to R/I_k$, as required.

We have seen that for every commutative ring $R$ the counit $\varepsilon_R$ is a geometric iso. We will now give an example of a reduced ring where $\varepsilon_R$ is not an iso, i.e.\ $\Spec$ is not fully faithful even on $\red\mcA$. Every noetherian ring $R$ possesses the smallest distinguished opcover $(R/\mathfrak{p}_i)$ with $\mathfrak{p}_i$ ranging over all minimal primes of $R$. Thus, $\overline R$ is the $H^0$ of
\[\xymatrix{
(R/\mathfrak{p}_i) \ar@<.5ex>[r] \ar@<-.5ex>[r] & (R/\rad(\mathfrak{p}_i + \mathfrak{p}_j)) \ar@<1ex>[r] \ar@<0ex>[r] \ar@<-1ex>[r] & \cdots
}\]
In particular, let $R$ be the coordinate ring of a union $X$ of two smooth curves $X_1$, $X_2$ meeting at a single point $X_0$ non-transversely. Then the minimal prime $\mathfrak{p}_i$ consists of functions vanishing on $X_i$ and clearly $\mathfrak{p}_1 \cap \mathfrak{p}_2 = 0$, while $\mathfrak{p}_1 + \mathfrak{p}_2$ is a non-radical ideal whose radical is the prime $\mathfrak{p}$ consisting of functions vanishing at $X_0$. Thus, $R$ is the pullback of
\[R/\mathfrak{p}_1 \to R/(\mathfrak{p}_1 + \mathfrak{p}_2) \la R/\mathfrak{p}_2\]
while $\overline R$ is the pullback of
\[R/\mathfrak{p}_1 \to R/\underbrace{\rad (\mathfrak{p}_1 + \mathfrak{p}_2)}_{\mathfrak{p}} \la R/\mathfrak{p}_2\]
which describes the same two curves intersecting transversely at the point $X_0$. These commutative rings are not isomorphic, for the minimal primes $\overline{\mathfrak{p}}_i$ of $\overline R$ are obtained as kernels of the projection onto $R/\mathfrak{p}_i$ and it is easy to see that $\overline{\mathfrak{p}}_1 + \overline{\mathfrak{p}}_2$, being the kernel of the projection onto $R/\mathfrak{p}$, is also prime, hence radical, unlike in the situation of $R$.
\end{example}

\begin{remark}
There is the following converse to Theorem~\ref{thm:good_objects_one}. If $\Spec R$ is compact and $\ell_R$ is an isomorphism for all $R$, then the local forms $p \colon R \to P$ preserve limits of finite hyperopcovers. This is so since $\O_R^\can$ is then a sheaf and $R \to H^0 K$ is thus an isomorphism; since $P \to H^0 \widetilde K$ is still an iso, the leftmost square in the proof is cocartesian trivially.

Similarly, there is the following converse to Corollary~\ref{cor:good_objects_two}. If  $\ell_R$ is an isomorphism for all $R$, then the hypothesis of Corollary~\ref{cor:good_objects_two} on the factorization of distinguished opcovers holds. This is so because $R \to H^0 K$ is an iso, as above, and thus the condition becomes trivial.

In both cases, in order for $\ell_R$ to be a geometric isomorphism, the fixed points must form a reflective subcategory where the reverse implication holds, so it is a good idea to look for such objects, i.e.\ objects whose local forms are flat w.r.t.\ finite hypercovers (this may be a bit complicated) or objects for which limit cones are opcovering.
\end{remark}

\begin{remark}
We will now give a third instance of idempotency of the adjunction $\Gamma \dashv \Spec$ which, however, does not apply to any of our standard examples. If monos are pushout stable and non-empty products are couniversal we get an (epi, regular mono) factorization in the product completion. Now consider a distinguished hypercover of $\Spec R$, given by a semi-cosimplicial object
\[\xymatrix{
R \ar[r] & K^0 \ar@<.5ex>[r] \ar@<-.5ex>[r] & K^1 \ar@<1ex>[r] \ar@<0ex>[r] \ar@<-1ex>[r] & \cdots
}\]
in the product completion, where the latching maps are opcovering collections of finite localizations. Thanks to $\mcA = \mono\mcA$, such collections are jointly monic and the dual of Theorem~\ref{thm:epi_mono_factorization} gives that the induced map $R \to H^0 K^\bullet$ is epic. Since a filtered colimit of epis is epic we get that the map from the canonical presheaf to its sheafification is objectwise epic; this is the counit of the adjunction $\varepsilon_R \colon R \to \Gamma \Spec R$. Applied to $R = \Gamma X$, it follows from the triangle identity
\[1 \colon \Gamma X \to \Gamma \Spec \Gamma X \to \Gamma X\]
that the counit $\varepsilon_{\Gamma X}$ is an isomorphism; Theorem~\ref{theorem:idempotent_adjunction} then gives idempotency.

As usual, for compact $\Spec R$, finite hypercovers are sufficient, so that finite product completion is sufficient and we only need binary products couniversal.
\end{remark}

\newpage
\part{Functor of points approach}
\section{Generalities on small sheaves}

The last goal of this paper is to explore what algebraic geometers call \emph{the functor of points approach}. In the classical case where $\mcA$ is the category of commutative rings and $\mcP$ is the subcategory of local rings and local homomorphisms, instead of a scheme $X$ in $\APTop$, we can consider the functor $NX:=\APTop(\Spec -, X) \colon \mcA^{\op}\to \Set$. This is a sheaf in the Zariski topology $J$ on $\mcA^{\op}$ and it turns out that the functor $N\colon \APTop \to \SH{J}{\mcA^\op}$ valued in $J$-sheaves is fully faithful when restricted to schemes. This fact is crucial since it enables us to view schemes as a full subcategory of the category of Zariski sheaves which is much better behaved from the categorical standpoint than the category $\APTop$.

In the rest of the paper, we will generalize this to our context. Having a set of cones $C$ in $\mcA$ as before, we can consider the Grothendieck topology $J$ on $\mcA^{\op}$ defined as follows: $J$-covering families correspond to distinguished opcovers $(R\to K_i)$ in $\mcA$. Now, we want to consider the category of $J$-sheaves on $\mcA^{\op}$ -- but as we will see later, we need to handle some issues coming from the fact that $\mcA$ need not be small. In particular, working with sheaves on a large category, it is in general not clear whether the left adjoint to the inclusion of sheaves into presheaves exists. Furthermore we would wish to consider the left Kan extension of $\Spec$ along the Yoneda embedding (composed with sheafification) and we want this to be cocontinuous, but this requires some care if $\mcA$ is large.



\subsection{Presheaves and sheaves over large categeories}

The material of this section will be applied to presheaves and sheaves on $\mcA^\op$, but in order to decrease the number of $\op$'s, and also with a view towards applications outside of this paper, we will be considering presheaves on a category $\mcC$. It will be important that $\mcC$ is not assumed small. In this case, we assume the axiom of universes and consider a universe $\SET$ with respect to which $\mcC$ is small. In addition, we require that the normal universe $\Set$ consists of $\kappa$-presentable objects in $\SET$ for some inaccessible cardinal $\kappa$. Then $\Set$ is closed in $\SET$ under all small (i.e.\ $\kappa$-small) limits and colimits and, in particular, these are computed the same way in the two categories. Since $\mcC$ is small w.r.t.\ $\SET$, the usual theory shows that $\PSH{\mcC} = [\mcC^\op, \SET]$ is a free completion of $\mcC$ under ``large'' colimits. Also, for any Grothendieck topology $J$ on $\mcC$, there exists a sheafification functor $\sheafify$ on $\PSH{\mcC}$, a reflection onto the full subcategory $\SH{J}{\mcC}$ of $J$-sheaves. The sheafification is given locally by a formula
\[(\sheafify F) X = \colim\limits_{U_\bullet \to X} \lim FU_\bullet\]
where the colimit runs over all $J$-hypercovers $U_\bullet \to X$.
Finally, we need to consider the full subcategory $\PSh{\mcC} \subseteq \PSH{\mcC}$ of small presheaves, i.e.\ the closure of the image $\yoneda \mcC$ of the Yoneda embedding under small colimits in $\PSH{\mcC}$; most importantly, $\PSh{\mcC}$ is the free completion of $\mcC$ under small colimits, as is easy to be seen (or apply Theorem~\ref{thm:universal_property_sheaves} to the trivial topology). Similarly, the full subcategory $\Sh{J}{\mcC} \subseteq \SH{J}{\mcC}$ of small $J$-sheaves is the closure of the image $\yonedash \mcC$ of the ``sheafified Yoneda embedding'' $\yonedash = \sheafify \yoneda$ under small colimits in $\SH{J}{\mcC}$.\footnote{While $\PSh{\mcC}$ is easily seen to consist precisely of small colimits of representables, the corresponding claim for sheaves, of being precisely small colimits of sheafified representables, is much harder but still true, see \cite{Shulman}.} We will now show that it has a crucial universal property.

\begin{theorem} \label{thm:universal_property_sheaves}
For any cocomplete category $\mcD$, there exists an adjoint equivalence, given by the left Kan extension and restriction
\[[\mcC, \mcD]_J \simeq [\Sh{J}{\mcC}, \mcD]_\mathrm{ccts},\]
between functors $K \colon \mcC \to \mcD$ that turn $J$-hypercovers to colimit cocones and cocontinuous functors $L \colon \Sh{J}{\mcC} \to \mcD$.
\end{theorem}

By definition, a $J$-hypercover of $X$ is a semi-simplicial object
\[\xymatrix{
\cdots \ar@<1ex>[r] \ar@<0ex>[r] \ar@<-1ex>[r] & U_1 \ar@<.5ex>[r] \ar@<-.5ex>[r] & U_0 \ar[r] & X
}\]
in the coproduct completion $\coprod \mcC$, augmented by $X$, such that all the matching maps are $J$-covers. A functor $K$ turns this into a colimit cocone if the canonical map $H_0 KU_\bullet \to KX$ is an isomorphism.

We will now compare the condition from the theorem for $J$-hypercovers with that for $J$-covers which might be easier to check; here a $J$-cover is interpreted as a particular $J$-hypercover -- one where $U_1 = U_0 \times_X U_0$, etc.

\begin{lemma} \label{lemma:hypercovers}
If a functor $K\colon \C\to \D$ sends $J$-covers to colimit cocones, then it also sends $J$-hypercovers to colimit cocones.
\end{lemma}

\begin{proof}
We will make use of $H_0 KU_\bullet \cong \colim \Sigma KU_\bullet$ where $\Sigma \colon \coprod \mcD \to \mcD$ is the functor from Section~\ref{sec:covers} (taking coproducts of the components). The hypothesis then guarantees that, for any $J$-cover $U_0$ of $X$, the induced map $\Sigma KU_0 \to KX$ is a coequalizer, hence an epimorphism.

Now consider an arbitrary $J$-hypercover $U_\bullet \to X$, compare it with the underlying $J$-cover $U_0$, apply $K$ and take coproducts:
\[\xymatrix{
\cdots \ar@<1ex>[r] \ar@<0ex>[r] \ar@<-1ex>[r] & \Sigma KU_1 \ar@<.5ex>[r] \ar@<-.5ex>[r] \ar[d] & \Sigma KU_0 \ar[r] \ar@{=}[d] & KX \ar@{=}[d] \\
\cdots \ar@<1ex>[r] \ar@<0ex>[r] \ar@<-1ex>[r] & \Sigma K(U_0 \times_X U_0) \ar@<.5ex>[r] \ar@<-.5ex>[r] & \Sigma KU_0 \ar[r] & KX \\
}\]
The map $\Sigma KU_1 \to \Sigma K(U_0 \times_X U_0)$ is a coproduct of epimorphisms by the first paragraph, and thus itself an epimorphism. By assumption, the lower row is a coequalizer, hence so is the upper one.
\end{proof}

\subsection{Weighted colimits}

It will be convenient to consider colimits weighted by possibly large weights -- these may well exist, e.g.\ a coproduct of an arbitrary number of copies of an initial object. Thus, a weight is just $F \in \PSH{\mcC}$. For a functor $K \in [\mcC, \mcD]$ the weighted colimit $F *_\mcC K$ is a representing object of the covariant functor on the left
\[\PSH{\mcC}(F, \mcD(K, -)) \cong \mcD(F *_\mcC K, -),\]
should it exist. By Yoneda lemma, $\yoneda X *_\mcC K \cong KX$. We have the following important cocontinuity property of the weighted colimit:

\begin{lemma} \label{lemma:cocontinuity_weighted_colimit}
Let $F \colon \mcI \to \PSH{\mcC}$ be a diagram of weights such that for each $i \in \mcI$ the weighted colimit $F_i *_\mcC K$ exists. Then
\[\colim_{i \in \mcI} (F_i *_\mcC K) \cong (\colim_{i \in \mcI} F_i) *_\mcC K\]
in the sense that one side exists iff the other side exists, in which case the canonical comparison map is an isomorphism.
\end{lemma}

\begin{proof}
They are, respectively, representing objects for the functors
\[\lim_{i \in \mcI} \PSH{\mcC}(F_i *_\mcC K, -) \cong \lim_{i \in \mcI} \PSH{\mcC}(F_i, \mcD(K, -)) \cong \PSH{\mcC}(\colim_{i \in \mcI} F_i, \mcD(K, -))\]
that are always isomorphic.
\end{proof}

One of the applications of weighted colimits is a pointwise formula for the left Kan extension: Given functors $K \colon \mcC \to \mcD$ and $I \colon \mcC \to \mcC'$, we have
\[(\lan_I K)X' \cong \mcC'(I-, X') *_\mcC K;\]
more precisely, if the weighted colimit exists then so does the left Kan extension and it is then given by this formula.

\subsection{Proof of the main theorem}

The proof revolves around left Kan extensions of a functor $K \colon \mcC \to \mcD$ along the Yoneda embedding to various full subcategories of presheaves. There is a pointwise formula for the left Kan extension at $F \in \PSH{\mcC}$ in terms of a weighted colimit
\[(\lan_{\yoneda} K)F = \PSH{\mcC}(\yoneda-, F) *_\mcC K \cong F *_\mcC K,\]
where the simplifying isomorphism is the Yoneda lemma. The weights for which this colimit exists form a full subcategory $\widehat\mcC_K$ and a pointwise left Kan extension of $K$ along the restricted Yoneda embedding $\yoneda \colon \mcC \to \widehat\mcC_K$ then exists and is given by the weighted colimit as above. For simplicity, we will denote it as $\widehat K = \lan_{\yoneda} K$:
\[\xymatrix{
\mcC \ar[r]^-K \ar[d]_-\yoneda & \mcD \\
\widehat\mcC_K \ar@{-->}[ru]_-{\widehat K}
}\]

Although the theorem concerns the pointwise left Kan extension along the composite $\yonedash = \sheafify \yoneda$, the isomorphism
\[(\lan_{\yonedash} K) F = \Sh{J}{\mcC}(\yonedash -, F) *_\mcC K \cong \PSH{\mcC}(\yoneda -, F) *_\mcC K \cong F *_\mcC K \cong \widehat K F\]
shows that this is equally a pointwise left Kan extension along the restricted Yoneda embedding $\yoneda \colon \mcC \to \Sh{J}{\mcC}$, so that an important step will be to show $\Sh{J}{\mcC} \subseteq \widehat\mcC_K$:

\begin{lemma}
$\widehat\mcC_K$ is closed under the following constructions:
\begin{itemize}
\item
    $\widehat\mcC_K$ contains all representables.
\item
    If $\mcD$ is (small) cocomplete then $\widehat\mcC_K$ is closed under small colimits and $\widehat K$ is cocontinuous.
\item
    If $K$ turns $J$-hypercovers to colimit cocones then $F \in \widehat\mcC_K$ $\Ra$ $\sheafify F \in \widehat\mcC_K$; moreover, the image $\widehat K F \to \widehat K \sheafify F$ of the sheafification map is an isomorphism.
\end{itemize}
\end{lemma}

\begin{proof}
The first point is obvious and the second is just Lemma~\ref{lemma:cocontinuity_weighted_colimit} so we proceed with the third. In order to apply Lemma~\ref{lemma:cocontinuity_weighted_colimit} again, we need to express the sheafification as a colimit in the category of presheaves. $J$-sheaves are exactly the objects orthogonal to the collection of maps $\alpha^0_{U_\bullet} \colon H_0 \yoneda U_\bullet \to y X$ coming from hypercovers $U_\bullet \to X$ or equivalently objects injective w.r.t.\ the $\alpha^0_{U\bullet}$ and the corresponding codiagonal maps from the cokernel pair $\alpha^1_{U_\bullet} \colon \yoneda X +_{H_0 \yoneda U_\bullet} \yoneda X \to \yoneda X$. To simplify the notation, we write $\alpha^n_{U_\bullet} \colon G^n_{U_\bullet} \to H^n_{U_\bullet}$.

The main idea now is that $\sheafify F$ is built from $F$ by gluing cells of the shape $\alpha^n_{U_\bullet}$, which we make more precise in the next paragraph, and these induce isomorphisms on weighted colimits: for $\alpha^0_{U_\bullet}$ this is just the condition from the statement
\[H_0 yU_\bullet *_\mcC K \cong H_0 KU_\bullet \to KX = yX *_\mcC K,\]
and it easily implies the corresponding claim for $\alpha^1_{U_\bullet}$. Consequently, the sheafification map $F \to \sheafify F$ also induces an isomorphism on weighted colimits.

The sheafification $\sheafify F$, being the reflection of $F$ onto this injectivity class, can be produced by the small object argument, since the domains $G^n_{U_\bullet}$ of $\alpha^n_{U_\bullet}$ are $\kappa$-presentable in $\PSH{\mcC}$. There results a smooth chain $F_\lambda$, for $\lambda \leq \kappa$, with $F_0 = F$ and $F_\kappa = \sheafify F$. We prove by transfinite induction that $F_\lambda \in \widehat\mcC_K$ and that the chain $F_\mu *_\mcC K$ consists of isomorphisms, for $\mu \leq \lambda$. For a limit $\lambda$, this is straightforward by Lemma~\ref{lemma:cocontinuity_weighted_colimit}. For a successor $\lambda$, we have a pushout square
\[\xymatrix{
\coprod G^{n_i}_{U_{\bullet,i}} \ar[r]^-{\coprod \alpha^{n_i}_{U_{\bullet,i}}} \ar[d] & \coprod H^{n_i}_{U_{\bullet,i}} \ar[d] & \ar[rd] \ar@/^1em/[rrr] & G^{n_i}_{U_{\bullet,i}} \ar@{.}[l] \ar@{.}[r] \ar[d] \ar@/^1em/[rrr]^-{\alpha^{n_i}_{U_{\bullet,i}}} & \ar[ld] \ar@/^1em/[rrr] & \ar[rd] & H^{n_i}_{U_{\bullet,i}} \ar@{.}[l] \ar@{.}[r] \ar[d] & \ar[ld] \\
F_{\lambda-1} \ar[r] & F_{\lambda} \po & & F_{\lambda-1} \ar[rrr] & & & F_\lambda
}\]
It is easy to rewrite this as $F_\lambda$ being a colimit of a large diagram with objects $F_{\lambda-1}$, $G^{n_i}_{U_{\bullet,i}}$ and $H^{n_i}_{U_{\bullet,i}}$, so that Lemma~\ref{lemma:cocontinuity_weighted_colimit} can be applied again, giving that $F_\lambda *_\mcC K$ is a colimit of a diagram consisting of the $F_{\lambda-1} *_\mcC K$ and isomorphisms; therefore this colimit exists and is isomorphic to $F_{\lambda-1} *_\mcC K$, as required.
\end{proof}

\begin{corollary}
If $\mcD$ is (small) cocomplete and $K$ turns $J$-hypercovers to colimit cocones then $\Sh{J}{\mcC} \subseteq \widehat \mcC_K$.
\end{corollary}

\begin{proof}
A colimit in $\SH{J}{\mcC}$ is the sheafification of the colimit in $\PSH{\mcC}$; thus, the previous lemma shows that $\widehat \mcC_K \cap \SH{J}{\mcC}$ is closed under colimits in $\SH{J}{\mcC}$. Since it also contains sheafified representables $\yonedash \mcC$, it must contain $\Sh{J}{\mcC}$.
\end{proof}

We are now in the position to prove the main theorem of this section:

\begin{proof}[Proof of Theorem~\ref{thm:universal_property_sheaves}]
We have seen that a left Kan extension along $\yonedash$ is a restriction of the left Kan extension along $\yoneda \colon \mcC \to \widehat\mcC_K$ and the previous corollary shows its existence under the assumption that $K$ turns $J$-hypercovers to colimit cocones. Clearly the left Kan extension is left adjoint to the precomposition functor; now for $K \in [\mcC, \mcD]$ and $L \in [\Sh{J}{\mcC}, \mcD]$ we need to verify:
\begin{itemize}
\item
    If $K$ turns $J$-hypercovers to colimit cocones, then $\widehat K$ is cocontinuous.
\item
    If $K$ turns $J$-hypercovers to colimit cocones, then $K = \widehat K \yonedash$.
\item
    If $L$ is cocontinuous, then $L \yonedash$ turns $J$-hypercovers to colimit cocones.
\item
    If $L$ is cocontinuous, then $L = \widehat{L \yonedash}$.
\end{itemize}

The first point follows from the previous lemma, since a colimit in $\Sh{J}{\mcC}$ is computed as a sheafification of a colimit in $\PSH{\mcC}$ and $\widehat K$ preserves colimits of presheaves and turns the sheafification map into an isomorphism:
\[\widehat K (\underbrace{\sheafify \colim_i}_{\makebox[0pt]{\scriptsize colimit in $\Sh{J}{\mcC}$}} F_i) \cong \widehat K(\colim_i F_i) \cong \colim_i \widehat K F_i.\]
For the second point, we have $\widehat K \yonedash \cong \widehat K \yoneda \cong K$, since the Yoneda embedding is fully faithful.

For the third point, the map
\[H_0 L \yonedash U_\bullet \lra L \yonedash X\]
is isomorphic to the $L$-image of $\sheafify \alpha^0_{U_\bullet} \colon \sheafify H_0 \yoneda U_\bullet \lra \sheafify \yoneda X$, which is an isomorphism (as follows e.g.\ from the analogous well known fact in $\PSH{\mcC}$). The fourth point follows since the collection of $J$-sheaves, at which the canonical map $\widehat{L \yonedash} \to L$ is an isomorphism, is clearly closed under colimits and contains all sheafified representables, by the iso in the proof of the second point (applied to $K = L \yonedash$).
\end{proof}

Next we promote the extension $\widehat K \colon \Sh{J}{\mcC} \to \mcD$ to an adjunction. Namely, consider the ``nerve functor'' $N_K \colon \mcD \to \PSH{\mcC}$, given as $N_K D = \mcD(K -, D)$.

\begin{proposition} \label{prop:nerve_realization_adjunction}
If $\mcD$ is (small) cocomplete and $K$ turns $J$-hypercovers to colimit cocones then the nerve functor takes values in $\SH{J}{\mcC}$ and there results a partial adjunction $\widehat K \dashv N_K$.

If, in addition, $K$ is fully faithful then $\yonedash = \yoneda$, i.e.\ the topology $J$ is subcanonical. Moreover, $\yoneda X \in \Sh{J}{\mcC}$ and $KX \in \mcD$ constitute a fixed point of this adjunction, for any $X \in \mcC$; in particular, $N_K$ maps the image of $K$ to $\Sh{J}{\mcC}$.
\end{proposition}

\begin{proof}
For the first part, the nerve functor is the composite
\[N_K D = \mcD(K -, D)\]
where $K$ takes any $J$-hypercover to a colimit cocone and $\mcD(-, D)$ turns this into a limit cone, as required for a presheaf to be a $J$-sheaf. Then
\[\mcD(\widehat K F, D) \cong \mcD(F *_\mcC K, D) \cong \PSH{\mcC}(F-, \mcD(K-, D)) \cong \SH{J}{\mcC}(F, N_K D).\]

The second part is easy:
\[N_K (KX) = \mcD(K -, KX) \cong \mcC(-, X) = \yoneda X\]
which must then be a sheaf by the first part, and thus equals $\yonedash X$. The other direction $\widehat K \yoneda X = \yoneda X *_\mcC K \cong K X$ is just Yoneda lemma.
\end{proof}

%

We finish with proving a sufficient condition for local smallness of $\Sh{J}{\mcC}$.

\begin{proposition}
Assume that, for any fixed $X$, the $J$-hypercovers of $X$ up to isomorphism form a (small) set. Then $\Sh{J}{\mcC}$ is locally small.
\end{proposition}

\begin{proof}
Consider the collection of all $J$-sheaves $F \in \Sh{J}{\mcC}$ for which the corresponding hom-functor $\Sh{J}{\mcC}(F, -)$ takes values in small sets. Our aim is thus to show that this collection equals $\Sh{J}{\mcC}$. It is clear that it is closed under small colimits and it thus remains to show that it contains sheafified representables. Since we have already seen that
\[\Sh{J}{\mcC}(\yonedash X, G) \cong GX,\]
this in turn amounts to showing that $\Sh{J}{\mcC} \subseteq [\mcC^\op, \Set]$.

The formula for the sheafification shows that $\sheafify$ restricts to an endofunctor on $[\mcC^\op, \Set]$ and thus provides a reflection onto a full subcategory $\mcS$. Since a colimit in $\mcS$ is computed in the same way as in $\SH{J}{\mcC}$, i.e.\ as the sheafification of the colimit in $\PSH{\mcC}$, the subcategory $\mcS$ is closed in $\SH{J}{\mcC}$ under small colimits and thus $\Sh{J}{\mcC} \subseteq \mcS \subseteq [\mcC^\op, \Set]$.
\end{proof}

\section{Schemes and functors}

Throughout this section, we assume that $\mcA = \fix\mcA$, i.e.\ that $\Spec$ is fully faithful. By Lemma~\ref{lemma:Spec_ff_implies_reduced}, this implies $\mcA = \red\mcA$.

Here we add an assumption on $\mcA$: We demand that there exists a (regular epi, mono) factorization system both in $\mcA$ and in the coproduct completion $\coprod \mcA$. This is the case e.g.\ when colimits are universal in $\mcA$, see Theorem~\ref{theorem:factorization_from_universal_colimits}.

We equip $\mcA^\op$ with a Grothendieck topology $J$ given by the distinguished opcovers in $\mcA$. Since $\Spec$ takes distinguished $J$-hyperopcovers to colimit cocones, Proposition~\ref{prop:nerve_realization_adjunction} gives a partial adjunction which we denote for simplicity by $|\ |$ and $N$:
\[\xymatrix{
& \mcA^\op \ar[ld]_-{\yoneda = \yonedash} \ar[rd]^-{\Spec} \\
\Sh{J}{\mcA^\op} \ar@<.5ex>[rr]^-{|\ |} \ar@{c->}[d] & {} \POS[];[ld]**{}?(.33)*{\perp} & \APTop \ar@<.5ex>[lld]^-N \\
\SH{J}{\mcA^\op}
}\]
In addition, since we assume that $\Spec$ is fully faithful, $\yonedash = \yoneda$ and $\yoneda R \in \Sh{J}{\mcA^\op}$ and $\Spec R \in \APTop$ provide a fixed point, for each $R \in \mcA$. It is easy to see that in any adjunction, the right adjoint is fully faithful on maps from a fixed point, giving the following lemma:

\begin{lemma} \label{lemma:N_fully_faithful_from_affine_schemes}
$N$ is fully faithful on maps from affine schemes, i.e.
\[\pushQED{\qed}N \colon \APTop(\Spec R, X) \xlra\cong \Sh{J}{\mcA^\op}(\underbrace{N \Spec R}_{\yoneda R}, NX).\qedhere\popQED\]
\end{lemma}

Clearly the fixed points of this adjunction are closed in $\APTop$ under all colimits that the right adjoint $N$ preserves. Since our goal is to show that all schemes are fixed points, this will be our main theorem for this section:

\begin{theorem}
The nerve functor $N$ preserves colimits of hypercovers. Consequently, $N$ is fully faithful on schemes.
\end{theorem}

\begin{proof}
Let $U_\bullet \to X$ be a hypercover, i.e.\ an augmented semisimplicial object
\[\xymatrix{
\cdots \ar@<1ex>[r] \ar@<0ex>[r] \ar@<-1ex>[r] & U_1 \ar@<.5ex>[r] \ar@<-.5ex>[r] & U_0 \ar[r] & X
}\]
in the coproduct completion of $\APTop$. Our aim is to show that the map $H_0 NU_\bullet \to NX$ is an isomorphism in $\Sh{J}{\mcA^\op}$,
\[\xymatrix{
\cdots \ar@<1ex>[r] \ar@<0ex>[r] \ar@<-1ex>[r] & NU_1 \ar@<.5ex>[r] \ar@<-.5ex>[r] & NU_0 \ar[r] & H_0 NU_\bullet \ar[r] & NX \\
& & & & \yoneda R \ar[u] \ar@{-->}[lu]^-{\exists!}
}\]
i.e.\ we need to show that $\Sh{J}{\mcA^\op}(\yoneda R, H_0 NU_\bullet) \to \Sh{J}{\mcA^\op}(\yoneda R, NX)$ is an isomorphism for every $R \in \mcA$ (since $\yoneda \mcA^\op$ is a generator).

First we prove surjectivity by finding a lift of an arbitrary $\yoneda R \to NX$. Such a map lies in the image of $N$ by the previous lemma. Pull back the hypercover $U_\bullet \to X$ to a hypercover $V_\bullet \to \Spec R$ and apply $N$ to obtain the top part of the following diagram
and choose a distinguished refinement $\Spec L^\bullet$ of $V_\bullet$ to obtain the bottom part.
\[\xymatrix{
\cdots \ar@<1ex>[r] \ar@<0ex>[r] \ar@<-1ex>[r] & NU_1 \ar@<.5ex>[r] \ar@<-.5ex>[r] & NU_0 \ar[r] & H_0 NU_\bullet \ar[r] & NX \\
\cdots \ar@<1ex>[r] \ar@<0ex>[r] \ar@<-1ex>[r] & NV_1 \ar@<.5ex>[r] \ar@<-.5ex>[r] \ar[u] & NV_0 \ar[r] \ar[u] & H_0 NV_\bullet \ar[r] \ar[u] & N \Spec R \ar[u] \\
\cdots \ar@<1ex>[r] \ar@<0ex>[r] \ar@<-1ex>[r] & \yoneda L^1 \ar@<.5ex>[r] \ar@<-.5ex>[r] \ar[u] & \yoneda L^0 \ar[r] \ar[u] & H_0 \yoneda L^\bullet \ar[r]_-\cong \ar[u] & \yoneda R \ar@{=}[u]
}\]
The map $H_0 \yoneda L^\bullet \to \yoneda R$ is an isomorphism in $\Sh{J}{\mcA^\op}$ and we thus obtain a lift as the composition $\yoneda R \xla\cong H_0 \yoneda L^\bullet \to H_0 NU_\bullet$.

The first part shows, in particular, that any hypercover yields a jointly epic family on nerves. Thus, the hypothesis of Theorem~\ref{thm:epi_mono_factorization} is satisfied for $NU \to NX$ and we may conclude that the map $H_0 NU_\bullet \to NX$ is monic. Together with the first part, it is thus an isomorphism.
\end{proof}

We finish by giving a more explicit description of the fixed points in $\Sh{J}{\mcA^\op}$, closely connected to the definition of schemes in $\APTop$. We will thus define open covers in $\Sh{J}{\mcA^\op}$ and schemes will then be objects admitting an open cover by representables. We define open subfunctors of $\yoneda R$ to be exactly the nerves of open subspaces of $\Spec R$.
Thus, let $U \subseteq \Spec R$ be an open subset and define a subfunctor of $\yoneda R = \mcA(R, -)$ by
\[\yoneda R_U (S) = \biggl\{f \colon R \to S \mathrel{\bigg|} \parbox{\widthof{$qf \colon R \to S \to Q$ factors through some $P \in U$}}{for each local form $q \colon S \to Q$, the composite \\ $qf \colon R \to S \to Q$ factors through some $P \in U$}\biggr\}\]

\begin{remark}
In the case that $U = \Pts k$ is a distinguished open then this is equivalently the representable associated with the localization $K = R_U$ so that the notation is consistent.
\end{remark}

An \emph{open embedding} is a monic $\alpha \colon G \to F$ such that for any map from a representable $\yoneda R$, the resulting pullback of $\alpha$ is an inclusion of some open subfunctor $\yoneda R_U \subseteq \yoneda R$:
\[\xymatrix{
G \ar[r]^-\alpha & F \\
\yoneda R_U \ar[r] \ar[u] \pbd & \yoneda R \ar[u]
}\]
An \emph{open cover} is a collection of open embeddings $G_i \to F$ such that for each local object $P \in \mcP$ the components $G_i(P) \to F(P)$ are jointly surjective (the argument of Lemma~\ref{lemma:surjectivity_of_sheaves_from_local_objects} implies that $\coprod G_i \to F$ is then epic in $\Sh{J}{\mcA^\op}$). Finally, $F$ is said to be a \emph{scheme}, if it admits an open cover by representables.

\begin{lemma}
The nerve of every scheme in $\APTop$ is a scheme in $\Sh{J}{\mcA^\op}$.
\end{lemma}

\begin{proof}
Since $N$ is fully faithful on maps from affine schemes by Lemma~\ref{lemma:N_fully_faithful_from_affine_schemes} and preserves pullbacks as a right adjoint, it is easy to see that any open cover $U \to X$ yields a collection of open subfunctors $NU^i \to NX$.
\[\xymatrix{
NU^i \ar[r] & NX \\
\rightbox{\yoneda R_{\varphi^{-1}U^i} = {}}{N(\varphi^{-1} U^i)} \ar[r] \ar[u] \pbd & \leftbox{N\Spec R}{{} = \yoneda R} \ar[u]_-{N\varphi}
}\]
The joint surjectivity means that any map $\varphi \colon \Spec P \to X$ factors through one of the $U^i$. But since the preimages $\varphi^{-1}U^i$ form an open cover of $\Spec P$, one of them must equal $\Spec P$, since this is the only open containing the point $1_P$.
\end{proof}

\begin{theorem} \label{thm:scheme_functors}
Every scheme $F \in \Sh{J}{\mcA^\op}$ is a fixed point, i.e.\ $F \cong N|F|$.
\end{theorem}

\begin{corollary} \label{corollary:equivalence_of_schemes}
The adjunction $|\ | \dashv N$ restricts to an adjoint equivalence between the respective categories of schemes.\qed
\end{corollary}

\begin{proof}[Proof of Theorem~\ref{thm:scheme_functors}]
Let $\yoneda R^0_i \to F$ be an open cover. Since the pullbacks $\yoneda R^0_{i_0} \times_F \yoneda R^0_{i_1}$ are open subfunctors of both $\yoneda R^0_{i_0}$ and $\yoneda R^0_{i_1}$, Theorem~\ref{thm:affine_communication_lemma} together with the previous lemma produce a distinguished open cover $\yoneda R^1_j$ and we thus obtain an augmented semisimplicial object
\[\xymatrix{
\cdots \ar@<1ex>[r] \ar@<0ex>[r] \ar@<-1ex>[r] & \yoneda R^1 \ar@<.5ex>[r] \ar@<-.5ex>[r] & \yoneda R^0 \ar[r] & F
}\]
whose higher matching maps are distinguished open covers and in particular are jointly epic in $\Sh{J}{\mcA^\op}$. Theorem~\ref{thm:epi_mono_factorization} then shows that $H_0 \yoneda R^\bullet \to F$ is monic and it is easily seen to be an open embedding:
\[\xymatrix{
\cdots \ar@<1ex>[r] \ar@<0ex>[r] \ar@<-1ex>[r] & \yoneda R^1 \ar@<.5ex>[r] \ar@<-.5ex>[r] & \yoneda R^0 \ar[r] & H_0 \yoneda R^\bullet \ar[r] & F \\
\cdots \ar@<1ex>[r] \ar@<0ex>[r] \ar@<-1ex>[r] & \yoneda S_{U_1} \ar@<.5ex>[r] \ar@<-.5ex>[r] \ar[u] \pbd & \yoneda S_{U_0} \ar[r] \ar[u] \pbd & H_0 \yoneda S_{U_\bullet} \ar[r] \ar[u] \pbd & \yoneda S \ar[u]
}\]
Namely, $\yoneda S_{U_\bullet}$ consists of nerves of open subspaces of $\Spec S$ with higher matching maps covers, and since $N$ preserves such colimits, its colimit is indeed an open subfunctor of $\yoneda S$. The proof is finished by an application of the following lemma, applied to $H_0yR^\bullet \to F$.
\end{proof}

\begin{lemma} \label{lemma:surjectivity_of_sheaves_from_local_objects}
An open embedding $\alpha \colon G \to F$ with surjective components on local objects is an isomorphism.
\end{lemma}

\begin{proof}
Since $\alpha$ is monic, we need only show that $G(R) \to F(R)$ is surjective for every $R \in \mcA$. Thus, take any $\yoneda R \to F$ and form a pullback
\[\xymatrix{
G \ar[r] & F \\
\yoneda R_U \ar[r] \ar[u] \pbd & \yoneda R \ar[u]
}\]
Clearly, the components $\yoneda R_U(P) \to \yoneda R(P)$ are also surjective and it follows that the open subspace $U \subseteq \Spec R$ contains all points and this inclusion is in fact an isomorphism. A lift is thus obtained as $\yoneda R \xla\cong \yoneda R_U \to G$.
\end{proof}

\newpage
\part*{Appendices}
\setcounter{section}{0}
\renewcommand{\thesection}{\Alph{section}}

\section{Grothendieck constructions and their completeness} \label{section:Grothendieck_constructions}

For a pseudofunctor $R \colon \mcB^\op \to \CAT$ valued in (possibly large) categories, we consider the associated fibration
\[P \colon \smallint R \to \mcB,\]
also called the Grothendieck construction of $R$. The following is well-known:

\begin{theorem}
If $\mcB$ is complete and $R$ is valued in complete categories and continuous functors then $\smallint R$ is also complete and the canonical projection $P$ is also continuous.

More precisely, for a diagram $F \colon \mcK \to \smallint R$ consider the universal cone $\lambda_k \colon \lim PF \to PFk$ in $\mcB$ and denote the cartesian lifts of its components as $\lambda_k^* Fk \to Fk$. Then the $\lambda_k^* Fk$'s form a diagram in the fibre over $\lim PF$ whose limit equals $\lim F$.
\end{theorem}

Dually, any pseudofunctor $L \colon \mcB \to \CAT$ induces an opfibration $P \colon \smallint L \to \mcB$. Since this is the opposite of the first Grothendieck construction, the previous theorem now holds with limits replaced by colimits.

Finally, we will discuss bifibrations and their bicompleteness. A fibration $\smallint R \to \mcB$ is easily seen to be a bifibration iff the diagram $R \colon \mcB^\op \to \CAT$ is valued in categories and right adjoint functors. If $L \colon \mcB \to \CAT$ is the diagram consisting of the same categories but with all functors replaced by their left adjoints, one observes that $\smallint L \simeq \smallint R$ over $\mcB$ and we thus obtain the following theorem. In this situation we will say that the bifibration is associated to the diagram $A \colon \mcB \to \CAT_\textrm{ad}$ of adjunctions $L \dashv R$.

\begin{theorem}
If $\mcB$ is bicomplete and $A$ is a diagram of bicomplete categories and adjunctions between them then $\smallint A$ is also bicomplete and the canonical projection $P$ is bicontinuous.\qed
\end{theorem}

We will now describe two examples that are relevant for this paper.

\begin{example} \label{example:Top}
Consider the diagram $A \colon \Set \to \CAT_\textrm{ad}$ associating to each set $X$ the poset $AX = \operatorname{Top}_X^\op$ of all topologies on $X$ ordered by the reverse inclusion; the adjunction associated to a mapping $f \colon X \to Y$ consists of the direct and inverse image $f_* \dashv f^*$, where the direct image $f_* \mathfrak X$ is the topology generated by $\{V \subseteq Y \mid f^{-1}V \in \mathfrak X\}$. The associated bifibration is then easily seen to be the category $\Top$ of topological spaces and continuous maps. The theorems say that limits and colimits in $\Top$ are created from those in $\Set$.

Although colimits of topological spaces are well known and the general theorem above gives the usual description of the topology as the largest topology making the components of the universal cocone continuous, we will make use of an alternative description of open subssets of the colimit. This description starts with the observation that an open subset of $X$ is given equivalently by a continuous map into the Sierpi\'nski space $\Sigma$. Thus, open subsets of $\colim X_k$ are given equivalently by cocones $X_k \to \Sigma$, i.e.\ by collections of open subsets $U_k \subseteq X_k$ that are compatible in the sense: if $\alpha \colon k \to l$ lies in $\mcK$ then $(\alpha_*)^{-1}(U_l) = U_k$; in terms of points this reads $p \in U_k$ iff $\alpha_*p \in U_l$.
\end{example}

\begin{example} \label{ex:ATop}
Consider the diagram $A \colon \Top \to \CAT_\textrm{ad}$ associating to each space $X$ the category $AX = {\Sh{}{X}}^\op$ of $\mcA$-valued sheaves on $X$, or rather its opposite; the adjunction associated to a map $f \colon X \to Y$ consists of the direct and inverse image functors $f_* \dashv f^*$. The associated bifibration is then easily seen to be the category $\ATop$ of $\mcA$-spaces. The theorems say that limits and colimits in $\ATop$ are created from those in $\Top$.

More explicitly, the colimit of a diagram $(X_k, \mcO_{X_k})$ is the colimit $X = \colim X_k$ of the underlying spaces with universal cocone $\lambda_k \colon X_k \to X$ and with the structure sheaf $\mcO_X = \lim (\lambda_k)_* \mcO_{X_k}$.
\end{example}

\begin{example}
For a category $\mcC$, consider the diagram $R \colon \Set^\op \to \CAT$ associating to each set $I$ the category $\mcC^I$, i.e.\ $R = \CAT(-, \mcC)$. When $\mcC$ has coproducts, each functor in the diagram $R$ has a left adjoint given by the left Kan extension, i.e.\ summation along fibres. The associated (bi)fibration is then easily seen to be the free completion $\coprod \mcC$ of $\mcC$ under coproducts. The limits and colimits as given in the previous theorems can then be spelled out explicitly as in the following lemma.
\end{example}

Let $D_k$ be a diagram in $\coprod \mcC$, i.e.\ it consists of a diagram of sets $I_k$ and a collection $(D_k^i)_{i \in I_k}$ of objects together with maps $D_k^i \to D_l^{\alpha(i)}$ for each $\alpha \colon k \to l$ that are closed under composition. This describes a diagram in $\mcC$, indexed by the category of elements associated with $I \colon \mcK \to \Set$ (Grothendieck construction). We denote this diagram by $\El D \colon \El I \to \mcC$, $(k, i) \mapsto D_k^i$. In the other direction, $D = \lan_P \El D$ is the left Kan extension in
\[\xymatrix{
\El I \ar[rr]^-{\El D} \ar[d]_-P & & \mcC \ar@{c->}[d] \\
\mcK \ar@{-->}[rr]_-{\lan_P \El D} & & \coprod \mcC
}\]
along the canonical projection $P \colon \El I \to \mcK$.

\begin{lemma} \label{lem:limits_colimits_coproduct_completion}
The limit of the diagram $D$ in the coproduct completion has indexing set $\lim I$ and, for each $\sigma \in \lim_{k \in \mcK} I_k$, i.e.\ a section of the projection $\El I \to \mcK$, the corresponding component of the limit is
\[H^0 D|_\sigma = \lim_{(k, i) \in \sigma} D_k^i = \lim_{k \in \mcK} D_k^{\sigma(k)}.\]

The colimit of the diagram $D$ in the coproduct completion has indexing set $\colim I$ and, for each $\gamma \in \colim_{k \in \mcK} I_k$, i.e.\ a path component of $\El I$, the corresponding component of the colimit is
\[\pushQED{\qed}H_0 D|_\gamma = \colim_{(k, i) \in \gamma} D_k^i = \colim_{k \in \mcK} \coprod_{i \in \gamma(k)} D_k^i.\qedhere\popQED\]
\end{lemma}

Dually, colimits in the product completion are computed over sections and limits over components. Thus, an object $(A_i)_{i \in I} \in \prod \mcA$ is finitely presentable if it is a finite collection of finitely presentable objects, i.e.\ if $I$ is finite and each $A_i$ is finitely presentable.

Any map $f \colon (A^i)^{i \in I} \to (B^j)^{j \in J}$ can be decomposed into a coproduct of maps into singletons: define $I^j \subseteq I$ to be the inverse image of $j \in J$ and then $f$ is a coproduct of $f^j \colon (A^i)^{i \in I^j} \to (B^j)$ with singleton target.

\begin{corollary} \label{cor:mono_epi_coproduct_completion}
A map $f \colon (A^i)^{i \in I} \to (B^j)^{j \in J}$ in the coproduct completion is mono iff its underlying map of index sets is mono and each summand $f^j$ is mono in $\mcC$. The map $f$ is epi iff its underlying map of index sets is epi and each summand consists of a jointly epi family in $\mcC$. \qed
\end{corollary}

\section{Cone small object argument}

Here we present a more detailed account of the cone small object argument; our aim here is the greatest generality, so we will not assume the existence of all small colimits -- this may prove useful in homotopically oriented applications.

Let us formulate and prove this variation in a general categorical context. In order to substantially simplify the notation in the argument it is useful to consider the free product completion, see the following well-known definition.
\begin{definition}
Let $\mathcal{K}$ be a category. Then the {\it free product completion} $\mathrm{Prod}(\mathcal{K})$ of $\mathcal{K}$ is defined as follows: It has as objects pairs $(X,A)$, where $X$ is a set and $A \colon X \to \mathcal{K}$ is a functor whose domain is viewed as a discrete category. A morphism is a pair $(p,f) \colon (X,A) \to (Y,B)$, where $p \colon Y \to X$ is a function and $f \colon A \cdot p \Rightarrow B$ is a natural transformation, i.e.\ a collection $(f_i \colon A(pi) \to Bi)_{i \in Y}$.
\end{definition}
\begin{remark}
We will sometimes write $(Ax)_{x \in X}$ instead of $(X,A)$.
\end{remark}
Note that one can use the free product completion to derive the definitions of cone injectivity of an object and of cone injectivity of a morphism.\par
The reason why the free product completion is going to simplify the notation is that at each stage in the cone small object argument we are dealing with possibly a collection of objects rather than with a single object as is the case in the usual small object argument, and the free product completion allows us to consider this collection as a single object instead.\par
When trying to perform the cone small object argument in a category $\mathcal{K}$ one might be tempted to simply pass to $\mathrm{Prod}(\mathcal{K})$ and perform the usual small object argument there. However, this gives a different result than we are looking for: The reason is that an injective object $(Ax)_{x \in X}$ of the free product completion can happen to have a non-injective component. For example, consider the category with only four objects $A, B, C, D$ and only two non-identity morphisms $A \to B$, $A \to C$. If we consider injectivity of objects with respect to $A \to B$, then the object $C$ is not injective because there does not exist an extension of $A \to C$ to $B$, whereas the object $(C,D)$ of the free product completion is injective because there does not exist a morphism $A \to (C,D)$.
\begin{remark}
If a category $\mathcal{K}$ has transfinite composites, then its free product completion also has transfinite composites and they're explicitly given as follows: Suppose that we have a diagram $$(X_0, A_0) \xrightarrow{(p_0,f_0)} (X_1, A_1) \xrightarrow{(p_1,f_1)} (X_2, A_2) \xrightarrow{(p_2,f_2)} \cdots$$ whose objects are indexed by ordinals less than some limit ordinal $\lambda$. Let $(q_\mu \colon X \to X_\mu)_{\mu < \lambda}$ be the limit of the diagram $$\cdots \xrightarrow{p_2} X_2 \xrightarrow{p_1} X_1 \xrightarrow{p_0} X_0$$ in $\mathbf{Set}$. Hence, $X$ consists precisely of the elements $(x_\mu)_{\mu < \lambda} \in \prod_{\mu < \lambda} X_\mu$ such that $p_{\mu, \kappa}(x_\mu) = x_\kappa$ for each $p_{\mu, \kappa} \colon X_\mu \to X_\kappa$ from the diagram. Each element $x = (x_\mu)_{\mu < \lambda} \in X$ induces a diagram $$A_0(x_0) \xrightarrow{f_{0,x_1}} A_1(x_1) \xrightarrow{f_{1,x_2}} A_2(x_2) \xrightarrow{f_{2, x_3}} \cdots.$$ Let $(f_{\mu,x} \colon A_\mu(x_\mu) \to A(x))_{\mu < \lambda}$ be the colimit of that diagram in $\mathcal{K}$. This defines a functor $A \colon X \to \mathbf{Set}$. Finally, it is easy to verify that the cocone ${((q_\mu, f_\mu) \colon (X_\mu, A_\mu) \to (X,A))_{\mu < \lambda}}$ is a colimit in $\mathrm{Prod}(\mathcal{K})$.
\end{remark}
The following theorem is already known (under mildly stronger assumptions than in the theorem statement below), see Theorem 2.53 and Proposition 4.16 in \cite{AdamekRosicky}. However, the proof there uses a different approach than the explicit construction that we give below.
\begin{theorem}
Let $\mathcal{K}$ be a category with pushouts and transfinite composites and let $I$ be a set of cones in $\mathcal{K}$ whose domains are $\lambda$-presentable for some regular cardinal $\lambda$. Then the full subcategory {$I$-Inj} of $\mathcal{K}$ consisting of cone injective objects with respect to $I$ is cone-reflective in $\mathcal{K}$. Explicitly, this means that for each object $A$ in $\mathcal{K}$ there exists a cone $(f_i \colon A \to A_i)_{i \in I'}$ such that each $A_i$ is in {$I$-Inj} and for each morphism $g \colon A \to B$ whose codomain $B$ belongs to {$I$-Inj} there exists $i \in I'$ and a morphism $h \colon A_i \to B$ such that $h \cdot f_i = g$.
\end{theorem}
\begin{proof}
We will transfinitely define pairwise disjoint sets $I_\alpha$, where $\alpha \in \lambda$, and a diagram $D \colon \bigcup_{\alpha \in \lambda} I_\alpha \to \mathcal{K}$ such that on $\bigcup_{\alpha \in \lambda} I_\alpha$ there is the obvious partial order such that for $x \in I_\alpha$, $y \in I_\beta$: $x < y$ iff $\alpha < \beta$, $I_0 = \{0\}$, $D0 = A$, each morphism ${D(x \to y)}$, where $x \in I_\alpha$, $y \in I_{\alpha + 1}$, is a transfinite composition of pushouts of morphisms contained in cones in $I$ for each ordinal $\alpha < \lambda$, and $DI_\alpha$ will be a ``fat colimit'' of all $Dx$, where $x \in I_\beta$, $\beta < \alpha$ for each limit ordinal $\alpha < \lambda$. The precise meaning of this imprecisely stated last fact will be clear from the construction.\par
Let $\mathscr{J}$ be the set of all spans
\begin{center}
\begin{tikzcd}
{(1,E)} \arrow[r, "g"] \arrow[d, "f\in I"'] & {(1,A)} \\
{(Z, C)}                                       &        
\end{tikzcd}
\end{center}
in $\mathrm{Prod}(\mathcal{K})$. When interpreted in $\mathcal{K}$, $g$ is a morphism and $f$ is a cone, and we can reinterpret the set $\mathscr{J}$ as $\{(f_{j,z},g_j) \mid j \in J, z \in Z_j\}$ for some indexing set $J$.
Well-order the set $J$, i.e.\ let $J = \delta$, where $\delta$ is an ordinal. Using transfinite induction we will obtain an object $(X_\beta, A_\beta)$ for each ordinal $\beta \leq \delta$. More precisely, we will define a chain
\begin{center}
\begin{tikzcd}
{(X_0,A_0)} \arrow[r, "{(p_1,\overline{f}_1)}"] & {(X_1,A_1)} \arrow[r, "{(p_2,\overline{f}_2)}"] & {(X_2,A_2)} \arrow[r, "{(p_3,\overline{f}_3)}"] & \cdots \arrow[r] & {(X_{\delta}, A_{\delta})}
\end{tikzcd}
\end{center}
in $\mathrm{Prod}(\mathcal{K})$. Such a chain induces for each pair $\beta \leq \beta' \leq \delta$ of ordinals a map $(p_{\beta', \beta}, \overline{f}_{\beta, \beta'}) \colon (X_\beta, A_\beta) \to (X_{\beta'}, A_{\beta'})$.\newline
Base Case: Define $D0 := A$. This gives us an object $(X_0,A_0)$ of $\mathrm{Prod}(\mathcal{K})$, where $X_0 := 1$ and $A_0(0) := A$.\newline
Successor step: If $\beta < \delta$ is an ordinal such that we've already performed the construction for all the ordinals less or equal to $\beta$, then for each $z \in Z_\beta$ and each $x \in X_\beta$ consider the following pushout
\begin{center}
\begin{tikzcd}
E_\beta(0) \arrow[d, "f_{\beta, z}"'] \arrow[r, "\ell_{\beta, x}\cdot g_\beta"] & A_{\beta}(x) \arrow[d, "\overline{f}_{\beta + 1, z, x}"] \\
C_{\beta}(z) \arrow[r, "h'_{\beta + 1, z, x}"']                                & A_{\beta + 1}(z, x)                     
\end{tikzcd}
\end{center}
where $\ell_{\beta, x}$ is the transfinite composite
$$A_0(p_{\beta,0}(x)) \xrightarrow{\overline{f}_{1,p_{\beta,1}(x)}} A_1(p_{\beta, 1}(x)) \xrightarrow{\overline{f}_{2,p_{\beta,2}(x)}} A_2(p_{\beta,2}(x)) \xrightarrow{\overline{f}_{3,p_{\beta,3}(x)}} \cdots \xrightarrow{} A_\beta(x).$$
In this way we obtain an object $(X_{\beta + 1}, A_{\beta + 1})$ of $\mathrm{Prod}(\mathcal{K})$ that's defined as follows: ${X_{\beta + 1} := Z_\beta \times X_\beta}$ and $A_{\beta + 1}(z, x)$ is the pushout from the diagram above. Moreover, the morphism $(p_{\beta + 1}, \overline{f}_{\beta + 1}) \colon (X_\beta, A_\beta) \to (X_{\beta + 1}, A_{\beta +1})$ is defined as follows: $p_{\beta + 1} \colon Z_\beta \times X_\beta \to X_\beta$ is the product projection and the component $\overline{f}_{\beta + 1, z, x}$ of $\overline{f}_{\beta + 1}$ at $(z, x)$ is the vertical morphism from the right side of the pushout square above.\newline
Limit Step: If $\beta \leq \delta$ is a limit ordinal such that we've already performed the construction for all the ordinals less than $\beta$, then we define the object ${(X_{\beta}, A_\beta) := \operatorname{colim}_{\gamma < \beta} (X_\gamma, A_\gamma)}$, where the transifinite composite is taken in $\mathrm{Prod}(\mathcal{K})$.\par
The first transfinite construction is finished. In this way we obtain a cone $(\tilde{f}_{0,x} \colon D0 \xrightarrow{} A_{\delta}(x))_{x \in X_{\delta}}$, and we define $I_1 := X_{\delta}$ and $D$ on $I_1$ to be $A_{\delta}$. Furthermore,  let $(Y_0, B_0) := (X_0, A_0)$ and $(Y_1, B_1) := (X_{\delta}, A_{\delta})$. There is an obvious morphism $(Y_0, B_0) \xrightarrow{} (Y_1, B_1)$. We repeat this transfinite procedure for each $x \in I_1$ by using $Dx$ instead of $D0$ and in this way we obtain a cone $(g_{x,y} \colon Dx \xrightarrow{} Dy)_{y \in I_{2,x}}$. Putting together all the codomains from these cones, where $x \in I_1$, we obtain the definition of $D$ on $I_2 := \bigcup_{x \in I_1} I_{2,x}$. This gives us an object $(Y_2, B_2)$ of $\mathrm{Prod}(\mathcal{K})$ in an obvious way and also an obvious morphism $(Y_1, B_1) \xrightarrow{} (Y_2, B_2)$. In this way we continue in each successor step. In limit steps $\beta \leq \lambda$ we take $(Y_\beta, B_\beta) := \mathrm{colim}_{\gamma < \beta} (Y_\gamma, B_\gamma)$. This transfinite construction finally gives us the functor $D$. Now define the cone $(f_i \colon A \to A_i)_{i \in I'}$ to be the colimit injection $(Y_0, B_0) \to (Y_\lambda, B_\lambda)$.\par
Now that we've finished the construction let us show that each $A_i$, where $i \in I'$, belongs to $I$-Inj. Suppose that we have a cone $(h_j \colon A \xrightarrow{} L_j)_{j \in J'}$ from $I$ and a morphism $g \colon A \to A_i$. By construction, there exists a transfinite composition from $A$ to $A_i$, in other words for each $\alpha < \lambda$ there exists $x_\alpha \in I_\alpha$ such that $A_i = \operatorname{colim}_{\alpha < \lambda} Dx_\alpha$. Using the fact that $A$ is $\lambda$-presentable we get that there exists a factorization of $g$ through some $Dx_\alpha$ via some morphism $f \colon A \to Dx_\alpha$. The diagram
\begin{center}
\begin{tikzcd}
{(1,A)} \arrow[r, "f"] \arrow[d, "h"] & {(1,Dx_\alpha)} \\
{(J',L)}                              &                
\end{tikzcd}
\end{center}
is one of the diagrams from the construction. Therefore there exists an index $j \in J'$ and a morphism $u \colon L_j \to Dx_{\alpha + 1}$. Composing this morphism with the colimit injection $\iota_{\alpha + 1} \colon Dx_{\alpha + 1} \to A_i$ we obtain a morphism $\iota_{\alpha + 1} \cdot u \colon L_j \to A_i$ and this morphism satisfies $$\iota_{\alpha + 1} \cdot u \cdot h_j = \iota_{\alpha + 1} \cdot (Dx_{\alpha} \to Dx_{\alpha + 1}) \cdot f = \iota_\alpha \cdot f = g.$$
Now suppose that $g \colon A \to B$ is a morphism whose codomain $B$ belongs to {$I$-Inj}. We want to show that there exists $i \in I'$ and a morphism $h \colon A_i \to B$ such that $h \cdot f_i = g$. Using {$B \in I$-Inj} we get that there exists $z \in Z_0$ and $g' \colon C_{0}(z) \to B$ such that $g \cdot g_0 = g' \cdot f_{0, z}$. Using the universal property of the pushout $A_{1}(z,0)$ we get a morphism $h_1 \colon A_{1}(z,0) \to B$ such that $h_1 \cdot \overline{f}_{1,z,0} = g$ and $h_1 \cdot h'_{1,z,0} = g'$. The previous two sentences are summed up in the following diagram.
\[\begin{tikzcd}
	{E_0(0)} && A \\
	{C_0(z)} & {A_1(z,0)} \\
	&& B
	\arrow["{f_{0,z}}", from=1-1, to=2-1]
	\arrow["g", from=1-3, to=3-3]
	\arrow["{g_0}", from=1-1, to=1-3]
	\arrow["{\overline{f}_{1,z,0}}", from=1-3, to=2-2]
	\arrow["{h'_{1,z,0}}"', from=2-1, to=2-2]
	\arrow["{h_1}", dashed, from=2-2, to=3-3]
	\arrow["{g'}"', curve={height=12pt}, dashed, from=2-1, to=3-3]
\end{tikzcd}\]
Repeating the previous procedure we get $z' \in Z_1$ and $h_2 \colon A_{2}(z', (z, 0)) \to B$ such that the obvious diagrams commute. We continue in this way in each successor step. In the limit steps we use the universal property of colimits. After this transfinite construction is finished we obtain an element $x \in I_1$ and a morphism $h_x \colon Dx \to B$ such that the obvious diagrams commute. Repeating this procedure we obtain $y \in I_2$ and a morphism $h_y \colon Dy \to B$ such that the obvious diagrams commute. In each successor step we repeat this. In limit steps we use the universal property of the colimit. In the end we obtain $i \in I'$ and a morphism $h \colon A_i \to B$ such that $h \cdot f_i = g$.
\end{proof}

\section{(Regular epi, mono) factorization in coproduct completion}

We give a brief account on the (regular epi, mono) factorization with applications towards semi-simplicial objects; most importantly, this will be applied to the coproduct completion of a category in order to study colimits of covers. Therefore, it will also be important that the (regular epi, mono) factorization exists in the coproduct completion.

\subsection{Factorization and colimits of semi-simplicial objects}

Assume that $\mcA$ is a complete and cocomplete category that admits a factorization system with the left class the regular epis and the right class the monos. Since
\begin{itemize}
\item
    every regular epi is the coequalizer of its kernel pair and
\item
    in the (regular epi, mono) factorization $f=me$, the kernel pair of $f$ is the same as that of $e$,
\end{itemize}
the factorization of $f$ is the so called coimage factorization, i.e.\ the middle object is the coequalizer of the kernel pair of $f$:
\[\xymatrix{
A \times_B A \ar@<.5ex>[r]^-{d_0} \ar@<-.5ex>[r]_-{d_1} & A \ar[r] \ar@/^2ex/[rr]^-f & \operatorname{coeq}(d_0, d_1) \ar[r] & B
}\]
Quite easily, the coimage is also the coequalizer of any pair $K \rightarrows A$ for which the induced map $K \fib A \times_B A$ is epic, since this condition implies that the coequalizer is the same as that of the kernel pair, thus proving the following theorem:

\begin{theorem}
Assume that a (regular epi, mono) factorization exists in $\mcA$. Let
\[\xymatrix{
\cdots \ar@<1ex>[r] \ar@<0ex>[r] \ar@<-1ex>[r] & A_1 \ar@<.5ex>[r] \ar@<-.5ex>[r] & A_0 \ar[r]^-f & B
}\]
be an augmented semi-simplicial object with matching maps $A_n \to M_nA$ epic for $n > 0$. Then the induced factorization
\[f \colon A_0 \fib \colim A_\bullet \cof B\]
consists of a regular epi followed by a mono. \qed
\end{theorem}

We will now apply this to the coproduct completion $\coprod\mcA$, assuming that it admits a (regular epi, mono) factorization -- its existence will be addressed later in the section. Considering for simplicity the case of a singleton $B$, a direct translation of the above theorem is simple enough: We require an augmented semi-simplicial object with matching maps epic, which in $\coprod\mcA$ means that all summands are non-empty jointly epic families; it is the non-emptiness assumption that we want to remove. On a related note, $\colim A$ may have more than one component and the object of interest is their coproduct in $\mcA$ -- by Lemma~\ref{lem:limits_colimits_coproduct_completion}, it is $H_0 A_\bullet = \colim \El A_\bullet$, where $\El A_\bullet$ is the diagram composed of all the components of all the $A_n$ and the maps between them the components of all the maps in the semi-simplicial object. In some sense, we are really thinking $\El A_\bullet$ but pack the information into $A_\bullet$ for simplicity.

Formally, introduce the summation functor $\Sigma \colon \coprod\mcA \to \mcA$, given on $A = (A^i)^{i \in I}$ as $\Sigma A = \sum_{i \in I} A^i$, clearly left adjoint to the inclusion functor, hence cocontinuous. Applying $\Sigma$, we replace the augmented semi-simplicial object in $\coprod\mcA$, together with its colimit, by a similar picture in $\mcA$:
\[\xymatrix{
\cdots \ar@<1ex>[r] \ar@<0ex>[r] \ar@<-1ex>[r] & \Sigma A_1 \ar@<.5ex>[r] \ar@<-.5ex>[r] & \Sigma A_0 \ar[r] & \underbrace{\Sigma \colim A_\bullet}_{{}\cong H_0 A_\bullet} \ar[r] & B
}\]
We still desire that this gives the (regular epi, mono) factorization of $\Sigma A_0 \to B$, and the first map, being a coequalizer, is indeed a regular epi. We will now repair the deficiency of the matching maps being possibly empty, by adding to each $A_n$ a number of dummy components consisting of the initial object $\emptyset$, in a consistent way. This does not alter the above diagram, while making the general theorem directly applicable.

For $B \in \coprod\mcA$, denote by $\widehat B$ the family with the same indexing set as $B$ but with all components consisting of the initial object, and let $\widehat B \to B$ be the unique map that is the identity on the underlying index sets. For $f \colon A \to B$, define $A' = A \amalg \widehat B$ and $f' \colon A' \to B$ to be $f$ on $A$ and the above canonical map on $\widehat B$. Now starting with an augmented semi-simplicial object $A$ in $\coprod\mcA$, inductively w.r.t.\ $n$, consider the $n$-th matching map $m_n \colon A_n \to M_nA$ and apply the above construction $m_n' \colon A_n' \to M_nA$. Replacing $A_n$ by $A_n'$, but keeping the rest, the map $m_n'$ says how to make this into an augmented semi-simplicial object. The final result is an augmented semi-simplicial object with all matching maps surjective on the indexing sets (i.e.\ all summands non-empty families) and with the same $\Sigma$-image. We thus obtain:

\begin{theorem} \label{thm:epi_mono_factorization}
Assume that a (regular epi, mono) factorization exists in $\coprod\mcA$. Let
\[\xymatrix{
\cdots \ar@<1ex>[r] \ar@<0ex>[r] \ar@<-1ex>[r] & A_1 \ar@<.5ex>[r] \ar@<-.5ex>[r] & A_0 \ar[r]^-f & B
}\]
be an augmented semi-simplicial object in $\coprod\mcA$ with a singleton $B$ and with matching maps $A_n \to M_nA$ jointly epic for $n > 0$. Then the induced factorization
\[f \colon \Sigma A_0 \fib \underbrace{\Sigma \colim A_\bullet}_{{}\cong H_0 A_\bullet} \cof B\]
consists of a regular epi followed by a mono. \qed
\end{theorem}


\subsection{Factorization in coproduct completion}

\begin{theorem} \label{theorem:factorization_from_universal_colimits}
If colimits are universal in $\mcA$ then the (regular epi, mono) factorization exists in $\coprod \mcA$.
\end{theorem}

The proof will consist of a series of implications.

\begin{lemma}
If pushouts are universal then epis are closed under pullbacks.
\end{lemma}

\begin{proof}
This comes from the characterization of epis as maps $f$ for which
\[\xymatrix{
A \ar[r]^-f \ar[d]_-f & B \ar[d]^-1 \\
B \ar[r]_-1 & B
}\]
is a pushout square. Pulling back such a square along $B' \to B$ yields another such square and thus the pullback $f' \colon A' \to B'$ is epi too.
\end{proof}

Here is the main technical statement of the theory of Reedy categories applied to pullbacks:

\begin{lemma}
Fix a class $\mcE$ of maps closed under pullbacks and composition. In the diagram
\[\xymatrix{
A' \ar[r] \ar[d]^-a & C' \ar[d]^-c & B' \ar[l] \ar[d]^-b & A' \times_{C'} B' \ar[d]^-{a \times_c b} \\
A \ar[r] & C & B \ar[l] & A \times_C B
}\]
if $c \colon C' \to C$ and both pullback corner maps $A' \to A \times_C C'$ and $B' \to B \times_C C'$ belong to $\mcE$ then so does the induced map $a \times_c b$.

More generally, the same conclusion holds if $a$ and the right pullback corner map belong to $\mcE$.
\end{lemma}

By applying this to the kernel pair, obtainable as a pullback with its two sides identical, we easily obtain:

\begin{proposition}
Assume that epis are closed under pullbacks and that in the square
\[\xymatrix{
A' \ar[r]^-{f'} \ar[d]_-a & B' \ar[d]^-b \\
A \ar[r]_-f & B
}\]
both $b$ and the pullback corner map are epis. Then so is the induced map $\ker f' \to \ker f$ on the kernel pairs. \qed
\end{proposition}

Now we apply this to factorization through the coimage:

\begin{theorem}
Assume that epis are closed under pullbacks and $f \colon A \to B$. Then the factorization through the coimage
\[f \colon A \fib C \cof B\]
consists of a regular epi followed by a mono.
\end{theorem}

\begin{proof}
The first map is a regular epi by definition. Form the square
\[\xymatrix{
A \ar[r]^-f \ar[d] & B \ar[d]^-1 \\
C \ar[r] & B
}\]
Since $1$ is epi and so is the pullback corner map (being the map $A \to C$), so is the induced map on kernel pairs:
\[\xymatrix{
A \times_B A \ar@<.5ex>[r] \ar@<-.5ex>[r] \ar@{->>}[d] & A \ar[r]^-f \ar@{->>}[d] & B \ar[d]^-1 \\
C \times_B C \ar@<.5ex>[r] \ar@<-.5ex>[r] & C \ar[r] & B
}\]
Now the two composites $A \times_B A \to C$ are equal by the definition of $C$ as the coequalizer; since the left vertical map is epic, the two structure maps $C \times_B C \to C$ must be equal and consequently $C \to B$ is monic.
\end{proof}

Lastly, in order to apply this theorem to the coproduct completion $\coprod \mcA$, we need to show that epis are pullback stable in $\coprod \mcA$. Since a map $f$ is epic iff each factor of $f$ is non-empty with components jointly epic, this translates to the requirement that (non-empty) jointly epic families should be pullback stable in $\mcA$. A jointly epic family $(A^i) \to B$ is equivalently a single epi $\coprod A^i \to B$. This implies easily the following claim:

\begin{proposition}
Assume that (non-empty) coproducts are universal in $\mcA$. If epis are closed under pullbacks in $\mcA$ then the same holds in $\coprod \mcA$. \qed
\end{proposition}

\section{Reduction and distinguished open sets} \label{sec:reduced_localizations}

While an $\red\mcA$-localization happens to be the same as an $\mcA$-localization, as can be shown quite easily, the analogue for finite localizations is not true. Since these define distinguished open sets, it is not immediately clear whether the notions of distinguished opens w.r.t.\ $\red\mcA$ and $\mcA$ coincide in $\Spec \red R = \Spec R$. It is the purpose of this section to show that they do.

Clearly, the reduction of a finite $\mcA$-localization $R \cof K$ is a finite $\red\mcA$-localization $\red R \cof \red K$ and they give the same points, inducing a diagram
\[\xymatrix{
\FinLoc_{\mcA}(R) \ar[d]_-{\red} \ar[r]^-{\Pts} & \DistOp(\Spec R)^\op \ar@{c->}[d] \\
\FinLoc_{\red \mcA}(\red R) \ar[r]_-{\Pts}^-\cong & \DistOp(\Spec \red R)^\op
}\]
with the bottom map bijective by Corollary~\ref{cor:dist_open_finite_localization_correspondence}.

\begin{theorem}
The map $\DistOp(\Spec R) \to \DistOp(\Spec \red R)$ is bijective.
\end{theorem}

\begin{proof}
The map on the left is surjective by Proposition~\ref{proposition:reduction_surjective_on_finite_localizations}, hence so is the map on the right.
\end{proof}

\begin{lemma} \label{lemma:pushouts_of_finite_localizations}
Let $f \colon R \cof S$ be a localization. Then the pushout along $f$ gives a surjective map
\[f_* \colon \FinLoc(R) \to \FinLoc(S).\]
In other words, the induced map on spectra, $f^* \colon \Spec S \to \Spec R$ when (co)restricted to its image takes distinguished opens to distinguished opens.
\end{lemma}

\begin{proof}
Clearly finite complexes are closed under pushouts, so the map is well defined. The main idea of the proof is quite simple: Thinking of $S$ as a localization of $R$, express it as a filtered colimit of its finite sublocalizations $S_\alpha$. Any finite localization of $S$ is obtained by attaching a cell along a map $A_i \to S$ and this lands in some $S_\alpha$. When this cell is attached to $S_\alpha$ instead of $S$, one obtains the required finite localization of $R$. Performing this inductively for all cells of $R$ produces the required finite localization of $R$.

In more detail, let $S \cof L$ be a finite localization and by induction we may assume that $S \cof L'$ lies in the image and $L' \cof L$ is a pushout of a single generator $a_{ij}$ as in the rectangle on the right of
\[\xymatrix{
R \ar@{ >->}[r] \ar@{ >->}[d] & K' \ar@{ >->}[d] & A_i \ar@{ >->}[r]^-{a_{ij}} \ar[d] & B_{ij} \ar[d] \\
S \ar@{ >->}[r] & L' \ar@{=}[r] \po & L' \ar@{ >->}[r] & L \po
}\]
Now since $K' \cof L'$ is a filtered colimit of its finite sublocalizations, the map $A_i \to L'$ factors through some finite subcomplex $K''$; denoting by $K$ the pushout as in
\[\xymatrix{
R \ar@{ >->}[r] \ar@{=}[d] & K' \ar@{ >->}[d] & A_i \ar@{ >->}[r]^-{a_{ij}} \ar[d] & B_{ij} \ar[d] \\
R \ar@{ >->}[r] \ar@{ >->}[d] & K'' \ar@{=}[r] \ar@{ >->}[d] & K'' \ar@{ >->}[r] \ar@{ >->}[d] & K \ar[d] \po \\
S \ar@{ >->}[r] & L' \ar@{=}[r] \po & L' \ar@{ >->}[r] & L \po
}\]
we observe that the bottom left square is also pushout (since $K' \cof K''$ is epi) and conclude that $L \cong f_* K$.
\end{proof}

\begin{proposition} \label{proposition:reduction_surjective_on_finite_localizations}
Let $R \in \mcA$. Then $\red \colon \FinLoc_\mcA (R) \to \FinLoc_{\red \mcA} (\red R)$ is surjective.
\end{proposition}

\begin{proof}
Since the factorization system $\red \mcA$ is generated by $\mathrm r a_{ij}$, every step of the localization in $\red \mcA$ as on the left can be viewed as the reduction of the localization in $\mcA$ as on the right
\[\xymatrix{
\red A_i \ar[r]^-{\red a_{ij}} \ar[d] & \red B_{ij} \ar[d] & & A_i \ar[r]^-{a_{ij}} \ar[d] & B_{ij} \ar[d] \\
X \ar[r] & \leftbox{Y}{{} = \red Y'} \po & & X \ar[r] & Y' \ar[r] \po & \red Y'
}\]
i.e.\ as a step of a localization in $\mcA$ followed by a reduction. Lemma~\ref{lemma:pushouts_of_finite_localizations} then shows that we may reorganize such a composition in the top row
\[\xymatrix{
R \ar@{ >->}[r] \ar@{ >-->}[rrd] & \red R \ar@{ >->}[r] & L_1 \ar@{ >->}[r] & \red L_1 \ar@{ >->}[r] & L_2 \ar@{ >->}[r] & \red L_2 \ar@{ >->}[r] & \cdots \\
& & K_1 \ar@{ >->}[r] \ar@{ >-->}[rrd] & \red K_1 \ar@{ >->}[r] \ar@{=}[u] & L_2 \ar@{ >->}[r] \ar@{=}[u] & \red L_2 \ar@{ >->}[r] \ar@{=}[u] & \cdots \\
& & & & K_2 \ar@{ >->}[r] & \red K_2 \ar@{ >->}[r] \ar@{=}[u] & \cdots
}\]
by replacing the first three maps (the composition across the bottom left in the diagram below)
\[\xymatrix{
R \ar@{ >->}[d] \ar@{ >-->}[r] & K \ar@{ >-->}[d] \ar@{ >-->}[r] & \red K \ar@{-->}[d]^-\cong \\
\red R \ar@{ >->}[r] & L \ar@{ >->}[r] \po & \red L
}\]
by the composition across the top (the map on the right is a pushout in $\red \mcA$ of the reduction of $R \cof \red R$ on the left and as such is an iso); repeating this process, any finite localization in $\red \mcA$ is obtained as a finite localization in $\mcA$ followed by a reduction.
\end{proof}

\section*{Acknowledgements}
We would like to thank John Bourke and Ivan di Liberti for very useful discussions and for their enlightening explanations.

\vskip 20pt
\vfill
\vbox{\footnotesize%
\noindent\begin{minipage}[t]{0.6\textwidth}
{\scshape Jan Jurka, Tomáš Perutka, Lukáš Vokřínek}\\
Department of Mathematics and Statistics,\\
Masaryk University,\\
Kotl\'a\v{r}sk\'a~2, 611~37~Brno,\\
Czech Republic
\end{minipage}
\hfill
\begin{minipage}[t]{0.3\textwidth}
jurka@math.muni.cz\\
xperutkat@math.muni.cz\\
vokrinek@math.muni.cz
\end{minipage}
}

\end{document}